\documentclass[dvipsnames]{siamonline190516}
\pdfoutput=1
\usepackage{enumitem}
\usepackage{arxiv}
\usepackage{tikz}
\usepackage{soul}
\usepackage{tcolorbox}
\usepackage{csquotes}

\usepackage{multirow}
\usepackage{amsmath}
\usepackage[utf8]{inputenc} 
\usepackage[T1]{fontenc}    
\usepackage{hyperref}       
\usepackage{url}            
\usepackage{booktabs}       
\usepackage{amsfonts}       
\usepackage{nicefrac}       
\usepackage{microtype}      
\usepackage{lipsum}
\usepackage{bbm}
\usepackage{graphicx}
\usepackage{algorithm}
\usepackage{algpseudocode}
\usepackage{amsmath}
\usepackage{amssymb}
\usepackage{float}
\usepackage{xcolor}
\usepackage{authblk}
\usepackage[export]{adjustbox} 
\usepackage{tabularx}
\usepackage{hyperref}
\hypersetup{
    colorlinks=true,
    linkcolor=blue,
    filecolor=magenta,      
    urlcolor=cyan,
    pdftitle={Overleaf Example},
    pdfpagemode=FullScreen,
    }

\graphicspath{ {./images/} }

\def\Var{{\textrm{Var}}}

\usepackage[edges]{forest}
\usepackage{xcolor}
\usetikzlibrary{arrows.meta}
\colorlet{linecol}{black!75}
\usepackage{xkcdcolors} 

\usetikzlibrary{calc, arrows.meta, intersections, patterns, positioning, shapes.misc, fadings,shadows,trees,mindmap, through,decorations.pathreplacing}

\usepackage{array, makecell}
\setcellgapes{3pt}

\usetikzlibrary{backgrounds}
\usetikzlibrary{arrows,shapes}
\usetikzlibrary{tikzmark}
\usetikzlibrary{calc}

\definecolor{ColorOne}{named}{MidnightBlue}
\definecolor{ColorTwo}{named}{Dandelion}
\definecolor{ColorThree}{named}{Plum}

\title{Numerical Methods for Computing the Discrete and Continuous Laplace Transforms}

\author{Yupeng Zhang $^{1\dag}$,
        Yueyang Shen  $^{1\dag}$, 
        Rongqian Zhang $^{1}$,
        Yuyao Liu $^{1}$,
        Yunjie Guo $^{1}$,
        Daxuan Deng $^{1}$, \and
        Ivo D. Dinov $^{1,2,3,*}$
}

\affil{$^1$ Statistics Online Computational Resource, University of Michigan, Ann Arbor, MI 48109\\
$^2$ Computational Medicine and Bioinformatics, University of Michigan, Ann Arbor, MI 48109\\
$^3$ Precision Health and Michigan Institute for Data Science, University of Michigan, Ann Arbor, MI 48109\\
$^{\dag}$ These authors have contributed equally to this work.\\
$^*$ Contact: Ivo Dinov, SOCR, University of Michigan,
\href{mailto:statistics@umich.edu}{statistics@umich.edu},
\href{https://www.SOCR.umich.edu}{\nolinkurl{https://www.SOCR.umich.edu}}, ORCID: 0000-0003-3825-4375.}

\numberwithin{equation}{subsection}

\begin{document} 



\maketitle 

\begin{figure}[htb]
    \centering    \includegraphics[width=14cm,height=5cm]{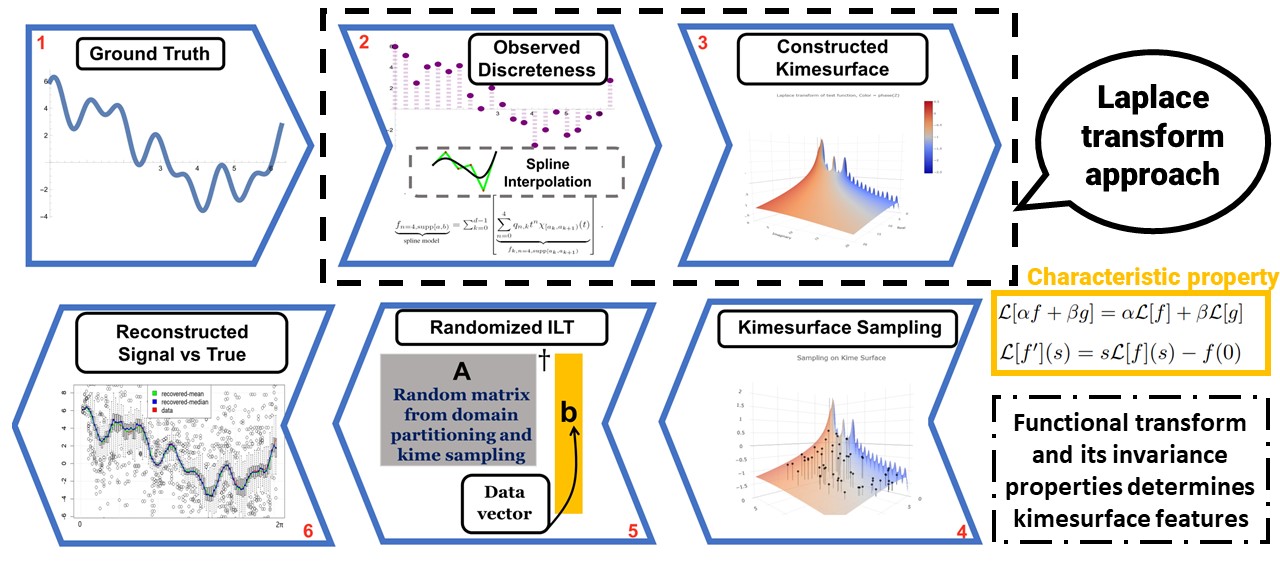}
    \caption{Graphical abstract}
    \label{graphical_abstract}
\end{figure}

\begin{abstract}
\textbf{Overview}: Many mathematical models and statistical computing techniques can be significantly simplified by mapping the original problem from its native state space representation into another space via specialized integral transformations. Examples of these include the Laplace transform (LT) and its inverse (ILT), which play vital roles in contemporary data science, probability modeling, and statistical inference. However, there are challenges with computing LT and ILT for some complex functions, noisy data, or incomplete observations. This article reports on a new numerical algorithm for computing the forward and inverse Laplace  transforms. The technique is applicable to both continuous functions and discrete  signals, and yields computationally  efficient  and  analytically  robust results. \vspace{1mm}

\textbf{Methods}:  We propose a numerical method to spline-interpolate discrete signals and then apply the integral transforms to the corresponding analytical spline functions. This represents a robust and computationally efficient technique for estimating the Laplace transform for noisy data. We revisited a Meijer-G symbolic approach to compute the Laplace transform and alternative approaches to extend canonical observed time-series. A discrete quantization scheme provides the foundation for rapid and reliable estimation of the inverse Laplace transform. We derive theoretic estimates for the inverse Laplace transform of analytic functions and demonstrate empirical results validating the algorithmic performance using observed and simulated data. We also introduce a generalization of the Laplace transform in higher dimensional space-time.
\vspace{1mm}

\textbf{Results}: We tested the discrete LT algorithm on data sampled from analytic functions with known exact Laplace transforms. The validation of the discrete ILT involves using complex functions with known analytic ILTs. For discrete signals, we also report empirical evidence of the asymptotic behavior of the expectation of the smallest singular value of the LT matrix, $\mathbb{E}(\sigma_n(\mathbf{A})) \sim \frac{1}{n^\gamma}$ for some $\gamma\in(0,2)$. We present an analysis on the smallest singular values of the random matrix that emerges in the ILT algorithm. We show that under idealized assumption a similar data matrix can be bounded. We also show that if phase sampling is independent with the domain and radial strategy then isotropicities for random matrix result will likely be violated.\vspace{1mm}

\textbf{Conclusions}: Our LT and ILT numerical algorithms perform well for estimating the Laplace transform of various discrete and continuous functions. We have designed, implemented, validated, and shared an R-implementation of several LT and ILT estimation methods. The techniques can be applied to analytical functions as well as observational signals. We also suggest a Clifford algebra approach to generalize the Laplace transform to higher dimensional space-time processes to distinguish space and time.

\end{abstract}

\section{Introduction}
\subsection{Motivation}
Solutions to many challenging scientific problems rely on mechanistic understanding of specific natural processes. Such understanding often depends on difficult calculations, accurate model approximations, robust parameter estimations, efficient likelihood quantification, or derivations of reliable sample statistics. In practice, this modeling process can be significantly simplified by mapping the original (complex) problem from its native state space representation into another (simpler) space via specialized transformations. For instance, {\it{integral transforms}} map functions, equations, and conditions from their original domains into other domains where the subsequent manipulation, modeling, and estimation processes may be substantially simpler, compared to their counterparts in the native domain. Indeed, such transformations are most useful when the solution derived in the new state space can be effectively mapped back (by a corresponding inverse transform) into the original space and provide pragmatic solutions to the original scientific problem. 

The Fourier, wavelet and Laplace transforms \cite{saitoh1997integral} represent examples of a special kind of maps called {\it{integral transforms}} between a domain $\mathcal{V}$ and a range $\mathcal{U}$, $T:\mathcal{V}\to \mathcal{U}$, that can be represented by:
\begin{equation}
  T[f](u)=F(u)=\int_{\mathcal{V}} {\ K(v,u)\ d\mu(v)}=\int_{\mathcal{U}} {f(v)\ K(v,u)\ dv}  
\end{equation}
where the function $f$ represents the input of the transform $T$, the output function is denoted by $F(\cdot)\equiv \hat{f}(\cdot)=T[f](\cdot)$, the kernel function $K$ specifies the concrete mathematical operation associated with each transformation, and $\mu$ is a Borel measure on the domain. The marginal probability function is a specific example of integral transform, $f(u)=\int \frac{f_{VU}(v,u)}{f_V(v)}f_V(v)dv=\int f_{U\mid V}(u|v)f_V(v)dv$, where the conditional distribution $f_{U\mid V}$ can be interpreted as a Markov transition matrix, in finite state spaces, or as a Markov kernel, in the more general sense.

A formal mathematical prescription to facilitate kernel analysis relies on the idea of \textit{reproducing kernel Hilbert space (RKHS)}, $\mathcal{H}_{K}(\mathcal{X})$, associated with the kernel $K(\cdot,\cdot)$. In the RKHS,  $\forall f\in\mathcal{H}\equiv\mathcal{H}_{K}(\mathcal{X})$, $f(x)=\langle K(x,\cdot)f(\cdot)\rangle_{\mathcal{H}_{K}(\mathcal{X})}$ \cite{saitoh2016theory}. In certain data science applications, using ordinary Hilbert spaces is problematic, due to lack of sufficiently strong {\it{a priori}} conditions to support effective predictive analytics. For instance, the Hilbert space representing all square integrable functions, $L^2([-1,1])$, can admit functions with a  finite number of point infinities, which still remain square integrable. However, such singularities (finite number of infinities) add analytical difficulties related to numerical computing and  scientific inference. The main RKHS restriction ensures the boundedness of the evaluation functionals, i.e., $\forall t\text{, } \exists M>0 \text{, } |f(t)|\leq M\|f\|_{\mathcal{H}_{K}(\mathcal{X})},\forall f\in \mathcal{H}_{K}(\mathcal{X})$. The reproducing kernel property then arises from the Riesz representation theorem guaranteeing the existence of a unique kernel $K(x,\cdot)\in\mathcal{H}_{K}(\mathcal{X})$ \cite{goodrich1970riesz}. This exact one-to-one correspondence between kernels and RKHS is known as the Moore-Aronszajn theorem \cite{aronszajn1950theory,gretton2013introduction}.

Kernels play important role in statistical learning theory. An example is the recent seminal discovery on the duality of kernel regression with an infinitely wide neural network\cite{jacot2018neural}, which facilitates natural comparative studies examining the representation power of kernel methods and neural networks. For kernels that are positive semidefinite and symmetric, $K(\cdot,\cdot):\mathcal{X}\times\mathcal{X}\to \mathbb{R}$, yield corresponding integral-transform operators that are compact, continuous, self-adjoint, and non-negative. Let's denote the eigenfunctions and eigenvalues of the integral-transform operator by $\phi_i$ and $\sigma_i$ ($\sigma_1\geq\sigma_2\geq....$), respectively, i.e., $T\phi_i=\sigma_i\phi_i$. Then, the Mercer's theorem \cite{ghojogh2021reproducing,berlinet2011reproducing} states that 

\begin{equation}
    K(x,y) = \sum_{j=1}^{\infty}\sigma_j \phi_j(x)\phi_j(y).
\end{equation}
Mercer's theorem can be regarded as infinite dimensional extension of eigen-decomposition for positive symmetric semidefinite (PSD) matrices in finite dimensional spaces, i.e., 
$K=\sum_{i=1}^n \sigma_i v_iv_i^T$. In kernel-based support vector machine (SVM) clustering, the mapping $\psi=[\psi_1,\psi_2,\cdots ]$ is responsible for transforming the original datapoint cloud $\{x_i\}_{i=1}^M$ to a linearly separable dataset $\{\psi(x_i)\}_{i=1}^M$. The functions $\psi$ are directly related to the eigenstates $\sqrt{\sigma_i}\phi_i=\psi_i$ and suggest the compact {\it{kernel trick formula}} $K(x,y) = \langle \psi_i(x),\psi_i(y) \rangle = \psi_i(x)\cdot\psi_i(y)$.

Many interesting integral transformations correspond to invertible kernels $K(t,u)\longleftrightarrow K^{-1}(u,t)$ that naturally lead to corresponding invertible integral transforms that facilitate function or signal recovery
\begin{equation}
  f(x)=\int_{\mathcal{U}} {F(u)\ K^{-1}(u,x)\ du\ .}   
\end{equation}

For example, the classical Laplace transform (LT) is an integral transformation with a non-domain-symmetric kernel $K(t,z)=e^{-tz}$ mapping functions of a real variable $t\in\mathbb{R}^+$ (typically {\it{time}}) into functions of a complex frequency variable $z\in \mathbb{C}_{+}$ on the right half plane with the invertibility via the Bromwhich integral
\begin{equation}
\begin{aligned}
\mathcal{L}:L^2([0,\infty))\to H^2(\mathbb{C}_{+}), \quad \underbrace{\mathcal{L}(f)}_{F}(z)&=\int_{0}^{\infty} {f(t)e^{-zt}dt} \\
 \underbrace{ \mathcal{L}^{-1}\{F\}}_{f}(t)&=\frac{1}{2 \pi i} \int_{-i \infty}^{i \infty} F(z) \mathrm{e}^{z t} \mathrm{~d} z. 
\end{aligned}
 \end{equation}  
The classical Laplace transform is a bounded linear operator in the sense that it maps square-integrable functions on the positive reals onto the Hardy space $H^2(\mathbb{C}_{+})$ of the right half plane, which is defined as $H^2(\mathbb{C}_{+})=\{f:\|f\|_{H^2}=\text{sup}_{x>0} \sqrt{\int_{-\infty}^{\infty}|f(x+iy)|^2dy}< \infty\}$, with 
$\|\mathcal{L}f\|_{H^2(\mathbb{C}_{+})}= \|f\|_{L^2([0,\infty))}$ (see Appendix). Furthermore,  the unitary isomorphism between $L^2([0,\infty))$ and $H^2(\mathbb{C}_{+})$
is realized by the Laplace transform, and the invertibility is given by a speical case of the Bromwhich integral, as a Wick rotated inverse Fourier transform. Although this invertibility is useful, a basic set of polynomial functions is excluded from the space of square integrable functions. Therefore, a proper decay factor $\gamma$ (in the Bromwhich integral) is needed to temper the polynomial growth by exponential decay, i.e., ensure that that $e^{-\gamma t}f(t)\in L^2([0,\infty))$, and this sufficient decay permits an extended definition for a wider class of functions of a complex (frequency) variable $z\in \mathbb{C}\setminus \{Re(z)=-\infty\}$ (complex-time, aka {\it{kime}} \cite{dinov2021data2}).

We can interpret a Laplace transformed signal $f(t)\to {F}(z)$ as a more general function of complex-time frequency, $z=re^{i\theta}\equiv \underbrace{r\cos(\theta)}_{Re(z)} + i \underbrace{r\sin(\theta)}_{Im(z)}$, where the real component of the complex-time frequency, $Re(z)$, corresponds to the degree of damping associated with an exponential decrease of the signal amplitude, the imaginary component of the complex-time frequency, $Im(z)$, corresponds to the rate of sinusoid cycles (classical frequency), the phase, $\theta$, represents the random (IID) sampling for repeated controlled experiments, and the magnitude, $r$, corresponds to the natural longitudinal event ordering \cite{dinovdata} of the original temporal process, $f$. 

Note that the {\it{Laplacian kernel}} is different from the {\it{kernel for the Laplace transform}},
\begin{equation}
    \underbrace{K(x,y) = e^{-c\|x-y\|}}_{\text{Laplace (RBF) kernel}}  \quad ,\quad   
    \underbrace{K(t,z)=e^{-tz}}_{LT\  kernel}.
\end{equation}

The kernel of the Laplace transform takes the product of the two kernel arguments, whereas the Laplacian  takes the difference. The domain kernel of the Laplace transform is not necessary symmetric.

A symmetric reproducing kernel $\mathbb{R}^{+}\times\mathbb{R}^{+}\to\mathbb{R}$: $K(s,t)=\int_{0}^{\min(s,t)}\xi e^{\xi}d\xi$ can generate a Hilbert space that has a direct real inversion formula for Laplace transform \cite{saitoh2016theory}:
\begin{subequations}
\begin{align}
(\mathcal{L}^{-1}F)(t)=\sum_{n=1}^{\infty}\frac{1}{\lambda_n}\Big(\int_{0}^{\infty}pu(p)F(p)u_n(p)dp\Big)v_n(t)\ , \label{inverse_lap} \\ Lv_n=\lambda_nu_n,L^*u_n=\lambda_nv_n , \label{basis_condi}
\end{align}
\end{subequations}

 where $Lf(p) = p(\mathcal{L}f)(p)$, and  $\{v_n\}_{n=1}^{\infty}\subset H_K(w)$ and $\{u_n\}_{n=1}^{\infty}\subset L^2(u)$ are orthonormal systems in the spaces $H_K(w)$ and $L(H_K(w))$, and $w$ and $u$ are the corresponding weight function in these Hilbert spaces. Here $H_K(w)\equiv \{f: f(0)=0,\|f\|_{H_K(w)<\infty}, f\text{ is real-valued absolutely continuous on }[0,\infty)\}$, and $L(H_K(w))$ is the $L$ operator transformed function space of $H_K(w)$. The operator $\mathcal{L}$'s invertibility is realized by the basis matching condition in \ref{basis_condi}. To see this, we can expand $pF(p)\in L(H_K(w))$ in the $u_n$ basis, i.e., $pF(p)=p\mathcal{L}\mathcal{L}^{-1}F(p)=L\mathcal{L}^{-1}F(p)=\sum_{n=1}^{\infty}\Big(\int_{0}^{\infty}qu(q)F(q)u_n(q)dq\Big)u_n(p)$, and applying \ref{basis_condi} gives the inversion formula \ref{inverse_lap}.

Many applications of the LT involve solving differential equations \cite{THONGMOON2010425}, (functional distribution equations transforming time-series in spacetime to kime-surfaces in spacekime \cite{dinov2021data2,dinovdata,wang2022determinism}. Such transformations are useful for enhancing statistical inference \cite{BELOMESTNY20162092}. The closed-form analytical representation of the LT facilitates conversion of differential equations into algebraic equations and the transformation of the convolution operation into multiplication \cite{bellman1984laplace,deakin1981development}. Specifically, these properties provide recipes for solving non-homogeneous ordinary differential equation, where the inhomogeneity may be discontinuous and does not have to be a Dirac delta function, $\delta(t)$.

In this manuscript, we utilize the Laplace transform to represent classical longitudinal processes, such as time-series \cite{chatfield2003analysis}, as complex-valued functions of complex-time (kime) \cite{zhang2022kimesurface}. The motivation for this mapping of time-varying processes as kime-surfaces is two-fold. First, random sampling from a well-defined process with some underlying probability distribution may be thought of as discrete foliation of the kime-surface manifold \cite{rosenberg1968foliations,palmeira1978open}, where each leaf corresponds to an observed instance of the time-series process \cite{dinov2021data2,dinovdata}. Second, mathematical modeling and statistical inference based on the holistic (higher dimensional) kime-manifold is expected to yield more robust inference compared to the classical time-series analysis using repeated IID sampling. In addition, statistical inference in the space-kime domain can be represented in a Bayesian framework, which naturally leads to corresponding time-series analyses restricted to each leaf of the kime-surface foliation.

We present a new computationally efficient and analytically robust representation of the forward and the inverse Laplace transforms for both continuous functions and discrete signals. The R-implementation of the algorithm is tested and validated with simulated and real data (e.g., functional magnetic resonance imaging, fMRI). The inverse Laplace transform is the most difficult part as many applications  require (perfect or nearly perfect) reconstruction of the original signal, i.e., $\mathcal{L}^{-1}(\mathcal{L})(f)\equiv f$ and $\mathcal{L}(\mathcal{L}^{-1})(F)\equiv F$, for all $f:\mathbb{R}^+\to \mathbb{C}$ and $F=\mathcal{L}(f):\mathbb{C}\to \mathbb{C}$. 

In this study, we are interested in the statistical and data science aspects of the Laplace transform and most of the work relates to functions defined over the positive reals, $f:\mathbb{R}^+\to \mathbb{C}$. However, there is a relation between the more general {\it{Bilateral Laplace Transform}}, $\mathcal{B}$, of functions supported on $\mathbb{R}$ and the Laplace transform of functions supported on $\mathbb{R}^+$.
Assuming the following two definite integrals exist, $\int_{0}^{\infty}{e^{-st}f(t)dt}$
and $\int_{-\infty}^{0}{e^{-st}f(t)dt}$, the Bilateral Laplace transform is defined as the improper integral
\begin{equation}
  \mathcal{B}(f)(s)=\int_{-\infty }^{\infty } {e^{-st}f(t)dt}.  
\end{equation}

The direct relation $\mathcal{L} \longleftrightarrow \mathcal{B}$ can be expressed in terms of the {\it{Heaviside}} step function, 
\begin{equation}
  H(x) =\mathbf{1}_{x>0} = {\begin{cases}1, &x>0\\0,&x\leq 0\end{cases}}:\ \mathbb{R} \to [0,1],  
\end{equation}

and the sign-negation function, $m(x)=- x:\ \mathbb{R} \to \mathbb{R}$:
\begin{equation}
    {\mathcal{L}}(f)={\mathcal{B}}(f\circ H) \ \ \longleftrightarrow \ \ \mathcal{B}(f)={\mathcal{L}}(f)+
{\mathcal {L}}(f\circ m)\circ m.
\end{equation}

The Laplace transform plays an important role in modern data science, probability modeling, and statistical computing. Suppose $X\geq 0$ is a random variable corresponding to a particular process and a specific probability density function $f_X : \mathbb{R}^+\to \mathbb{R}$. Then, the Laplace transform of the density, i.e., the Laplace transform of the (positive-value) random variable $X$, is the {\it{expectation}} of an exponential function of $X$, i.e., 
$\mathcal{L}\{f_X\}(z)= \mathbb{E}_X \left [e^{-zX}\right ]=\int{e^{-zx} f_X(x)dx}.$

The {\it{moment generating function (MGF)}} of $X\ge 0$ is obtained by substituting $-t$ for the LT argument $z$. This is important, as MGFs facilitate the calculations of the distribution moments. Just like the coefficients of a Taylor series expansion of a nice (infinitely differentiable) function completely describes the function (in a polynomial basis), knowing all moments of a nice distribution completely describes the distribution.  Using the Laplace transform, we can recover the cumulative distribution function, $F_X$, of a continuous random variable $X$, 
\begin{equation}
  \mathcal{L}\left\{F_{X}'\right\}(x)=\underbrace{F_X(0)}_{0}+s\mathcal{L}\left\{F_{X}\right\}(x)\Longrightarrow \boxed{F_X(x)={\mathcal {L}}^{-1}\left\{{\frac {1}{s}} {\mathbb{E}} \left[ e^{-sX}\right]\right \}(x)=
{\mathcal {L}}^{-1}\left\{{\frac {1}{s}}{\mathcal {L}}\{f_X\}(s)\right\}(x).}  
\end{equation}

Before we dive into the newly proposed numerical algorithms, we will review some of the prior work in approximating the inverse Laplace transform (ILT). 

\subsection{Prior Work}

Davies and Martin \cite{davies1979numerical} investigated a number of different methods for numerically inverting the Laplace transform. Their study used (i) methods for expanding $f(t)$ using exponential functions; (ii)  techniques utilizing Gaussian quadrature; (iii) bilinear transformation strategies; and (iv) Fourier series expansion approaches. Their findings indicate that using Laguerre or Chebyshev polynomials accelerated the ILT convergence for a broad set of test functions. However, the authors reported that no method was optimal across all test-cases.

Miller and Guy \cite{miller1966numerical}
investigated numerical strategies to invert the Laplace transform using Jacobi polynomials. They approximated the complex-time function values $F(z)$ at discrete points $z = (\beta + 1 + k)\delta ,\, k = 0,1, \cdots $ by determining a finite number of the coefficients of the series expansion of $f(t)$ in terms of Jacobi polynomials.

Honig and Hirdes \cite{honig1984method} improved the ILT by simultaneously reducing the discretization error, accelerating the convergence of the Fourier series, and approximating the free parameters. The performance of the numerical ILT using only a few function evaluations of the Laplace transform  was reported to be good in a range of applications.

Valkó and Abate \cite{valko2004comparison} compared the numerical inversions of Laplace transform using Gaver functionals with Salzer summation, Wynn's rho algorithm, Levin's u-transformation, Lubkin's iterated w-transformation, and Brezinski's theta algorithm. Although some empirical findings were encouraging, the results were not generalized to identify specific conditions minimizing the ILT reconstruction error.

López-Fernández and Palencia used a quadrature formula and sinc interpolation to invert the Laplace transform \cite{lopez2004numerical}. This technique improves on the prior error bound in terms of the number of nodes used in the quadrature approximation.

Abate and Whitt \cite{abate1995numerical} specifically addressed the problem of numerically approximating the inverse Laplace transform for  probability  distribution functions. Their approach was based on the Fourier-series representation and the use of the finite Fourier cosine transform. 

Durbin proposed a method for numerical inversion of LT that bounds the error on the reconstructed signal independent of the argument $t$ and argues that the error rate can be controlled even in difficult test cases. 

More recently, Brzeziński \cite{brzezinski2018review} examined the performance of different ILT numerical algorithms using C++ and Python implementations. Using classical and difficult test-cases, the author evaluated multiple techniques and reported that methods involving deformed Bromwich contour integration and Fourier series expansions converged faster, had high-accuracy in recovering the original signals, and had a wider range of applications (i.e., are more general).

Rani, Mishra, and Cattani \cite{rani2018numerical,mishra2020laplace,rani2019solving} used orthonormal Bernstein operational matrices to obtain numerical solutions of the ILT. The core of their algorithms approximates the unknown longitudinal signal using truncated series of Bernstein basis polynomials. The authors estimated the Bernstein polynomial coefficients by the operational matrix of integration. Their findings suggest that the root mean square error between exact and the recovered longitudinal signals is small and they obtained rough upper bounds for the error. The same group also employed the numerical inverse Laplace transform for solving fractional order differential equations. Their approach is based on converting the linear fractional differential equations into the linear system of algebraic equations and expressing the unknown function as a series of orthogonal polynomials; specifically, Chebyshev polynomials of the second kind.  


Firouzjaei, Adibi, and Dehghan \cite{mohammadi2021local} used the Laplace transform, along with the local discontinuous Galerkin methods, to solve distributed‐order time‐fractional diffusion‐wave equations. By using LT to convert the PDE to a time‐independent problem, the authors solved these stationary equations by the local discontinuous Galerkin method and simultaneously  discretized the diffusion operators. The ILT provided a direct mechanism we derive the solutions to the original wave equation. This approach is computationally appealing as it is effective and provides a mechanism for parallel execution.

Expanding on most of these prior studies, this paper is focused on robust and practical numerical solutions to the ILT problem for discrete and continuous signals. The manuscript is organized as follows. In section 2, we review the laplace transforms and its various extensions. In section 3, we highlight the contribution of our methodology. In section 4, we define the analytical forms of the forward and inverse Laplace transforms with corresponding algorithmic approximations for various types of functions and domain quantizations. We also introduce a generalization of the Laplace transform for higher dimensional space-time signals. In section 5, we provide theoretical and empirical evidence for algorithm validation of the ILT using analytical functions, observed, and simulated data. Here we also explicate the assumptions and analyze the approximations of the discrete ILT algorithm relative to stochastic domain partitioning and selection of data points within each region. 
The core experiments and algorithmic validation results along with some conjectures about appropriate assumptions and the theoretical approximation limits are shown in section 5.
In the final section 6, we draw study conclusions and suggest future directions for algorithmic improvements. An extended appendix (section 8) provides a discussion of Laplace transform boundedness, some asymptotic analysis derivations, and additional examples of random domain partitioning using uniform, exponential, and equidistant breakpoints.

\section{Laplace Transforms and extensions}
\subsection{Laplace Transform and Dirichlet Series}

The discrete analog of the (continuous) Laplace transform may be defined as follows
\begin{equation*}
    f(s) = \sum_{n=1}^{\infty}a_n e^{-\lambda_n s}, \lambda_1\leq \lambda_2 \leq \cdots \ . 
\end{equation*}
\begin{minipage}{0.45\textwidth}
\vspace{-0.1cm}
Both the Laplace transform and Dirichlet series can be represented via the Laplace-Stieltjes integral
\begin{equation*}
    f(s) = \int_{0}^{\infty} e^{-st}d\mu(t)\ .
\end{equation*}
When $\mu(t)$ is a continuous measure, this produces the Laplace transform, whereas when $\mu(t)$ is a stepwise function (where the differential is the delta function), this effectively produces the Dirichlet series expression. If the Dirichlet series converges at some point $z_0$, then the converging series can help give reasonable numerical approximations for compact subsets within $Re(z')>Re(z_0)$.
\end{minipage}
\begin{minipage}{0.5\textwidth}
\vspace{-1cm}
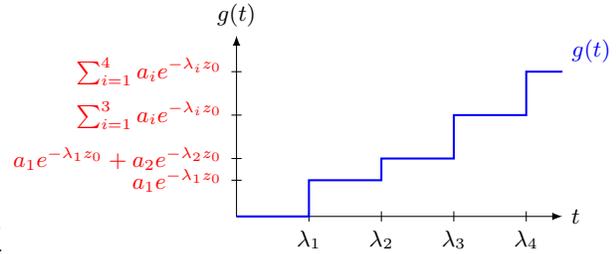
\begin{figure}[H]
\resizebox{\textwidth}{!}{\begin{tikzpicture}[>=latex, y=2cm, font=\small]
    \draw [<->] (0,1.25) node [above] {$g(t)$} |- (4.5,0) node [right] {$t$};
    \foreach \x [evaluate=\x as \y using \x/4] in {1,...,4} \draw (\x,2pt) -- (\x,-2pt) node [below] {$\lambda_{\x}$} ;

    \draw (2pt,0.25) -- (-2pt,0.25) node [left,red] {$a_1e^{-\lambda_1 z_0}$};
    \draw (2pt,0.4) -- (-2pt,0.4) node [left,red] {$a_1e^{-\lambda_1 z_0}+a_2e^{-\lambda_2 z_0}$};
    \draw (2pt,0.7) -- (-2pt,0.7) node [left,red] {$\sum_{i=1}^3a_ie^{-\lambda_iz_0}$};
    \draw (2pt,1) -- (-2pt,1) node [left,red] {$\sum_{i=1}^4a_ie^{-\lambda_iz_0}$};
    \draw [blue,thick] (0,0) -| (1,.25)  -| (2,.4) -| (3,.7) -| (4,1) -- (4.5,1) node [above right] {$g(t)$};  
\end{tikzpicture}
}
\caption{The stepwise function}
\label{stepwisefunc}
\end{figure}
\end{minipage} 

 Specifically, $\forall Re(z)>0$, let  $g(t)=\sum_{\lambda_n\leq t, n\geq 1}a_ne^{-\lambda_nz_0} $ denote a stepwise function (we illustrate a real case in Figure \ref{stepwisefunc})
\begin{equation}
    \frac{f(z+z_0)}{z} =  \sum_{n=1}^{\infty}a_ne^{-\lambda_nz_0}\frac{e^{-\lambda_nz}}{z}=\int_{\lambda_1}^{\infty}(\sum_{\lambda_n\leq t}a_ne^{-\lambda_nz_0})e^{-tz}dt=\mathcal{L}\Big( H(t-\lambda_1)g(t) \Big) .
\label{numerical_dir}
\end{equation}

Equation \ref{numerical_dir} establishes a correpsondence between the laplace transform of a stepwise signal and its landscape in the neighborhood. It suggests approximating with stepwise function may be a viable numerical strategy to compute and invert the Laplace transform which we will see in later sections.

\subsection{Laplace transform on matrix exponentials}
Laplace transform's definition can be extended to matrix exponentials scenario, where the definition naturally arises from considering complex systems that can be modelled by continuous time homogeneous first order ODE systems $\dot{x}=Ax$, where $A$ is an $n\times n$ matrix. This is an effective model for certain biological systems and observed time invariant linear dynamical processes \cite{may1972will}. The stability and the evolution of this equation can be investigated through LT/ILT where $f(t)=e^{tA}, f: \mathbb{R}^{+}\to \mathbb{R}^{n\times n}$
\begin{subequations}
\begin{equation}
    \int_0^{\infty} e^{-ts}e^{tA} dt = \underbrace{(sI-A)^{-1}}_{\text{resolvent}}=\frac{I}{s}+\frac{A}{s^2}+\cdots , \text{ for } |s| \text{ sufficient large}
\end{equation}
\begin{equation}
    x(t) = e^{tA}x(0), \mathcal{L}^{-1}((sI-A)^{-1}) = e^{tA} = I+tA+\frac{(tA)^2}{2!}+\cdots ,
\end{equation}
\end{subequations}
where the time translational invariant property is manifested by $x(t+t')=e^{tA}x(t'), \forall t'>0$.

\subsection{Clifford Algebra Laplace Transform Extensions}


The Clifford algebra, also known as geometric algebra \cite{ablamowicz1996clifford,hestenes2001new,doran2003geometric}, is a generalization for linear vector algebra. Clifford proposed adding inner product with Grassmann's outer product to generalize Grassman Algebra

\begin{equation}
\underbrace{\pmb{ab}}_{\substack{\text{geometric algebra}\\\text{geometric product}}} = \underbrace{\langle \pmb{a},\pmb{b}\rangle}_{\substack{\text{linear algebra}\\\text{symmetric inner product}}} +\underbrace{\pmb{a}\land \pmb{b}}_{\substack{\text{Grassman algebra}\\\text{antisymmetric Wedge product}}}, \langle \pmb{a},\pmb{b}\rangle = \frac{1}{2}(\pmb{ab}+\pmb{ba}), \pmb{a}\land \pmb{b}=\frac{1}{2}(\pmb{ab}-\pmb{ba}) \ .
\end{equation}

This formulation led to a unified framework that can flexibly encode quaternions ($\mathbb{H}\cong \mathrm{Cl}_{0,2}(\mathbb{R})$), complex numbers ($\mathbb{C}\cong \mathrm{Cl}_{0,1}(\mathbb{R})$), and spacetime algebra ($\mathrm{Cl}_{3,1}(\mathbb{R})$). Furthermore, electromagnetic fields $\pmb{F} = \pmb{E}+\pmb{B}i$ may be examined by the $\mathrm{Cl}_{3,0}(\mathbb{R})$ representation using three bivectors representing the magnetic field $B$  \cite{brandstetter2022clifford,hestenes2003oersted}.
The geometric algebra $\text{Cl}_{p,q}(\mathbb{R})$ offers flexible encoding where multidimensional data may be embedded into higher dimensional spaces spanned by the $2^{p+q}$ combinations $\left\{1,\underbrace{e_1,e_2,\cdots,e_{p+q}}_{singltons},\underbrace{e_1e_2,\cdots,e_{p+q-1}e_{p+q}}_{pairs},...,\underbrace{e_1e_2\cdots e_{p+q}}_{n-tuples}\right\}$ of an orthogonal basis set, $\{e_1,e_2,\cdots,e_{p+q}\}$, where the algebraic structure is specified by the {\it{norm}} and the {\it{anticommutative properties}}
\begin{equation}
    e_i^2 =\|e_i\|^2=  1 \text{ for }1\leq i\leq p, \quad e_j^2 = \|e_j\|^2 = -1 \text{ for } p+1\leq j\leq p+q,\quad e_ie_j = -e_je_i \text{  for } i\neq j.
\end{equation}

Next, by transforming the multivector valued vector components of $\pmb{f}(x)$, we will explicate the definition of the Clifford Fourier transform in various situations.
\begin{itemize}[leftmargin=*,noitemsep,topsep=0pt]
\item \underline{Case 1: $\mathrm{Cl}_{2,0}(\mathbb{R})$}. The vector space $G^2$ of the algebra contains the basis $\{1,e_1,e_2,e_1e_2\}$, i.e., $G^2=\text{span}\{1,e_1,e_2,e_1e_2\}$. It's an interesting fact that $(e_1e_2)^2 = e_1e_2e_1e_2 = -e_1e_2e_2e_1=-1$. The bivector $e_1e_2$ is associated with a pseudoscalar (corresponding to highest grade basis) $i_2=e_1e_2, i_2^2=-1$. The pseudoscalar gives an alternative expression for the basis decomposition $\pmb{f}(x)=f_{0}+f_{1}e_1+f_{2}e_2+f_{12}e_{12}=1 (f_0(x)+f_{12}(x)i_2)+e_1(f_1(x)+f_2(x)i_2)$. In the algebraic constraints, $i_2$ behaves like the imaginary unit $i$,  hence the 2D Clifford Fourier transform, operating on spinor $(f_0(x)+f_{12}(x)i_2)$ and vector $(f_1(x)+f_2(x)i_2)$ component separately, is naturally defined as \cite{brandstetter2022clifford}
\begin{equation}
    \pmb{\hat{f}}(\xi) = \mathcal{F}\{\pmb{f}\}(\xi) = \int_{\mathbb{R}^2} \pmb{f}(x)e^{-2\pi i_2 \langle x,\xi\rangle}dx,\quad \forall \xi\in \mathbb{R}^2 ,
\end{equation}
\begin{equation}
    \pmb{f}(x) = \mathcal{F}^{-1}\{\mathcal{F}\{\pmb{f}\}\}(x) =  \int_{\mathbb{R}^2} \pmb{\hat{f}}(\xi)e^{2\pi i_2\langle x,\xi\rangle}d\xi,\quad \forall x\in \mathbb{R}^2 .
\end{equation}

\item \underline{Case 2: $\mathrm{Cl}_{0,2}(\mathbb{R})$}. Quarternions ($\mathbb{H}\cong \mathrm{Cl}_{0,2}(\mathbb{R})$) 
obey a similar decomposition for the Euler's formula, where $e^{|\xi|\pmb{\xi}} = \text{cos}|\xi|+\pmb{\xi}\text{sin}|\xi|$. The unit pure quaternion $\pmb{\xi}=b\pmb{\hat{i}}+c\pmb{\hat{j}}+d\pmb{\hat{k}}, b^2+c^2+d^2=1$ plays a similar role as the imaginary number $\pmb{i}$ in canonical Euler's formula.
Therefore   \cite{ell2014quaternion}
\begin{equation*}
    \pmb{\xi}^n=\begin{cases} (-1)^k \text{ if } n=2k,\\(-1)^k \pmb{\xi}\text{ if } n=2k+1 \end{cases} e^{\xi} = \sum_{n=0}^{\infty}(-1)^k\frac{|\xi|^{2k}}{(2k)!}+\pmb{\xi}\sum_{n=0}^{\infty}(-1)^k\frac{|\xi|^{2k+1}}{(2k)!}= \text{cos}|\xi|+\pmb{\xi}\text{sin}|\xi|\ .
\end{equation*}
This allows us to generalize the general full quarternion's polar representation as
\begin{equation}
    q=a+b\pmb{\hat{i}}+c\pmb{\hat{j}}+d\pmb{\hat{k}}=|q|\Big(\frac{a}{|q|}+\pmb{\xi}\frac{\sqrt{b^2+c^2+d^2}}{|q|}\Big)=|q|e^{\pmb{\xi}\phi}=|q|(\text{cos}\phi+\pmb{\xi}\text{sin}\phi),
\end{equation}
where $|q| = \sqrt{a^2+b^2+c^2+d^2}, \pmb{\xi}= \frac{b\pmb{\hat{i}}+c\pmb{\hat{j}}+d\pmb{\hat{k}}}{\sqrt{b^2+c^2+d^2}},\phi=\text{arctan}(\frac{\sqrt{b^2+c^2+d^2}}{a})$. The Clifford Fourier transform and quaternionic Fourier transform can be defined to take different forms giving rise to different properties \cite{hitzer2007quaternion,hitzer2013quaternion}
\begin{equation}
    \mathcal{F}^{cl}\{f\}(\underbrace{u_1,u_2}_{\pmb{u}}) = \int_{\mathbb{R}^2} f(\pmb{x}) e^{-2\pi e_1 u_1 x_1}e^{-2\pi e_2 u_2 x_2} d\pmb{x}\ ,
\end{equation}
\begin{equation}
    f(\pmb{x}) = (\mathcal{F}^{-1})^{cl}\{\hat{f}\}(\pmb{x}) = \int_{\mathbb{R}^2} \hat{f}(\pmb{u})e^{2\pi e_2 x_2 u_2}e^{2\pi e_1 x_1 u_1}d\pmb{u}\ .
\end{equation}
Notice that due to the non-commutativity the ordering of exponentials to obtain the original $f$ is crucial. The more specialized transform for quaternions is to sandwiching the function using exponentials  \cite{hitzer2007quaternion,hitzer2013quaternion}
\begin{equation}
    \mathcal{F}^{q}\{f\}(\underbrace{u_1,u_2}_{\pmb{u}}) = \int_{\mathbb{R}^2} e^{-2\pi e_1 u_1 x_1} f(\pmb{x})e^{-2\pi e_2 u_2 x_2} d\pmb{x}
\end{equation}
\begin{equation}
    f(\pmb{x}) = (\mathcal{F}^{-1})^{q}\{\hat{f}\}(\pmb{x}) = \int_{\mathbb{R}^2} e^{2\pi e_1 x_1 u_1}\hat{f}(\pmb{u})e^{2\pi e_2 x_2 u_2}d\pmb{u}\ ,
\end{equation}
where $e_1 = \hat{\pmb{i}}, e_2 = \hat{\pmb{j}}, e_1e_2=\hat{\pmb{i}}\hat{\pmb{j}}=\hat{\pmb{k}}$. \\ 

\item \underline{Case 3: Spacetime algebra $\mathrm{Cl}_{3,1}(\mathbb{R})$}.
We can extend the definition to complex signatures and obtaining transforms for spacetime algebras utilizing the pseudoscalars and the quarternion Fourier transform definition with metric signature $(+,+,+,-)$ \cite{hitzer2013quaternion,el2020heisenberg}

\begin{equation}
    \hat{f}(\pmb{u}) = \mathcal{F}\{f\}(\pmb{u}) = \int_{\mathbb{R}^{3,1}} e^{-2\pi e_0 ts} f(\pmb{x})e^{-2\pi i_3 \langle x,u\rangle} d^4\pmb{x}\ ,
\end{equation}
\begin{equation}
    f(\pmb{x}) = (\mathcal{F}^{-1})\{\hat{f}\}(\pmb{x}) = \int_{\mathbb{R}^{3,1}} e^{2\pi e_0 ts}\hat{f}(\pmb{u})e^{2\pi i_3 \langle x,u\rangle}d^4\pmb{u}\ ,
\end{equation}
where $f: \mathbb{R}^{3,1}\to \mathrm{Cl}_{3,1}(\mathbb{R})$, and $i_3$ is the pseudo-scalar in $\mathrm{Cl}_{3,0}(\mathbb{R})$, where the spacetime vectors and spacetime frequency are defined by $\pmb{x} = te_0 +x, x=x_1e_1 +x_2e_2+x_3e_3$ and $\pmb{u} =se_0 +u, u=u_1e_1 +u_2e_2+u_3e_3$ respectively.
\end{itemize}

The main challenge to defining a consistent exponential transform over non-commutative algebras, such as Clifford algebra and matrix algebra, is manifested by the Baker-Campbell-Hausdroff (BCH) condition
\begin{equation}
\begin{split}
    e^{M_1}e^{M_2} &= e^{M_1+M_2+\frac{1}{2}[M_1,M_2]+\frac{1}{12}([M_1,[M_1,M_2]]+[M_2,[M_2,M_1]])+\cdots}\\
    e^{M_1}e^{M_2} &\neq e^{M_1+M_2} \text{ for } [M_1,M_2]\neq 0 \ .
\end{split}
\end{equation}

Clifford-Fourier transforms have direct applications in deep learning and signal processing \cite{brandstetter2022clifford,ell2014quaternion,hitzer2013quaternion}. 
Several extension can be made to define Clifford-Laplace transforms \cite{schertzer2015multifractal,ludkovsky2012multidimensional}, however, the application of the transforms may need to be backed by experiments and evidence of tractability of the inversion formulas. This may need to be investigated separately in future prospective studies.

As a motivational driver, consider a functional magnetic resonance imaging (fMRI) signal $f: \mathbb{R}^{1,3}\to \mathbb{F}$, where $f$ maps the spatiotemporal position to some real or complex value at $(t,x,y,z)$:
\begin{equation}
w=f(t,x,y,z).    
\end{equation}
In the previous expression, $t\in \mathbb{R}^+$ is longitudinal time dimension and $x,y,z$ represent the spatial Cartesian coordinates. The distinction in Clifford algebra is a sign convention difference of choice for metric signature with $\mathrm{Cl}_{1,3}(\mathbb{R})$ correspond to metric signature  $(+,-,-,-)$.  Then, a natural way to extend laplace transform is to replace the canonical scalar product with the bilinear form \cite{schertzer2015multifractal}
\begin{equation}
    g(p) = \mathcal{L}_4(f)(p) := \int_{\mathbb{R}^{1,3}}
f(t,x,y,z)e^{\langle p,q\rangle}d^4 \pmb{x}
\end{equation}

where $p,q = t\gamma^0+x\gamma^1+y\gamma^2+z\gamma^3\in \mathrm{Cl}_{1,3}(\mathbb{R})$ and $\langle p,q\rangle$ is defined by the polarization identity $\langle p,q\rangle=\frac{1}{2}(Q(p+q)-Q(p)-Q(q))$ with $Q$ being the quadratic form, for example $p = p_0\gamma^0+p_1\gamma^1+p_2\gamma^2+p_3\gamma^3$, $Q(p)=t^2-x^2-y^2-z^2$ and $f: \mathbb{R}^{1,3}\to \mathbb{F}$
, $d^4\pmb{x} = dtdxdydz$, $\gamma_0^2=1$,$\gamma_1^2=\gamma_2^2=\gamma_3^2=-1$  

Clifford algebra valued representation can be obtained by convolving with Clifford-valued filters
promoting the original signal space to Clifford spectrum, where each point in the domain $(t,x,y,z)$ is associated with the spacetime Clifford algebra $\mathrm{Cl}_{1,3}(\mathbb{R})$ element through function mapping $F$. Under such procedure another laplace transform analog in the Clifford spectrum is $F: \mathbb{R}^{1,3}\to \mathrm{Cl}_{1,3}(\mathbb{R})$, we can define the (extended Clifford-algebra) Laplace transform $\mathcal{L}_{cl}$ over $\mathrm{Cl}_{1,3}(\mathbb{R})$ by:
\begin{equation}
\begin{aligned}
g(p) = \mathcal{L}_{cl}(f)(p) :&= \int_{\mathbb{R}^{1,3}}
F(t,x,y,z)e^{-pq}d^4 \pmb{x}
\\
&= \int_{\mathbb{R}^{1,3}}F(t,x,y,z)
\sum_{n=0}^{\infty}\frac{(-pq)^n}{n!}d^4\pmb{x} ,
\end{aligned}
\end{equation}
where $g(p)\in \mathrm{Cl}_{1,3}(\mathbb{R})$ and we need gamma matrix $\gamma^i$ representation for elements in the Clifford algebra to perform the full exponentiation operation. Although no new information is generated through this process, Clifford ``frequency'' may yield a new representation of the original signal $f$ encoding some (complex) physical symmetries. Furthermore, it also provides a distinction between the time and space component that we desire for analytics. At the current moment, we do not have a tractable inversion formula and a formal study of convergence is expected.

\subsection{Laplace transform on Groups}
The canonical analysis of Fourier analysis can be extended to topological groups from the result from Peter and Weyl \cite{peter1927vollstandigkeit}, which is known as non-commutative harmonics analysis\cite{gross1978evolution}. The forward and inverse generalized Fourier transform on compact group $G$ are given by \cite{folland2016course,munthe2006group,kondor2018clebsch,maslen1998efficient}

\begin{equation}
\begin{aligned}
\text{Forward}: \hat{f}(\rho_l)=\left[\mathcal{F}_Gf\right]_{l} &= \int_G f(g)\rho_l(g)dg, \\
\text{Backward}: \left[\mathcal{F}^{-1}_G\hat{f}\right]_{l} &= \sum_l d_{\rho_l} tr\left[\hat{f}(\rho_l)\rho_l(g^{-1})\right],
\end{aligned}
\end{equation}
where $dg=d\mu_G(g)$ and $\mu_G(g)$ is the Haar measure, which is an invariant volumetric measure on the group $G$, $d_{\rho_l}$ is the dimension of the representation $\rho_l: H\to \mathbb{C}^{d_{\rho_l}\times d_{\rho_l}}$, and $H$ is a subgroup of $G$. $\rho_l$ denotes the irreducible representation (irreps) of type $l$ and any unitary representation $\rho'$ for compact group $G$ can be decomposed into a direct sum $\oplus$ of irreducible representation via a change of basis. That is \cite{kondor2018generalization},
\begin{equation}
    \rho'(g) = D^{-1}(\underbrace{\rho_1\oplus\rho_2\oplus \cdots}_{\rho})D \approx D^{-1}\underbrace{\begin{bmatrix}
    \rho_1(g) & & \\
    & \ddots & \\
    & & \rho_N(g)
  \end{bmatrix}}_{\text{block diagonal form } \rho^{(N)}}D 
\end{equation}
the irreps are not necessary finite but in actual computation we truncate $\rho$ to the block diagonal form $\rho^{(N)}$ with some reasonable precision. Notice that the change of basis is realized by some transformation $D$, represented by a matrix when operating on finite dimensional vector spaces, or $D=\mathcal{F}_G$ for infinite dimensional vector spaces, $\forall f\in L^2(G)$. Let's consider two examples as special cases.

\underline{$SO(2)$ Example}.
The formulation is fairly straightforward in $SO(2)$, since the unit circle group in complex plane is isomorphic to $SO(2)$. This is becuase we can associate the complex proper rotation $e^{i\theta}\in U(1)$ with orthogonal matrices $\begin{pmatrix}
\cos\theta & \sin\theta \\
-\sin\theta & \cos\theta
\end{pmatrix}\in SO(2)$  i.e., $U(1)=S^1\cong SO(2)$. Since the one dimensional complex irreps $\rho_l(g\equiv e^{i\theta})=e^{il\theta},l\in \mathbb{Z}$ are the circular harmonics, which is exactly the complex Fourier-Euler basis $\left \{\sqrt{\frac{1}{2\pi}}e^{in\theta}\right \}_{n=-\infty}^{\infty}$ for 1D Fourier series in $L^2([0,2\pi])$. Therefore, Peter-Weyl theory for generalized Fourier transform is exactly the 1D Fourier series expansion with basis decomposition from $[0,2\pi]$, i.e.,
\begin{equation}
    f(x) = \overbrace{\sum_{l\in\mathbb{Z}}\frac{1}{2\pi}\underbrace{\int_{0}^{2\pi} f(\theta)e^{il\theta}d\theta}_{\mathcal{F}_Gf}\cdot e^{-ilx}}^{\mathcal{F}^{-1}_G\hat{f}} .
\end{equation}

It is also worth noting the connection to differential equations where circular harmonics ($SO(2)$) and spherical harmonics ($SO(3)$) correspond to the basis solutions of polar and spherical Laplace differential equations\cite{wang2022determinism}.\\

\underline{$SO(3)$ Example}.
$SO(3)$ case is slightly more complicated than $SO(2)$, since the irreducible representations are the Wigner D matrices in $S^2$ represent spherical harmonics. The group element $R\in SO(3)$ can be parametrized by three Euler Angles $\alpha,\gamma\in [0,2\pi)$ and $\beta\in[0,\pi]$ with \cite{lee2018real}
\begin{equation}
    R(\alpha,\beta,\gamma) = 
    \begin{bmatrix} 
	\text{cos}\alpha & \text{sin}\alpha & 0 \\
	-\text{sin}\alpha & \text{cos}\alpha & 0\\
	0 & 0 & 1 \\
	\end{bmatrix}
	\begin{bmatrix} 
	1 & 0 & 0 \\
	0 & \text{cos}\beta & \text{sin}\beta\\
	0 & -\text{sin}\beta &  \text{cos}\beta \\
	\end{bmatrix}
    \begin{bmatrix} 
	\text{cos}\gamma & \text{sin}\gamma & 0 \\
	-\text{sin}\gamma & \text{cos}\gamma & 0\\
	0 & 0 & 1 \\
	\end{bmatrix} .
\end{equation}
The parameters represent the $z-x-z$ Euler angles. The irreducible matrix representations for $SO(3)$ are presented by the Wigner D matrices. A Wigner D matrix of type $l$ $D^{(l)}\equiv [D_{mn}^{l}]_{m,n=-l}^{l}$ is a $(2l+1)\times(2l+1)$ matrix and the Fourier expansion on $SO(3)$ is given by
\begin{equation}
    f(R) = \sum_{l\geq 0}\sum_{m=-l}^l\sum_{n=-l}^l (2l+1) \left (\underbrace{\int_{SO(3)}f(R)D_{mn}^{(l)}(g)dg}_{\hat{f}_{mn}^{(l)}}\right ) D_{nm}^{l}(R^{-1}) .
\end{equation}

Unlike Peter-Weyl theory, which operates on compact groups \cite{peter1927vollstandigkeit}, most work have tried to extend Laplace transform formulation on semigroups or \underline{locally} compact group scenarios \cite{mackey1948laplace,peng1998laplace}. Mackey's \cite{mackey1948laplace} approach to study locally compact abelian group G is to draw the analogy of classical laplace and fourier transform analogy $e^{-st}=e^{-Re(s)t}e^{-i\cdot Im(s)t}$ and generalizing the first term to real-character $r(t)$(taking value on the positive real number with $r(x\cdot y)=r(x)\cdot r(y)$) and the second term to ordinary character $\chi_{y}(t)$(taking values on the unit complex circle with $\chi_{y_1+y_2}=\chi_{y_1}\chi_{y_2}$) :$\mathcal{L}f(\chi_y)=\int_G f(t)r(t)\chi_y(t)dt$. Let's consider the following example. \\

\underline{Laurent Series Example} If $t\in G=(\mathbb{Z},+)$ is the additive integer group, $y\in \hat{G}=(\mathbb{T},\cdot)$ is the multiplicative circle group, $r(t)\in \mathbb{R}_{+}$. This degenerates to the discrete laplace transform analogy (Laurent series or bilateral Z-transform)
\begin{equation}
    \mathcal{L}f(\chi) = \sum_{n=-\infty}^{\infty}f(n)z^n.
\end{equation}
\subsection{Laplace transform as a special case for Meijer-G symbolic computation}
One of the main features of the general Meijer-G functions is their compact representation, which makes them ideal for abstract symbolic computing.

\begin{equation}
\begin{split}
    G_{p,q}^{m,n}(a_1,...a_p;b_1,b_2,...b_q|z)\equiv G_{p,q}^{m,n}\left(z\bigg| \begin{array}{c}
a_p\\
b_q
\end{array}\right)\equiv G_{p,q}^{m,n}\left(z\bigg| \begin{array}{c}
\overbrace{a_1,...,a_n}^{\text{n component}};\overbrace{a_{n+1},...,a_{p}}^{\text{p-n component}}\\
\underbrace{b_1,...,b_m}_{\text{m component}};\underbrace{b_{m+1},...,b_{q}}_{\text{q-m component}}
\end{array}\right)=\\\frac{1}{2\pi i}\int_{C}\frac{\Gamma(b_1-s)...\Gamma(b_m-s)\Gamma(1-a_1+s)...\Gamma(1-a_n+s)}{\Gamma(1-b_{m+1}+s)...\Gamma(1-b_q+s)\Gamma(a_{n+1}-s)...\Gamma(a_{p}-s)}z^sds\ .
\end{split}
\end{equation} 

The  Meijer-G functions are closed under the action of the Laplace transform \cite{beals2013meijer}. Therefore, transformation and inversion of signals can be approximated by symbolic Meijer-G computation (See \cite{alaa2019demystifying} for a neural network symbolic regression example). The closedness property can be seen from the fact that the kernel $e^{-st}$ is a Meijer-G function where
\begin{equation}
    e^{-sx} = G_{0,1}^{1,0}\left(sx\bigg|\begin{array}{c}
;\\
0;\end{array} \right) 
\label{meijergexp}
\end{equation}
and integral of the product of two Meijer-G functions is another Meijer-G function \cite{mathai2006generalized}

\begin{equation}
\begin{split}
&\int_{0}^{\infty}x^{\sigma -1} G_{p,q}^{m,n}\left(wx\bigg|\begin{array}{c}
a_p\\
b_q\end{array}\right) G_{\gamma,\delta}^{\alpha,\beta}\left(\eta x^{k/\rho}\bigg|\begin{array}{c}
c_{\gamma}\\
d_{\delta}\end{array}\right)dx\\&=C_1\cdot G_{\rho\gamma+kq,\rho\delta+k\rho}^{\rho\alpha+kn,\rho\beta+km}
\Bigg(\frac{\eta^{\rho}\rho^{\rho(\gamma-\delta)}}{w^kk^{k(p-q)}}\bigg|\begin{array}{c}
\overbrace{\Delta(\rho,c_1),...,\Delta(\rho,c_{\beta}),\Delta(k,1-b_1-\sigma),...\Delta(k,1-b_m-\sigma)}^{\beta\rho+km};\\
\underbrace{\Delta(\rho,d_1),...,\Delta(\rho,d_{\alpha}),\Delta(k,1-a_1-\sigma),...\Delta(k,1-a_n-\sigma)}_{\rho\alpha+kn};\end{array}\\
&\begin{array}{c}
\overbrace{\Delta(k,1-b_{m+1}-\sigma),...,\Delta(k,1-b_{q}-\sigma),\Delta(\rho,c_{\beta+1}),...,\Delta(\rho,c_{\gamma})}^{\rho(\gamma-\beta)+k(q-m)}\\
\underbrace{\Delta(k,1-a_{n+1}-\sigma),...,\Delta(k,1-a_{p}-\sigma),\Delta(\rho,c_{\alpha+1}),...,\Delta(\rho,c_{\delta})}_{\rho(\delta-\alpha)+k(p-n)}\end{array}\Bigg),
\end{split}
\label{meijergg-gen}
\end{equation}
where
$$C_1=w^{-\sigma}\rho^{1+\frac{1}{2}(\delta-\gamma)+\sum_{j=1}^{\gamma}c_j-\sum_{j=1}^{\delta}d_j} k^{\sigma(p-q)+\frac{1}{2}(q-p)+\sum_{j=1}^pa_j-\sum_{j=1}^qb_j}(2\pi)^{\frac{\rho-1}{2}(2\alpha+2\beta-\gamma-\delta)}(2\pi)^{\frac{k-1}{2}(2m+2n-p-q)}$$
and the notation $\Delta(\cdot,\cdot)$ is a short hand for a set of indexing argument in the Meijer-G function. 
For example, $\Delta(\rho, c_j)=\{\frac{c_j+1}{\rho},\frac{c_j+2}{\rho},...,\frac{c_j+\rho-1}{\rho}\}$, and the indices by default are non-negative integers. The complete proof for the Meijer-G function property is given in (Appendix \ref{meijerg}). Using \ref{meijergg-gen},\ref{meijergexp} and setting $\sigma=k=\rho=1$, we can get one possible version of laplace transform
\begin{equation}
\int_{0}^{\infty}e^{-wx}G_{\gamma,\delta}^{\alpha,\beta}\left(\eta x\bigg|\begin{array}{c}
c_{\gamma}\\
d_{\delta}\end{array}\right)dx=\frac{1}{w}G_{\gamma+1,\delta}^{\alpha,\beta+1}\left(\frac{\eta}{w} \bigg|\begin{array}{c}
0,c_{\gamma}\\
d_{\delta}\end{array}\right) .
\end{equation}

Some animations and a supplementary tutorial can be accessed on \url{https://www.socr.umich.edu/TCIU/HTMLs/Chapter4_Laplace_Transform_Meijer_G_Functions.html}. 
\section{Contributions and applications}
Our numerical strategy imposes  an assumption about the degree of well-behavedness on the signals. In this study, we do not attempt to decode any information near function singularities, which often turn out to contain useful physical information\cite{dingle1971asymptotic}. We utilize the Laplace transform as it provides a natural extension of time series to kime surfaces. Our motivation and justification of this approach is shown below.  

\begin{itemize}
    \item This \textbf{1D to 2D process} transforming the time domain to complex time space reflects many human experiences, such as multiple listener perceptions a joint experience of an audio signal (or musical performance), a temporally dynamic visual scene, or a joint interpretation of textual content or a series of 2D cine images.
    \item There is a need for a robust \textbf{invertible} pipeline time-series $\to$ kime-surfaces $\to$ time-series, and the Laplace transform is invertible for square integrable functions.
    \item From a probabilistic perspective, a moment generating function uniquely defines the distribution of a
random variable and characterize the sequence of its statistical moments.
\item Finally, we note that Laplace transform interfaces with many branches of mathematics and we provide a pictorial overview of this interaction (See Figure \ref{laplace-landscape}).
\end{itemize}
 \tikzstyle{descript} = [text = black,align=center, minimum height=1.8cm, align=center, outer sep=0pt,font = \footnotesize]
 
 \tikzstyle{des} = [text = black,align=center, minimum height=0.5cm, align=center, outer sep=0pt,font = \footnotesize]
 \tikzstyle{activity} =[align=center,outer sep=1pt]

\begin{figure}[h]
    \centering
   \begin{tikzpicture}[>=stealth, node distance=3cm]
\coordinate (main) at (0,0); 


 \node[rectangle, draw, text width=2cm, align=center,fill=ColorTwo!15,text=ColorTwo](main)  {Laplace Transform};

 \node (w) [rectangle, draw, text width=2cm, align=center, left of=main,fill=GreenYellow,text=Maroon] {Probabilistic Extensions};

\node [font=\footnotesize] (top)    at (-4,2) {$\bullet$Stieltjes: Double integration of laplace };
    \node [font=\footnotesize] (middle) at (-4,1.5)  {$\bullet$ Spectral analog of MGF (r.v. ensemble)};
    \node [font=\footnotesize] (bottom) at (-4,1)  {$\bullet$ ``non-commutativity''};

 \node (n) [rectangle, draw, text width=2cm, align=center, yshift=2cm,fill=ColorOne!15,text=ColorThree] {Signal Processing};
 \node [font=\small] (top)    at (0
,3) {$\bullet$Fourier, Radon, Z transforms};

\node (e)[rectangle, draw, text width=4cm, align=center,right of=main, xshift=1cm, ,fill=ColorThree!15,text=blue] {Asymptotic Analysis, Resummation Methods,\\ Divergent Series};

\node (be)[above left of=e, xshift=2cm,yshift=-0.9cm]  {${\scriptstyle \text{Borel Laplace:} \sum_{n=0}^{2N} \frac{c_n}{x^{n+1}} = \int_0^{\infty} dt e^{-xt}\overbrace{\scriptstyle\Big(\sum_{n=0}^{2N}\frac{c_n}{n!}t^n\Big)}^{B_{2N}[f](t)} }$};

 \node (sw) [rectangle, draw, text width=2cm, align=center, below left of=main,fill=Emerald!15, text=Black] {Symbolic Computation};
 
\node (swe) [above of=sw, xshift=-1.5cm,yshift=-2cm,text width=6cm, align=center] {
${\scriptstyle \text{``Parametric''} \int_0^{\infty}\text{ integration: Meijer G}}$\\ ${\scriptstyle\text{\underline{Closed} under Laplace, Convolution}}$};

 \node (se) [rectangle, draw, text width=2cm, align=center, below right of=main,fill=Emerald!15, text=ColorOne] {Kernel Persepectives};
\node (see1) [right of=se, xshift=0.0cm,text width=6cm, align=center] {
${\scriptstyle \text{Markov transition, } f(x\mid Y)\text{ kernel}}$};

\node (see) [below of=see1, text width=6cm, align=center,yshift=2.7cm] {
${\scriptstyle \text{Mercer's theorem; RHKS }}$};

 \draw [->] (main) -- (w);
 \draw [->] (main) -- (n);
 \draw [->] (main) -- (e);
 \draw [->] (main) -- (sw);
  \draw [->] (main) -- (se);

\end{tikzpicture}
    \caption{Schematic of the Laplace transform landscape.}
    \label{laplace-landscape}
\end{figure}
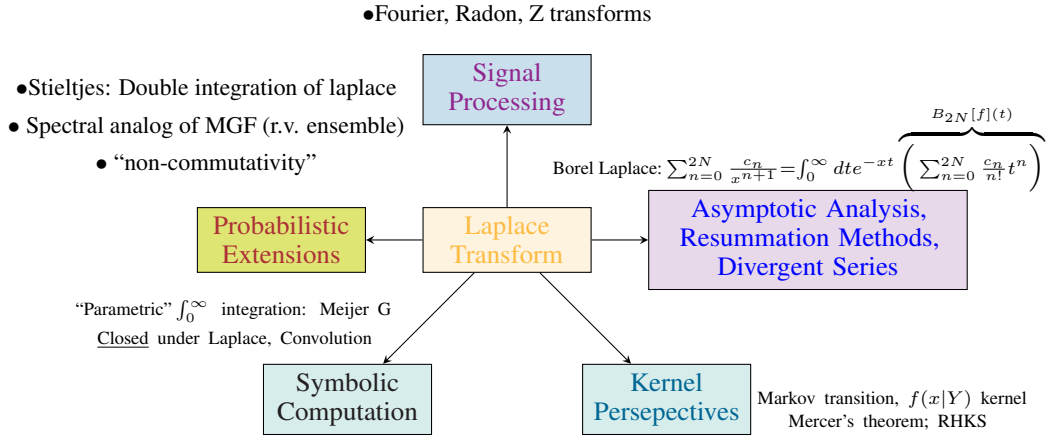

Of course, there are multiple alternative approaches, e.g., Meijer G functions provide natural generalization for almost any real valued special function and thus offers natural kimesurface representations.
The following Lemma, suggests the properties of this specific approach. Namely, that for analytic functions, the Laplace transform naturally \textbf{emerges} from only a pair first principles - {\it{linearity}} and {\it{algebraic-differentiability}}. This highlights the properties and application scenarios of this proposed approach\\

\begin{lemma}
Given any analytic function $f$, the following two conditions determine a unique integral transform:

\begin{itemize}
    \item Linearity: $\mathcal{L}[\alpha f+\beta g]=\alpha \mathcal{L}[f]+\beta \mathcal{L}[g]$
    \item Algebraic Derivatives: $\mathcal{L}[f'](s) = s\mathcal{L}[f](s)-f(0)$
\end{itemize}

The Laplace transform $\mathcal{L}[f](s)=\int_0^{\infty} e^{-st}f(t)dt$ is the \textbf{only} transform that satisfies these two properties (linearity and algebraic differentiability). 
\end{lemma}
For the special case of signals with finite support, without loss of generality, we restrict the function support to the compact interval $[0,2\pi]$. In this case, the algebraic derivative condition is slightly modified to reflect the potential for funciton periodicity:

\begin{itemize}
    \item Algebraic Derivatives(finite): $\mathcal{L}[f'](s) = s\mathcal{L}[f](s)-f(0)+e^{-2\pi s }f(2\pi) .$
\end{itemize}

In other words, linearity and algebraic differentiability completely characterize the (finite domain) Laplace transform  $\mathcal{L}[f](s)=\int_0^{2\pi} e^{-st}f(t)dt$.

The proof of this lemma is provided in Appendix \ref{two-rules}.

In summary,
\begin{enumerate}
    \item The algebraic derivative condition suggests that this approach may provide a unique advantage for giving a \textbf{convenient and robust reconstruction of the first order dynamic function} $f'$, using discretely sampled observations, where the boundary term can be calculated from our spline interpolation. Similarly, the $s\mathcal{L}[f](s)$ term can also be numerically  calculated.
    \item The canonical linearity implies that this methodology naturally inherits the addition and scalar multiplication structure into the kime domain.
    \item Laplace transform is useful to analyze Linear Time Invariant systems, which encode similar
information for repeated observational processes using different sample points
anchored at the specific observed times. 
\end{enumerate}

\section{Numerical LT/ILT Methods}

We will start by formulating domain-partitioning and function base representation strategies for estimating the forward and inverse Laplace transforms.

\subsection{Forward Laplace Transform}
In practice, many observed signals represent windowed versions of discrete processes with restricted domains (finite support). Such discrete signals can be approximated by continuous analytical functions, splines, polynomials, Taylor or Fourier series expansions. Some of these interpolations are based on power functions like $f_{n}(t)=t^{n}$.

In this section, we will demonstrate an approach for Laplace transforming observed discrete data by spline-interpolation with exact integration of the corresponding analytical functions. 
In such analytical expansions, reducing the support of power functions is necessary to control the (model or base function) extreme growth, or decay, as the argument tends to infinity. 

The characteristic function, a.k.a. indicator function, of a half-open real interval $[a, b), \chi_{[a, b)}(t),$ may be expressed in terms of the Heaviside function, $H_{a}(t):$

\begin{equation}
\begin{array}{l}
\underbrace{\chi_{[a, b)}}_{\text {characteristic indicator }}(t)=H(t-a)-H(t-b)=\left\{\begin{array}{l}
0, t<a \\
1,a \leq t<b \\
0, t \geq b
\end{array}\right. ;\ \ 
\underbrace{H}_{\text {Heaviside }}(t)=\left\{\begin{array}{l}
0, t<0 \\
1, t \geq 0
\end{array}\right. .
\end{array} 
\end{equation}

For instance, to constrain the support of the power function $f_{n}(t)=t^{n}$ to the interval $[0,2 \pi),$ we can multiply it by $\chi_{[0,2 \pi)}(t),$ i.e.,
\begin{equation}
f_{n, \operatorname{supp}[0,2 \pi)}(t)=f_{n}(t) \times \chi_{[0,2 \pi)}(t)=\left\{\begin{array}{cc}
f_{n}(t)=t^{n}, & 0 \leq t<2 \pi \\
0, & \text { otherwise }
\end{array}\right. .    
\end{equation}
As a linear operator, the Laplace transform of the power function compactly supported over the interval $[a, b)$ will be:
\begin{equation}
\mathcal{L}\left(f_{n, \operatorname{supp}[a, b)}\right)(z)=\mathcal{L}\left(f_{n} \chi_{[a, b)}\right)(z)=\underbrace{\mathcal{L}\left(f_{n} H(t-a)\right)(z)}_{A_{a}(z)}-\underbrace{\mathcal{L}\left(f_{n} H(t-b)\right)(z)}_{A_{b}(z)} .    
\end{equation}

We need three derivations to explicate this Laplace transform of the windowed power function.
First, let's validate that the Laplace transform of the power function is $\mathcal{L}\left(f_{n}(\cdot) H(\cdot)\right)(z)=\frac{n !}{z^{n+1}}$. This is based on Cauchy-Goursat theorem [58], which states that for a simply connected open set $D \subseteq \mathbb{C}$ with a boundary $\partial D$ representing the closed contour bounding $D,$ a function $f: D \rightarrow \mathbb{C}$ that is holomorphic everywhere in $D$ :
\begin{equation}
\oint_{\partial D} f(z) d z=0 .
\end{equation}

This is the case since all holomorphic functions
\begin{equation}
f(z)=u(z)+iv(z),\quad z=x+iy,\quad dz=dx+idy
\end{equation}

satisfy the Cauchy-Riemann equations:

\begin{equation}
\begin{aligned}
\left| \begin{array}{cc} 
\frac{\partial v}{\partial x}+\frac{\partial u}{\partial y}=0\\ \frac{\partial v}{\partial y}-\frac{\partial u}{\partial x}=0
\end{array} \right .
\end{aligned} .
\end{equation}

Thus,
\begin{equation}
    \begin{aligned}
\oint_{\partial D} f(z) d z=& \oint_{\partial D}(u+i v)(d x+i d y)=\oint_{\partial D}(u d x-v d y)+i \oint_{\partial D}(v d x+u d y) \underbrace{=}_{\text {Green's Thm}} \\
& \iint_{D}\left(-\frac{\partial v}{\partial x}-\frac{\partial u}{\partial y}\right) d x d y+{i}\iint_{D}\left(\frac{\partial u}{\partial x}-\frac{\partial v}{\partial y}\right) d x d y=0 .
    \end{aligned}
\end{equation}

Let's use the following change of variables transformation to show that $\mathcal{L}\left(f_{n}(\cdot) H(\cdot)\right)(z)=\frac{n !}{z^{n+1}}$. We start with
\begin{equation}
\begin{array}{c}
u=z t=(x+i y) t, \quad\epsilon \leq t=\frac{u}{z} \leq L,\quad
z d t=d u
\end{array} 
\end{equation}

Then,
\begin{equation}
\int_{\epsilon}^{L} t^{n} e^{-z t} d t=\oint_{C_{3}}\left(\frac{u}{z}\right)^{n} e^{-u} d \frac{u}{z}= 
\oint_{x \epsilon+i y \epsilon}^{x L+i y L}\left(\frac{u}{z}\right)^{n} e^{-u} d \frac{u}{z}=\frac{1}{z^{n+1}} \oint_{x \epsilon+i y \epsilon}^{x L+i y L} u^{n} e^{-u} d u\ .
\end{equation}

Let's denote the general path line integral in the
complex plane by $I_{C}=\int_{C} u^{n} e^{-u} d u$ and consider the closed (boundary) contour $C=C_{1}+C_{2}-C_{3}+C_{4},$ as in Figure \ref{fig:ct}. We will apply the Cauchy-Goursat theorem over the entire boundary $\partial D=C_{1}+C_{2} - C_{3}+C_{4}$ of the open set $D$.
The main idea is to transform the $I_{C_3}$, which is the Laplace transform, into a more amenable term $I_{C_1}$, as follows
\begin{figure}[H]
\vspace{-0.8cm}
\renewcommand\tabularxcolumn[1]{m{#1}}
\begin{tabularx}{\linewidth}{@{} p{6cm}X @{}}
\adjustimage{width=\linewidth, height=2.5cm,valign=c}{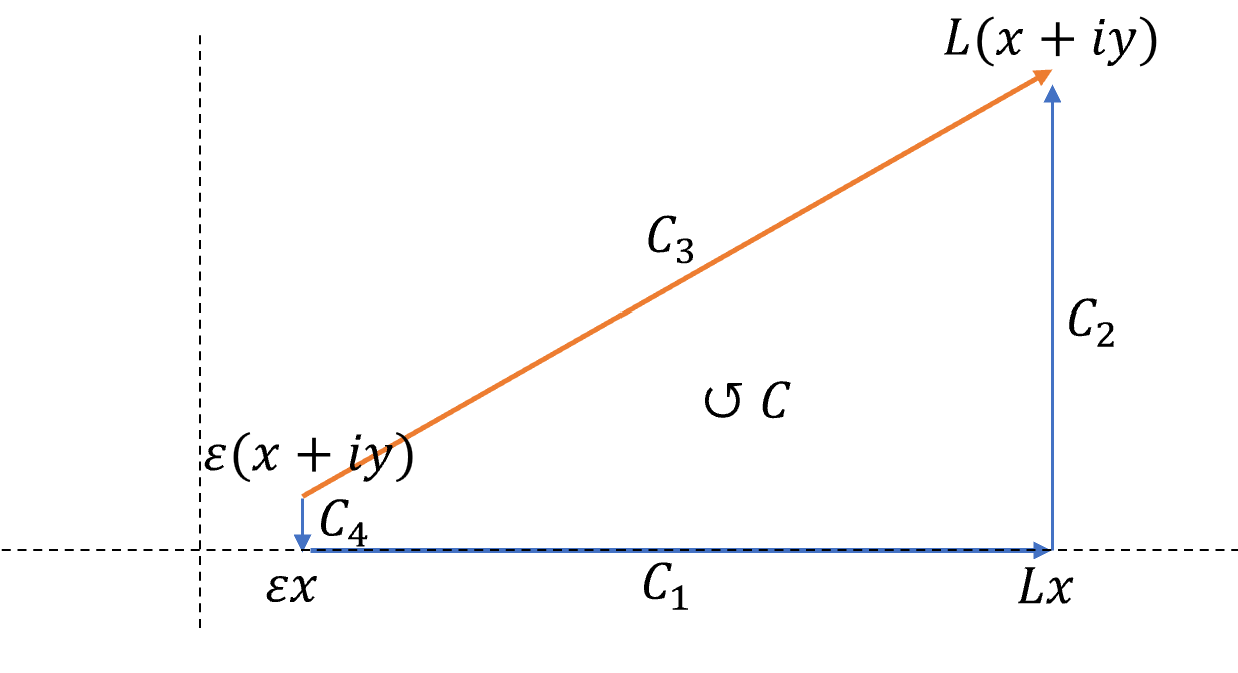}
\caption{Integration contour}
    &   \begin{equation}
\begin{aligned}
\underbrace{\oint u^{n} e^{-u} d u}_{I_C} &=\underbrace{\oint_{C_{1}} u^{n} e^{-u} d u}_{I_{C_1}}+\underbrace{\oint_{C_{2}} u^{n} e^{-u} d u}_{I_{C_{2}}} \\
&-\underbrace{\oint_{C_{3}} u^{n} e^{-u} d u}_{I_{C_{3}}}+\underbrace{\oint_{C_{4}} u^{n} e^{-u} d u}_{I_{C_{4}}}=0 .
\end{aligned}
\end{equation}              \\ 
  \label{fig:ct}
\end{tabularx}
\vspace{-1.5cm}
    \end{figure}
Let's show that $\lim _{L \rightarrow \infty} I_{C_{2}}=\lim _{\epsilon \rightarrow 0} I_{C_{4}}=0$:
\begin{equation}
  \begin{array}{l}
\left|I_{C_{2}}\right|=\left|\oint_{C_{2}} u^{n} e^{-u} d u\right|= 
\left\{\begin{array}{c}u=Lx+iy'\\du=idy'\end{array}\right\}=\left|\int_{0}^{L}(Lx+iy')^n e^{-(Lx+iy')}\cdot idy'\right| \\
\leq \max _{u \in C_{2}} L\left| u^{n} e^{-u}\right|=\max _{u \in C_{2}} L\left| u^{n}\right|e^{-R e(u)} =L \left |(\sqrt{x^2+y^2}L)^{n}\right | e^{-Lx}  \underset{L \rightarrow \infty}{\longrightarrow} 0\ . \\
\left|I_{C_{4}}\right|=\left|\oint_{C_{4}} u^{n} e^{-u} d u\right|=
\left\{\begin{array}{c}u=\epsilon x+iy\\du=idy\end{array}\right\}=\left|\int_{0}^{\epsilon}(\epsilon x+iy)^n e^{-(\epsilon x+iy)}\cdot idy\right|\\
\leq \max _{u \in C_{4}}\left| u^{n} e^{-u}\right|\epsilon=\max _{u \in C_{4}}\left| u^{n}\right|e^{-R e(u)} \epsilon =
\epsilon \left |(\sqrt{x^2+y^2}\epsilon)^{n}\right | e^{-\epsilon x}
\underset{\epsilon \rightarrow 0}{\longrightarrow} 0\ .
\end{array}   
\end{equation}
Therefore, in the limit, $I_{C}=I_{C_{1}}-I_{C_{3}}=0, I_{C_{1}}=I_{C_{3}},$ and
\begin{equation}
\lim _{\epsilon \rightarrow 0^{+}} \lim _{L \rightarrow \infty} I_{C_{1}}(\epsilon, L)=\lim _{\epsilon \rightarrow 0^{+}} \lim _{L \rightarrow \infty} \oint_{\epsilon x}^{L x} u^{n} e^{-u} d u \underbrace{=}_{I_{C_{1}}=I_{C_{3}}} \int_{0}^{\infty} u^{n} e^{-u} d u\ . 
\end{equation}
Summarizing,
\begin{equation}
\begin{split}
\mathcal{L}\left(f_{n}(\cdot) H(\cdot)\right)(z)&=\frac{1}{z^{n+1}}\lim _{\epsilon \rightarrow 0^{+}} \lim _{L \rightarrow \infty} I_{C_3}(\epsilon, L)=
\frac{1}{z^{n+1}}\lim _{\epsilon \rightarrow 0^{+}} \lim _{L \rightarrow \infty} I_{C_1}(\epsilon, L)\\&=
\frac{1}{z^{n+1}}
\int_{0}^{\infty} u^{(n+1)-1} e^{-u} d u
=\frac{\Gamma(n+1)}{z^{n+1}}=\left\{\begin{array}{l}
\frac{n !}{z^{n+1}}, n=\text { integer } \\
\frac{\Gamma(n+1)}{z^{n+1}}, n>0
\end{array}\right . .
\end{split}
\end{equation}

Next, to estimate the term $A_{a, n}(z)=\mathcal{L}\left(f_{n}(t) H(t-a)\right)(z), n\in\mathbb{Z}$  noting $s=t-a$:
\begin{equation}
\begin{split}
A_{a, n}(z)&=\mathcal{L}\left(f_{n}(t) H(t-a)\right)(z)=e^{-a z} \mathcal{L}\left(f_{n}(s+a) H(s)\right)(z)=e^{-a z} \mathcal{L}\left((s+a)^{n} H(s)\right)(z)\\&=e^{-a z} \int_{0}^{\infty}(s+a)^{n} e^{-z s} d s=e^{-a z} \int_{0}^{\infty}\left(\sum_{k=1}^{n}\left(\begin{array}{l}
n \\
k
\end{array}\right) s^{k} a^{n-k}\right) e^{-z s} d s
\\&=
e^{-a z} \sum_{k=1}^{n}\left[\left(\begin{array}{l}
n \\
k
\end{array}\right) a^{n-k} \int_{0}^{\infty} 
s^k e^{-z s} d s\right]=
e^{-a z} \sum_{k=0}^{n}\left [\left(\begin{array}{l}
n \\
k
\end{array}\right) a^{n-k}
\underbrace{\frac{k !}{z^{k+1}}}_{\mathcal{L}\left(s^{k}H(s)\right)(z)}
\right] .
\end{split}
\end{equation}
The windowed Laplace transform power function can be obtained by substituting $a$ with $b$
\begin{equation}
\begin{array}{c}
\mathcal{L}\left(f_{n, \text { supp }[a, b)}\right)(z)=\mathcal{L}\left(f_{n} \chi_{[a, b)}\right)(z)=A_{a, n}(z)-A_{b, n}(z)= \\
e^{-a z} \sum_{k=0}^{n}\left[\left(\begin{array}{l}
n \\
k
\end{array}\right) a^{n-k} \frac{k !}{z^{k+1}}\right]-e^{-b z} \sum_{k=0}^{n}\left[\left(\begin{array}{l}
n \\
k
\end{array}\right) b^{n-k} \frac{k !}{z^{k+1}}\right] .
\end{array} 
\end{equation}
In practice, integration by parts may be used to implement the Laplace transform in a closed-form analytical expression
\begin{equation}
\int t^{n} e^{-z t} d t=-\frac{t^{n}}{z} e^{-z t}+\frac{n}{z} \int t^{n-1} e^{-z t} d t = \frac{1}{z^{n+1}} e^{-z t}\left(\sum_{k=0}^{n} c_{k} z^{k}\right),
\end{equation}

where we need to estimate the coefficients $c_{k}$ and we can transform the indefinite to definite integral over the support $[a, b)$.

Now, suppose we are dealing with discrete data sampled over a finite time interval $[a, b]$. One strategy to compute the Laplace transform of the discrete signal is by direct numerical integration. Alternatively, we can first approximate the data using some base-functions that permit a closed-form Laplace transform analytical representations. Then, we can aggregate the piece-wise linear transforms of all the base functions to derive an analytical approximation to the Laplace transform of the complete discrete dataset. 
For instance, suppose we decide to preprocess the discrete signal by using a fourth-order quartic spline interpolation of the discrete dataset to model the observations to guard against model misspecfication. Thus, we will numerically obtain a finitely-supported spline model $f(t)=f_{n=4, \text { supp }[a, b)}(t)$ of the data over the time interval $t \in$
$[a, b)$ [59]. The Laplace transform for this spline model may be computed exactly as an analytical function, by integrating it against the exponential term:
\begin{equation}
\mathcal{L}\left(f_{n=4, \text { supp }[a, b)}\right)(z)=\int_{0}^{\infty} f_{n=4, \operatorname{supp}[a, b)}(t) e^{-z t} d t=\int_{a}^{b} f_{n=4}(t) e^{-z t} d t .
\end{equation}

Of course, higher order spline interpolations can be used in practice to reflect specific prior assumptions about the process, e.g., inductive bias. For simplicity, we will work with a fixed hyperparameter value ($4$) and assume that the fourth-order spline model smoothly concatenates a set of $d$ 4-order polynomial functions, $f_{n=4, \text { supp }[a, b)} \equiv\left\{f_{k, n=4, \text { supp }\left[a_{k}, a_{k+1}\right)}\right\}_{k=0}^{d-1},$ defined over the domain support on the corresponding partition $\forall 0 \leq k \leq d-1$.

\begin{equation}
\begin{tabular}{ |l|l| }
\hline
\multirow{2}{*}[0.5em]{$\overbrace{f_{n=4, \operatorname{supp}[a, b)}}^{\text {spline model }}=\sum\limits_{k=0}^{d-1} \Big[\overbrace{\sum_{n=0}^{4} q_{n, k} t^{n} \chi_{\left[a_{k}, a_{k+1}\right)}(t)}^{f_{k, n=4, \operatorname{supp} [a_{k}, a_{k+1})}}\Big] $} & $\underbrace{[a, b)}_{\text {support }}=\bigcup_{k=0}^{d} \underbrace{\overbrace{\left[a_{k}, a_{k+1}\right)}^{\text {partition interval}}}_{
\text { support }},$\scalebox{0.8}{%
${\displaystyle a \equiv a_{o}<\cdots<a_{d} \equiv b,}  $} \\
\cline{2-2}
 & $\overbrace{f_{k, n=4, \operatorname{supp}\left[a_{k}, a_{k+1}\right)}}^{\text {spline model component }}=\sum\limits_{n=0}^{4} q_{n, k} t^{n} \chi_{\left[a_{k}, a_{k+1}\right)}(t)   $ \\
 \hline
\end{tabular} 
\end{equation}

Then, the closed-form analytical expression of the Laplace transform of the (approximate) signal can be derived as


\begin{equation}
\scalebox{0.87}{$
\begin{aligned}
\mathcal{L}(\text { signal })(z) \cong & \mathcal{L}\left(f_{n=4, \operatorname{supp}[a, b)}\right)(z)
=\int_{0}^{\infty} f_{n=4, \operatorname{supp}[a, b)}(t) e^{-z t} d t =\sum_{k=0}^{d-1}\left[\sum_{n=0}^{4}\left(q_{n, k}
\left(A_{a_{k}, n}(z)-A_{a_{k+1}, n}(z)\right)\right)\right] \\
=&\underbrace{\sum_{k=0}^{d-1}}_{\text {spline }} \left [ \underbrace{\sum_{n=0}^{4}}_{ \text { polynomial }} \left ( \underbrace{q_{n, k}}_{\text { poly coef }} \underbrace{\int_{a_{k}}^{a_{k+1}} t^{n} e^{-z t} d t}_{\mathcal{L}\left(f_{n, \text { supp }\left[a_{k}, a_{k+1}\right)}\right)(\mathrm{z})} \right ) \right ]
=:F_{d,4}(z) ,
\end{aligned}
$
}
\end{equation}
where
\begin{equation}
A_{a_{k}, n}(z)-A_{a_{k+1}, n}(z)=e^{-a_{k} z} \sum_{l=0}^{n}\left[\left(\begin{array}{l}
n \\
l
\end{array}\right) a_{k}^{n-l} \frac{l !}{z^{l+1}}\right]-e^{-a_{k+1} z} \sum_{l=0}^{n}\left[\left(\begin{array}{l}
n \\
l
\end{array}\right) a_{k+1}^{n-l} \frac{l !}{z^{l+1}}\right] .
\end{equation}
Hence, the strategy of first spline-interpolating the discrete 
signal, followed by exact integration of these analytical functions, represents an alternative approach to numerically integrating observed signals even under a significant presence of noise. This method circumvents potential instabilities and extremely time-consuming numerical integration. To obtain the Laplace transform of the dataset, we effectively sum up the integrals of the Laplace transformed power functions, across
  
polynomial powers {$0,1,2,\cdots ,n$ } and across the spline partition of the signal time-domain support.

By estimating an upper bound on the error magnitude, we can contrast the analytic definition of Laplace Transform, $F_a(z):= \mathcal{L}(f)(z)$, with its estimate, $F_{d,4}(z)$,
\begin{equation}
\begin{aligned}
\left |F_{d,4}(z)-F_a(z)\right |=& \left |\int_a^b (f_{n=4, \operatorname{supp}[a, b)}(t)-f(t))e^{-z t} d t\right |
\le\int_{a}^{b} \sup_{a\le t<b}|f_{4, \operatorname{supp}[a, b)}(t)-f(t)|\cdot |e^{-z t}| d t \\
=&\sup_{a\le t<b}|f_{4, \operatorname{supp}[a, b)}(t)-f(t)|\times \frac{e^{-Re(z)a}-e^{-Re(z)b}}{Re(z)}\ .
\end{aligned}
\end{equation}
Our discrete LT method is compatible with the analytic one when the difference between the spline function and original function is minimized. In practice, we often work with functions 
or signals defined over the same time interval, e.g., fix $a=0, b=2\pi$. Then, the accuracy of this algorithm depends only on the fit of the chosen spline approximation function. On the other hand, if we consider naturally discrete datasets, the discrete LT still provides a valid definition on its own.

For a time-varying function, $f(t):\mathbb{R}^{+} \rightarrow \mathbb{C}$, there is also a direct connection between the Laplace ($\mathcal{L}$) and the Fourier ($\mathcal{F}$) transforms, where $H(t):\mathbb{R}^{+} \rightarrow \mathbb{C}$ is the Heaviside function:
\begin{equation}
\scalebox{0.9}{$\mathcal{L}(f(\cdot) H(\cdot))(z) \equiv \int_{0}^{\infty} f(t) H(t) e^{-z t} d t \underbrace{=}_{z=i \omega \atop-i z=\omega} \int_{0}^{\infty} f(t) H(t) e^{-i \omega t} d t =\mathcal{F}(f(\cdot) H(\cdot))(\omega)\equiv \mathcal{F}(f(\cdot) H(\cdot))(-i z).$}
\end{equation}
Suppose $w \in \mathbb{C}$ and we consider the Laplace transform of the exponential decay function
\begin{equation}
f(t)=e^{-w t} H(t), t \geq 0 , 
\mathcal{L}(f)(z)=\mathcal{L}\left(e^{-w t} H(t)\right)(z)=\frac{1}{z+w}\ .
\end{equation}
The Laplace transform of the shifted function $f(t)=e^{-w(t+a)}$ will be
\begin{equation}
\mathcal{L}(f)(z)=\mathcal{L}\left(e^{-w(t+a)} H(t)\right)(z)=e^{-w a} \mathcal{L}\left(e^{-w t} H(t)\right)(z)=\frac{e^{-w a}}{z+w} .
\end{equation}
For discrete data, alternative Laplace transform approximations may be obtained using other base functions. For instance, in the Fourier trigonometric basis, the Laplace transforms of the windowed trigonometric functions may be obtained as follows
\begin{equation}
\mathcal{L}(\sin (w t) H(t))(z)=\frac{w}{z^{2}+w^{2}},\quad \mathcal{L}(\cos (w t) H(t))(z)=\frac{z}{z^{2}+w^{2}} .
\end{equation}

Table 4.1 contains examples of some common basis functions involved in solving differential equations and their corresponding Laplace transforms.
\begin{center}
\begin{table}[H]
\center
\makegapedcells
    \begin{tabular}{||c|c||c||}
    \hline
    Basis Function & Transformation formula&Orthogonality Condition\\
    \hline
    \hline
    \textit{Sine Basis} &  $\mathcal{L}(\sin (w t) H(t))(z)=\frac{w}{z^{2}+w^{2}}$&$\int_{0}^{2\pi}\sin(nt)\sin(mt)dt=\pi\delta_{m,n}$\\
    \hline
    \textit{Cosine Basis} & $\mathcal{L}(\cos (w t) H(t))(z)=\frac{z}{z^{2}+w^{2}}$&$\int_{0}^{2\pi}\cos(nt)\cos(mt)dt=\pi\delta_{m,n}$ \\
    \hline
    \textit{Bessel Basis} & $\mathcal{L}(J_n(at) H(t))(s)=\frac{(\sqrt{s^2+a^2}-s)^n}{a^n\sqrt{a^2+s^2}}$ &$\int_{0}^{1}tJ_{n}(at)J_{n}(bt)dt=\frac{1}{2}J_{n}'(a)^2\delta_{a,b}$ \\
    \hline
    \textit{Laguerre  Basis}\cite{laplacetr} & ${\scriptstyle\mathcal{L}(L_{n}^{(\alpha)}(t) H(t))(z)=\frac{\Gamma(1+\alpha+n)}{\Gamma(1+\alpha)n!s}M(-n,1+\alpha;\frac{1}{s})}$ &${\scriptstyle\int_{0}^{\infty}t^{\alpha}e^{-t}L_{n}^{(\alpha)}(t)L_{m}^{(\alpha)}(t)dt= \frac{\Gamma(1+\alpha+n)\delta_{m,n}}{n!}}$ \\
    \hline
    \textit{Legendre Basis}\cite{legendre}&\begin{math}\begin{aligned}[t]
    &{\scriptstyle\mathcal{L}(P_n(t)H(t))(s)=\frac{1}{2}\sqrt{\pi}\Bigg(\sqrt{\frac{2}{s}}I_{-n-1/2}(s)+ }\\&{\scriptstyle(-1/2s)^{\lfloor\frac{|n|+2}{2}\rfloor-\lceil\frac{|n|}{2}\rceil} 
   {}_1F_2(1;\frac{1}{2}n+2-\frac{1}{2}\lceil\frac{|n|}{2}\rceil),}\\&{\scriptstyle1+\frac{1}{2}(\lceil\frac{|n+1|}{2}\rceil-\lfloor\frac{|n+1|+2}{2}\rfloor)-\frac{1}{2}n;\frac{1}{4}s^2)}\Bigg)\\
    \end{aligned} 
    \end{math}&$\int_{-1}^{1}P_{n}(t)P_{m}(t)dt=\frac{2}{2n+1}\delta_{m,n}$\\
    \hline
    \textit{Hermite Basis} &${\scriptstyle\mathcal{L}(H_{n}(t) H(t))(z)=2^n\frac{\Gamma(1+n)}{s^{1+n}}{}_{2}F_2(-\frac{n}{2},\frac{1-n}{2};\frac{-n}{2},\frac{1-n}{2};\frac{s^2}{4})}$&
$\int_{-\infty}^{\infty}H_m(x)H_n(x)e^{-x}dx=\sqrt{\pi}2^nn!\delta{m,n}$
    \\
    \hline
    \textit{Chebyshev Basis}\cite{piessens2000computing}&
    ${\scriptstyle\mathcal{L}(H_{n}(t) x^{\alpha}T_n(1-2x))(z)=\frac{\Gamma(1+\alpha)}{z^{1+\alpha}}{\scriptstyle{}_{3}F_1(-n,n,1+\alpha;1/2;1/z)}}$& $\int_{-1}^{1}\frac{T_{n}(x)T_m(x)}{\sqrt{1-x^2}}dx=\frac{1}{2}^{\delta{n,0}}\pi\delta_{n,m}$\\
    \hline
\end{tabular}
\caption{Examples of analytical Laplace transforms of some common base functions. The function $M(\cdot)$ is the confluent hypergeometric function and $I_{n}(\cdot)$ is the modified Bessel function of the first kind, or order $n$. We provide selected sample calculations using Meijer-G functions in Appendix \ref{laplace_calmeijerg} }
\end{table}

\end{center}

Since the LT is a linear operator, any function or signal that can be expressed in the trigonometric basis, or more generally any orthonormal basis, would have a closed form analytical Laplace transform representation. For instance, this study \cite{ricci2021tricomi} explicates the Laplace transform using of the Tricomi's method based on Laguerre polynomials.

\subsection{(Analytic) Inverse Laplace Transform of Entire Functions}

To validate the ILT approximation, we use an approximation formula \cite{valsa1998approximate} proposed by Valsa and Brancik for the inverse Laplace transform (ILT). Their method provides satisfactory and adjustable accuracy with impressive computational efficiency. The method tries to recover the (time) signal $f=ILT(F)$, starting with the Bromwich integral

\begin{equation}
f(t)=\mathcal{L}^{-1}\{F(s)\}=\frac{1}{2 \pi i} \int_{\gamma-i \infty}^{\gamma+i \infty} F(s) \mathrm{e}^{s t} \mathrm{~d} s\ .  
\end{equation}

The basic assumptions on the function $F(z)$, which implicitly represents a Laplace transform of some original function or signal, include: (1) $F(z)$ is univariate analytic (regular) function over $\operatorname{Re}\{z\}>0$; (2) $\lim_{|z|\to \infty} F(z)=0$; and (3) $F^{*}(z)=F\left(z^{*}\right)$, where the asterisk denotes complex conjugation.

Valsa and Brancik's ILT estimate is based on approximation of $e^{zt}$:
\begin{equation}
\mathrm{e}^{z t} \sim E_{\mathrm{c}}(z t, a):=\frac{\mathrm{e}^{a}}{2 \cosh (a- z t)}=\frac{\mathrm{e}^{a}}{\mathrm{e}^{a} \mathrm{e}^{-z t}+\mathrm{e}^{-a} \mathrm{e}^{z t}}=\frac{\mathrm{e}^{z t}}{1+\mathrm{e}^{-2 a} \mathrm{e}^{2 z t}},
\end{equation}

where the subscript "c" refers to $\cosh$ and $a >\gamma t$ is a parameter chosen later to ensure that  $\left|\mathrm{e}^{-2 a} \mathrm{e}^{2 s t}\right| \ll 1$.

Then, we can express the ILT reconstruction as
\begin{equation}
\begin{aligned}
f_{\mathrm{c}}(t, a)&=\frac{1}{2 \pi i} \int_{\gamma-i \infty}^{\gamma+i \infty} F(z) E_{\mathrm{c}}(z t, a) \mathrm{d} z \\
&=\frac{1}{2 \pi i} \int_{\gamma-i\infty}^{\gamma+i \infty} F(z) \mathrm{e}^{z t} \mathrm{~d} z +\frac{1}{2 \pi i} \int_{\gamma-i \infty}^{\gamma+i \infty} F(z) \sum_{n=1}^{\infty}(-1)^{n} \mathrm{e}^{-2 n a} \mathrm{e}^{(2 n+1) z t} \mathrm{~d} z \\
&=f(t)+\sum_{n=1}^{\infty}(-1)^{n} \mathrm{e}^{-2 n a} f[(2 n+1) t]
=f(t)+\epsilon_{\mathrm{c}}(t, a) .
\end{aligned}
\end{equation}

Note the function $E_{\mathrm{c}}(z t, a)$ can be expressed as an infinite sum of rational functions of $z t$ using the series expansion for $\frac{1}{\cosh z}$, which is a meomorphic function (See figure \ref{fig:cosh}) with simple pole $i\frac{\pi}{2}+\pi i \mathbb{Z}$ (See Supplementary material for derivation and evaluation)
\begin{equation}
\frac{1}{\cosh z}=2 \pi \sum_{n=0}^{\infty} \frac{(-1)^{n}(n+1 / 2)}{(n+1 / 2)^{2} \pi^{2}+z^{2}} \Longrightarrow E_{\mathrm{c}}(z t, a)=\pi \mathrm{e}^{a} \sum_{n=0}^{\infty} \frac{(-1)^{n}(n+1 / 2)}{(n+1 / 2)^{2} \pi^{2}+(a-z t)^{2}}.
\label{cosh:eq}
\end{equation}

Therefore,
\begin{equation}
\begin{aligned}
f_{\mathrm{c}}(t, a)&=\frac{\mathrm{e}^{a}}{2 i} \int_{\gamma-i \infty}^{\gamma+i \infty} F(z) \sum_{n=0}^{\infty} \frac{(-1)^{n}(n+1 / 2)}{(n+1 / 2)^{2} \pi^{2}+(a-z t)^{2}} \mathrm{~d} z\\
&=\frac{\mathrm{e}^{a}}{2 i} \sum_{n=0}^{\infty}(-1)^{n}(n+1 / 2) \int_{\gamma-i \infty}^{\gamma+i \infty} \frac{F(z)}{(n+1 / 2)^{2} \pi^{2}+(a-z t)^{2}} \mathrm{~d} z =\frac{\mathrm{e}^{a}}{2 i} \sum_{n=0}^{\infty}(-1)^{n}(n+1 / 2)I_{n} .
\end{aligned}    
\end{equation}

The path integral $I_{n}$ can be considered as part of a large semicircle on the right half plane and can be evaluated through its residue at two simple poles $z_{1,2}=\frac{a \pm i(n+1 / 2) \pi}{t}$. Since $a\in\mathbb{R}$, then $z_1 = z_2^*$.
We assume that the function $F$ will annihilate the path integral, i.e.,

\begin{equation}
\begin{aligned}
\text{res}_{-1,z_1}g &= \lim_{z\to z_1} (z-z_1)\frac{F(z)}{(n+1 / 2)^{2} \pi^{2}+(a-z t)^{2}}\underbrace{=}_{\text{l'H\^{o}spital}}-\frac{F(z_1)}{2 \pi i t(n+1 / 2)} ,\\
I_{n} &=2\pi i(\text{res}_{-1,z_1}g+\text{res}_{-1,z_2}g)=-2 \pi i\left[\frac{F\left(z_{1}\right)}{2 \pi i t(n+1 / 2)}-\frac{F\left(z_{2}\right)}{2 \pi i t(n+1 / 2)}\right]\\
&=\frac{F\left(z_{1}{ }^{*}\right)-F\left(z_{1}\right)}{t(n+1 / 2)}=\frac{F^{*}\left(z_{1}\right)-F\left(z_{1}\right)}{t(n+1 / 2)} =\frac{-2 i \operatorname{Im}\{F([a+i(n+1 / 2) \pi] / t)\}}{t(n+1 / 2)} .
\end{aligned}  
\end{equation}

Hence, we get
\begin{equation}
f_{\mathrm{c}}(t, a)=-\frac{\mathrm{e}^{a}}{t} \sum_{n=0}^{\infty}(-1)^{n} \operatorname{Im}\left\{F\left[\frac{a}{t}+i\left(n+\frac{1}{2}\right) \frac{\pi}{t}\right]\right\},
\end{equation}

which represents a good approximation of the original function $f$ we aimed to recover.

Valsa and Brancik \cite{valsa1998approximate} use the so-called {\it{Euler transformation}} to accelerate the convergence of the alternating sum above. The original infinite series and the Euler transformation are given by \cite{ooura2001continuous}
\begin{equation}
\begin{aligned}
S&=y_{1}-y_{2}+y_{3}-y_{4}\cdots = \sum_{i=1}^{\infty}{(-1)^{i+1}y_{i}} = \sum_{t=1}^{\infty}{e^{i\pi t}y_{t}} \\
S_{\text{Euler}} &= \frac{1}{2}\sum_{n=1}^{\infty} \left (-\frac{1}{2}\Delta \right )^{n-1} y_1, \quad \text{where }\Delta y_n\equiv y_{n+1}-y_n .
\end{aligned}
\end{equation}

As shown in the appendix, $S_{\text{Euler}}=S$ because $\sum_{n=i}^{\infty}\frac{ \binom{n}{i}}{2^n}=2$ is replaced by a finite sum truncated at $N$
\begin{equation}
\begin{aligned}
S_{\text{Euler}}^{(N)} &= \frac{1}{2}\sum_{n=1}^N \left (-\frac{1}{2}\Delta \right )^{n-1} y_1  = \sum_{j=1}^N\overbrace{\sum_{m=j+1}^{N+1}\frac{1}{2^{N+1}}\binom{N+1}{m}}^{w_j^{(N)}}(-1)^{j-1}a_j\\
&=
2^{-N}\left(E_{1} y_{1}-E_{2} y_{2}+\cdots+(-1)^{N-1} E_{N} y_{N}\right) ,   
\end{aligned}
\end{equation}

where truncation parameter $N$ is a parameter up to our choice and
\begin{equation}
E_{n_{\mathrm{dif}}}=1, E_{k}=E_{k+1}+\left(\begin{array}{c}
N+1 \\
k
\end{array}\right), N,N-1, \cdots, 1 .
\end{equation}

Thus, the Euler transformation means only that $N$ additional terms are added with gradually decreasing weights. Otherwise, their evaluation remains the same as that of the terms of the basic sum. 

\subsection{(Discrete) Inverse Laplace Transform}

For nice analytic functions, there are some efficient algorithms to compute the inverse Laplace Transform. However, these approaches heavily rely on the complex analytic properties of the input functions, which in general can be subtle and delicate. In practice, we often encounter situations where the data represents a limited number of sample points of an unknown complex function. Thus, the kime-surface, defined as a complex-valued function over the complex time domain, is only observed with limited accuracy over a limited region in $\mathbb{C}$. This necessitates the development of new algorithms for estimating the corresponding time-domain functions representing the inverse Laplace transforms of general discrete-lattice sampled kime-surfaces.

\subsubsection{Discrete  ILT Quantization}
We set our function space as $L^2([0,2\pi))$, and always use 2-norm, unless otherwise specified. Given a set of (sampled) data points $S = \{(z_i, b_i= F(z_i))\}_{i=1}^N$ on the complex plane $\mathbb{C}$, we want to recover a temporal signal
\begin{equation}
f = \mathcal{L}^{-1}F    
\end{equation}
from observed kime-surface. Assume the sample data is indeed produced by some unknown $f$ that has somewhat nice time regularity, e.g., $f$ is differentiable with a bounded derivative. If $f$ is a piecewise constant function subject to the partition $\mathbf{p}=\{p_0,p_1,\cdots,p_n\}$, where $0 = p_0<p_1<\cdots <p_n = 2\pi$, then $f$ can be considered as a vector 
\begin{equation}
\mathbf{u}^p = (u_1,u_2,\cdots ,u_n),  
\end{equation}
where $\mathbf{u}^p$ indicates the vector representing the value of $f$ on a partition $p$ and $f^p(x) = u_i$,  $\forall x\in(p_{i-1},p_i)$.

\textbf{Remark}: The goal is to find $f$, not $\mathbf{u}^p$. In fact, the partition $p$ is not really given by the problem. We are also free to choose other kinds of finite representations, as may be appropriate. The specific partition is not part of the problem. There are some subtle considerations with this generic quantization approach. Conventionally, we may use finite differencing or finite element methods to set up a minimizing problem. However, then we have no way of controlling the (potentially problematic) optimizer. This may lead to unstable solutions then the matrix involved in the solution is (close to being) non-invertible. However, this quantization scheme targets quantifying and analyzing the approximation error.

\subsubsection{Pseudo algorithm}
\label{randilt}
Given a sampled kime-surface (data), $S = \{(z_i, b_i= F(z_i))\}_{i=1}^N$, over a domain where $z_i$ has bounded (non-negative) real part, we can recover the function $f$ using the pseudo algorithm shown in {\bf{Algorithm 1}} box. We utilized randomization and resampling (as opposed to deterministic selection) along the kime surface to enhance numerical and computational stability.
\begin{algorithm}[H]
\caption{Randomized ILT}\label{alg:cap}
\begin{algorithmic}[1]
\State $N_1,N_2 \gets prior\_estimate$ is the partition size
\State $itn \gets g(N_1, N_2, prior\_estimate)$ is the number of attempts ({\it{iterations}})
\For{$1\le k \le itn$} 
    \State $\mathbf{p}^k \gets$ Random $N_1$ size partition on interval $(0,2\pi)$, according to distribution $P$.
    \State $\mathbf{b} = (b_i)\gets $ Random $N_2$ points $(z_i, b_i)$ from the dataset $S$
    \State $\mathbf{A} \gets (a_{ij} = \frac{1}{-z_i}(\exp(-z_ip_j)-\exp(-z_ip_{j-1}))$, note that $\mathbf{A}$ is the matrix computing LT of a quantized piecewise constant function
    \State $\mathbf{u}^{pk} = \mathbf{A}^{\dagger}\mathbf{b}$
    \State $f^{pk}\gets \mathbf{u}^{pk}$ as a piecewise constant function
\EndFor
\State $f \gets \frac{1}{itn}\sum_k f^{pk}$
\end{algorithmic}
\end{algorithm}
\textbf{Remark}: The parameters $N_1, N_2$ and $itn$ can be determined by some prior estimates, although the algorithm itself may serve as a prior estimate when $N_1, N_2$ can be defaulted to something trivial like $\sqrt{N}$, and $itn$ may take a similar magnitude. We will clarify this hyperparameter estimation protocol in the validation section. The intuition between the randomized approach is to utilize randomness to (at least try to) gauge the problem of smallest singular value and the fact that sampling on a curve or densely in a neighborhood is sufficient to reconstruct a complex analytic function.

\section{Experiments and Validation}

Next we validate the performance of the proposed LT and ILT approximations using experimental data. The specific experiments are chosen to demonstrate practical evidence, provide intuition behind the proposed discrete ILT quantization approach, and to identify some potential algorithmic limitations. All reported results are obtained by an R implementation of the LT/ILT algorithm, which is available in the TCIU R-package and in the project GitHub repository ({\url{https://github.com/SOCR/TCIU}}) and on CRAN ({\url{https://cran.r-project.org/web/packages/TCIU/}}).

\subsection{Validation of LT}
We will apply the proposed discrete LT algorithm to several discrete signals represented as sets of data points on the bounded interval $[0,2\pi]$ (time or space). To accurately estimate the performance and quantify error rates, the observed data samples are derived from functions with known analytic LTs, which allows us to compare the discrete estimates with their analytical counterparts. We will also invert the function back using the continuous ILT  \cite{valsa1998approximate} and compare with the original data points.


In Figure \ref{fig:km01}, we apply the Laplace transform to the {\it{sine}} function supported on $(0,2\pi)$, and the plot is over the range $(0,20)+(0,20)*1i$, where the absolute value of the transformed signal is the surface height and the color represents the phase parameter of the complex value. In Figure \ref{fig:lt01}, we compare the analytic result with our numerical implementation when the {\it{sine}} function is sampled at regular grid points separated by $2\pi/200$. Note $\mathcal{L}(\sin\cdot\chi_{[0,2\pi]})(z)=\int_{0}^{2\pi} {\sin(t)e^{-zt}dt} =\frac{1}{2i} \left (\frac{e^{2\pi(i-z)}}{i-z}+\frac{e^{2\pi(-i-z)}}{i+z}+\frac{2i}{z^2+1}\right )=
\frac{1-e^{-2\pi z}}{z^2+1}$, 
and we plot that function (red color) on the line $z = 0.1+ (0,20)*1i$. The blue points on the graph represent the result of the proposed numerical LT approach. The absolute values of the real and estimated LT are plotted on a logarithm scale to facilitate easy comparison. The results suggest a very close visual match between the the blue points (estimated) and red line (true) values. In our experiments, the accuracy is on the order $10^{-10}$, and even higher accuracy can be achieved with using denser point sampling, which is natural as the forward Laplace transform is a linear integral transform.

\begin{figure}[H]
  \centering
  \includegraphics[width=0.54\textwidth]{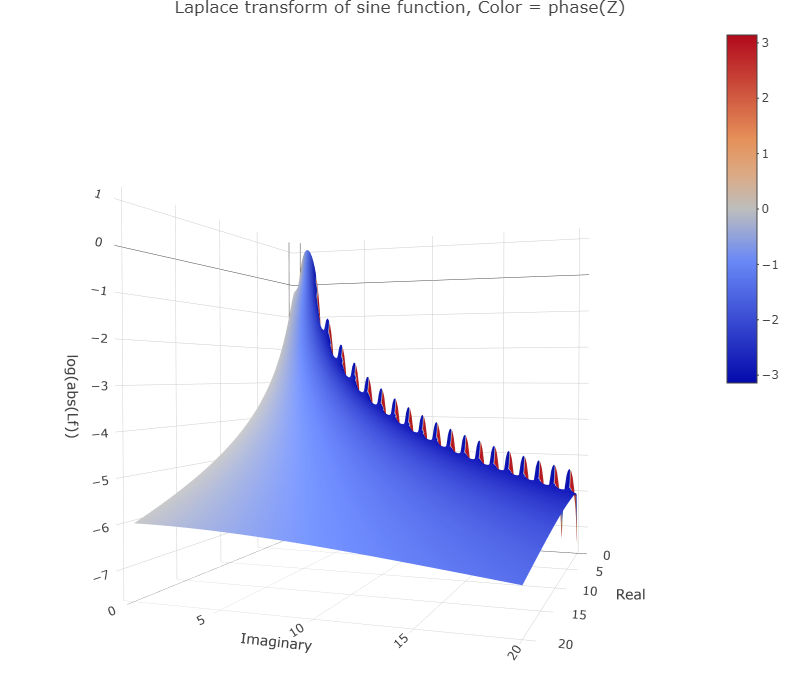}
  \includegraphics[width=0.45\textwidth]{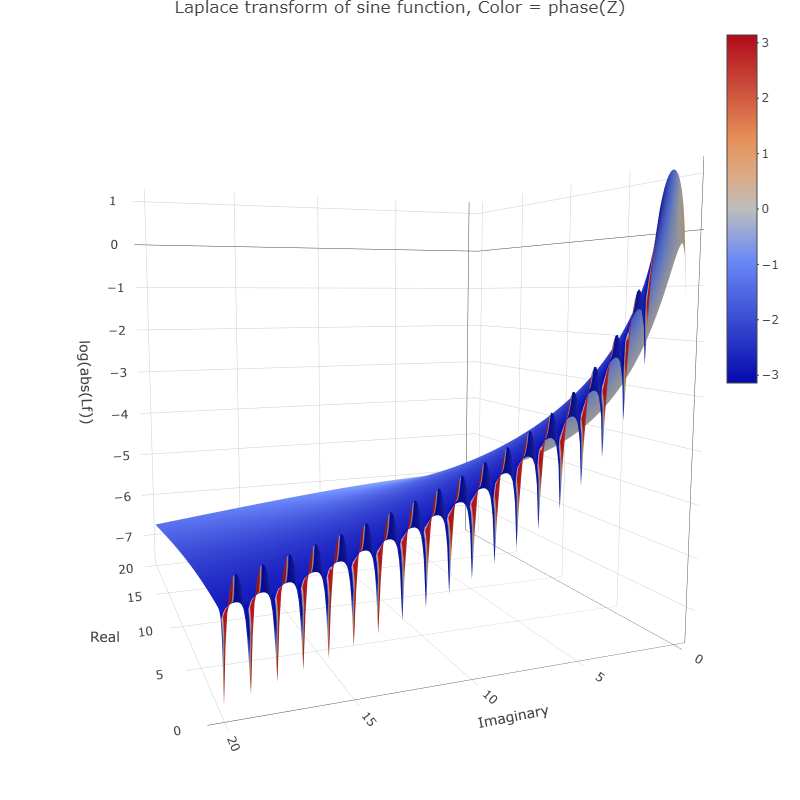}
  \caption{Laplace Transform of the {\it{sine function}}, vertical-axis (amplitudes of the functional values) are on a logarithm scale.}
  \label{fig:km01}
\end{figure}

\begin{figure}[H]
  \centering
  \includegraphics[width=0.7\textwidth]{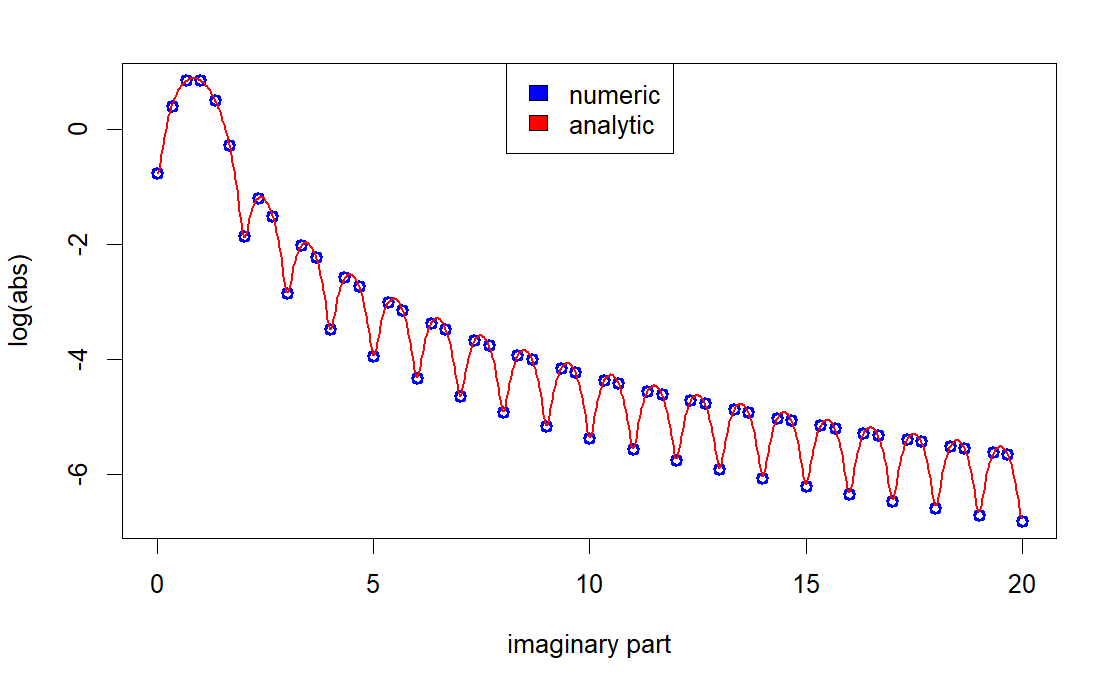}
  \caption{Compare the magnitudes of the numeric LT approximation (blue) with the corresponding analytic expression (red).}
  \label{fig:lt01}
\end{figure}

In Figure \ref{fig:km02} and \ref{fig:ailt}, we first apply the numeric LT to the function 
\begin{equation}
f(x) = 2\sin(x)+\cos(4x)+\sin(7x+0.5)+0.3(x-3)(x-5)+\epsilon(x).   
\end{equation}

The function $\epsilon$ represents random noise on the interval $[0,2\pi)$. The function $f$ is plotted as a red curve in Figure \ref{fig:ailt}, and its kime surface is shown in Figure \ref{fig:km02}.  The analytical ILT recovers the original function in the time domain (blue points). Clearly, the blue points closely match the red curve on the $[0,2\pi)$ interval, and naturally fall back to zero outside that range, as expected. In practice, computing the analytic ILT requires evaluation of the LT function at very high accuracy, and this figure confirms that both the analytic and numerical ILT methods work well.

\begin{figure}[H]
  \centering
  \includegraphics[width=0.6\textwidth]{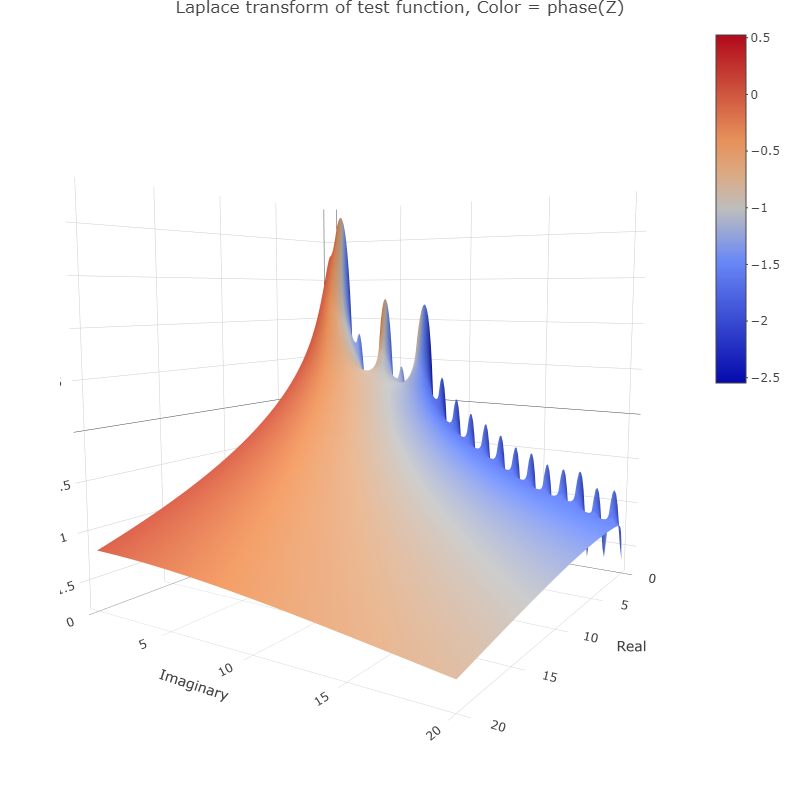}
  \caption{Laplace Transform of the test function $f$, logarithm scale}
  \label{fig:km02}
\end{figure}

\begin{figure}[H]
  \centering
  \includegraphics[width=0.8\textwidth]{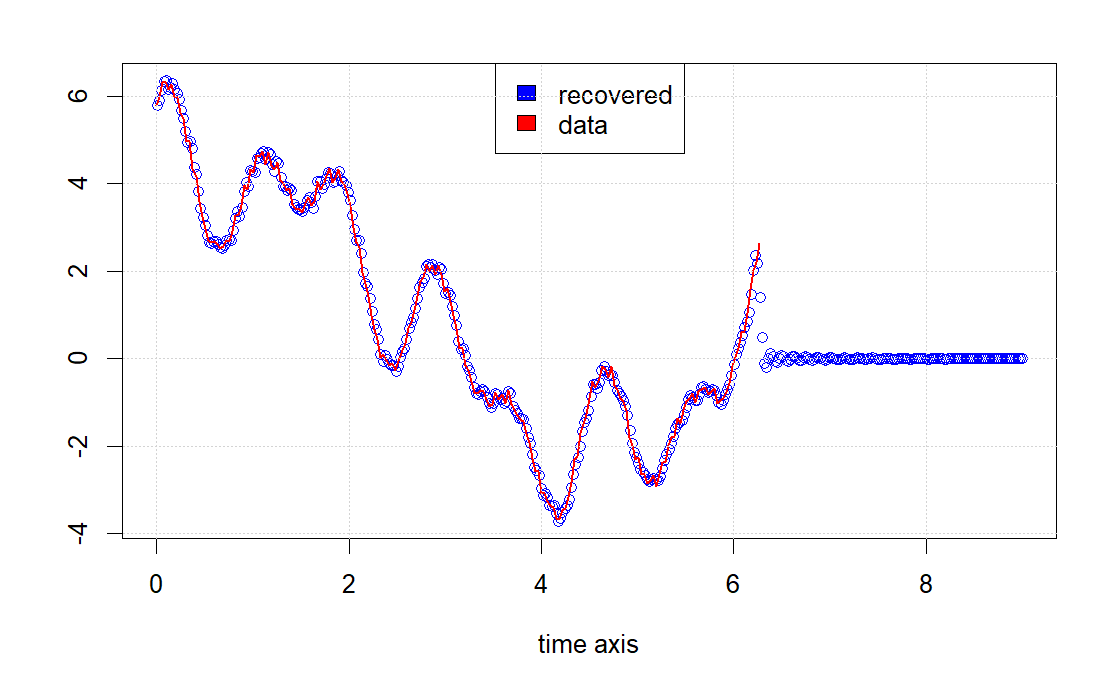}
  \caption{Plot of the time-domain function reconstruction using the analytic ILT method, $\hat{f}=\mathcal{L}^{-1}{\mathcal{L}(f)}$ (blue points), based on data points generated by the proposed LT algorithm compared to the original function, $f$, (red curve).}
  \label{fig:ailt}
\end{figure}

Visually, the numeric LT performs very well when the signal is sufficiently densely sampled. This is expected as in such cases, integration is often easy to implement and it adds another degree of smoothness. Since we compute the analytic integration result in each sub-interval, the domain point sampling strategy and the choice of a functional basis are critical. These choices may need to consider the signal frequency spectrum and choose appropriate domain partitioning scheme, function basis, and degrees of freedom.

In terms of computational complexity, when using a {\it{polynomial basis}}, we observe that the number of operations will increase rapidly as a function of the degrees of freedom, $n!$. This is because (analytic) integration of $t^ne^{-zt}$ needs to be decomposed into lower order terms. On the other hand, when using a {\it{trigonometric basis}}, the number of computational operations is less sensitive to the degree of freedom in each domain partitioning sub-interval, since in this case, integrating $\sin(nt)e^{-zt}$ is straightforward. 

\subsection{Validation of the discrete ILT}
In this section, we will compare the inverse Laplace transform estimates against their exact counterparts. We will apply the proposed discrete ILT method to data obtained by sampling from complex functions with known analytic ILT and also complex functions obtained above through LTs of specific longitudinal functions. 

We first apply the discrete ILT to the LT of $f(x)=\sin(3x)$ over the interval $(0,2\pi)$. In Figure \ref{fig:recs2}, the recovered and re-interpolated value of all $f^p$ (random domain partitioning scheme $p$) on each grid point is displayed as a box plot. We see that the inverse LT estimates tend to vary wildly. Mind the width of each vertical box, especially towards the end. This right side instability is due to the impact of the negative exponential in Laplace transform. 
However, the randomized algorithm estimates are much tighter on the left. The recovered estimate (green line) represents a {\it{weighted mean}} of all recovered function attempts, whereas the blue line estimate correspond to the {\it{median}} of the multiple recoveries at each grid point. 

Both ILT estimates (green and blue curves)  closely follow the true analytical ILT function (red curve). In particular, Figure \ref{fig:recs2} provides an example demonstrating the adaptability of the ensembled (mean or median based) randomized algorithm to recover the original data from a large number of {\it{imperfect}} individual attempts. 


\begin{figure}[H]
  \centering
  \includegraphics[width=0.9\textwidth]{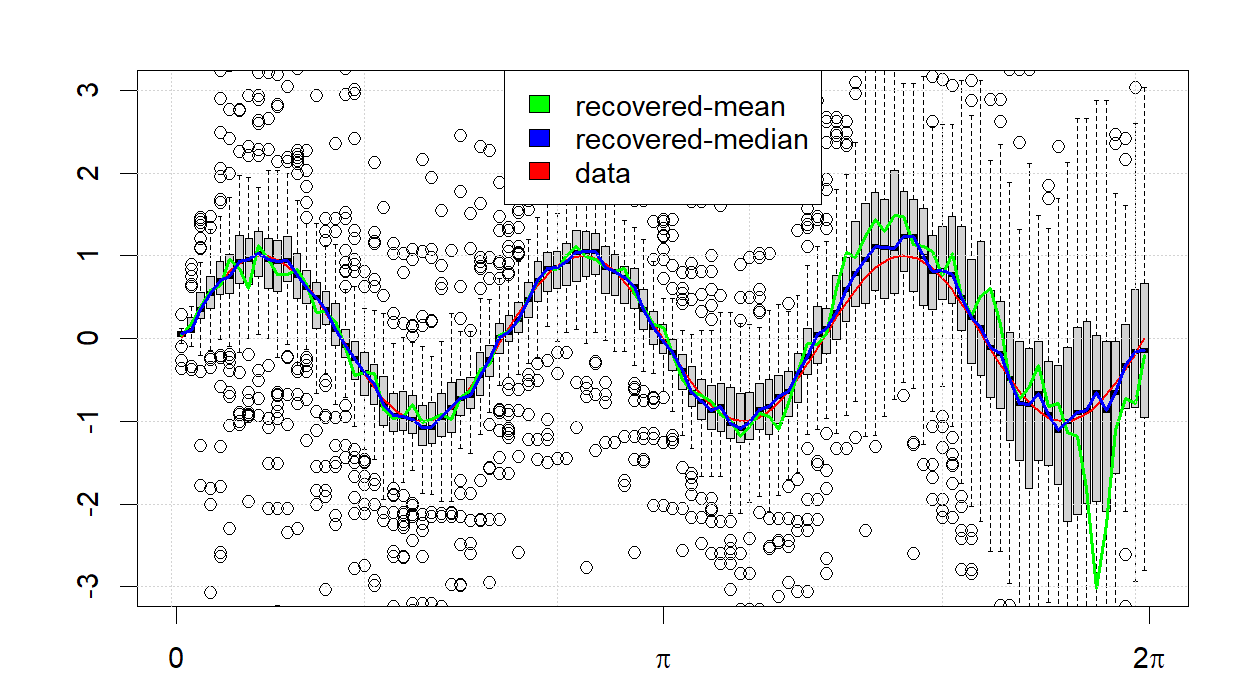}
  \caption{Box plot of recovered (interpolated) $f^p$ versus original $f$ (red).}
  \label{fig:recs2}
\end{figure}

Figures \ref{fig:rec} and \ref{fig:rec2} show the same experiment using the function $f(x) = 2\sin(x)+\cos(4x)+\sin(7x+0.5)+0.3(x-3)(x-5)+\epsilon(x)$ (see Figure \ref{fig:ailt}). Again, the results demonstrate that the discrete ILT approximation recovers the original function, see the resemblance between the reconstructions (green and blue lines) and the true signal (red curve). 

\begin{figure}[H]
  \centering
  \includegraphics[width=0.9\textwidth]{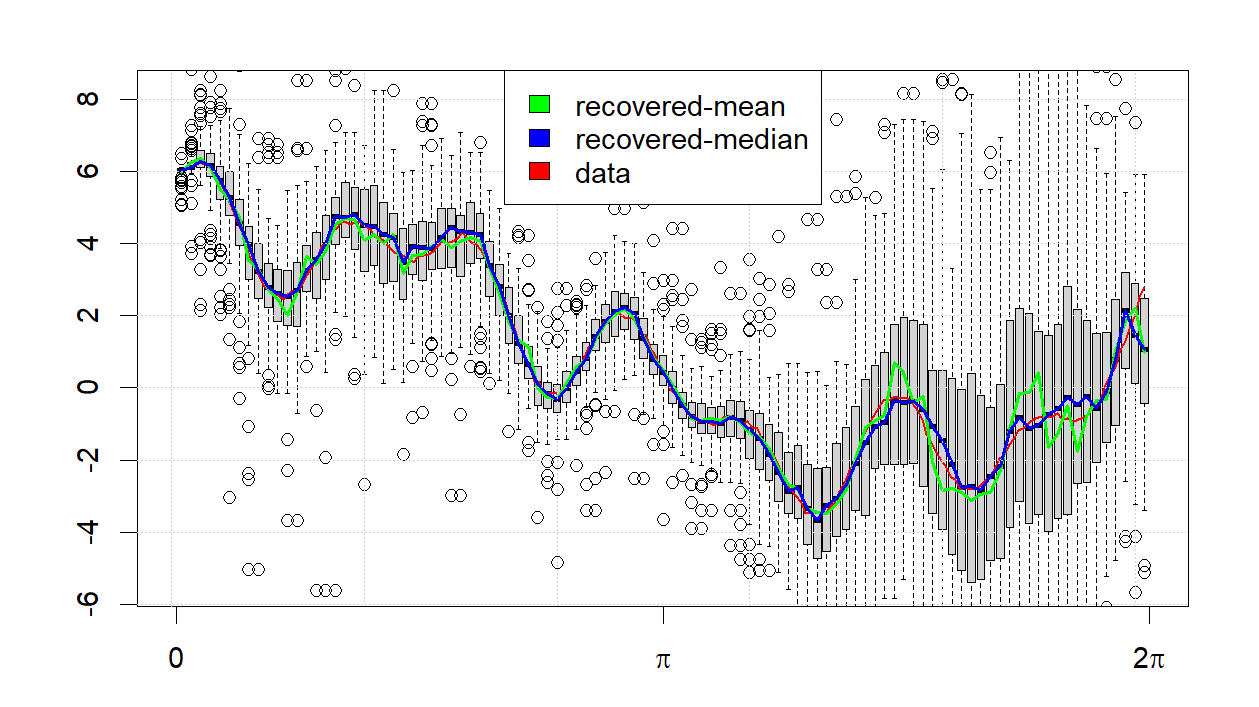}
  \caption{Box plot of recovered (interpolated) $f^p$ versus original $f$ (red). Horizontal and vertical axes represent the time and signal intensity, respectively.}
  \label{fig:rec}
\end{figure}

\begin{figure}[H]
  \centering
  \includegraphics[width=0.9\textwidth]{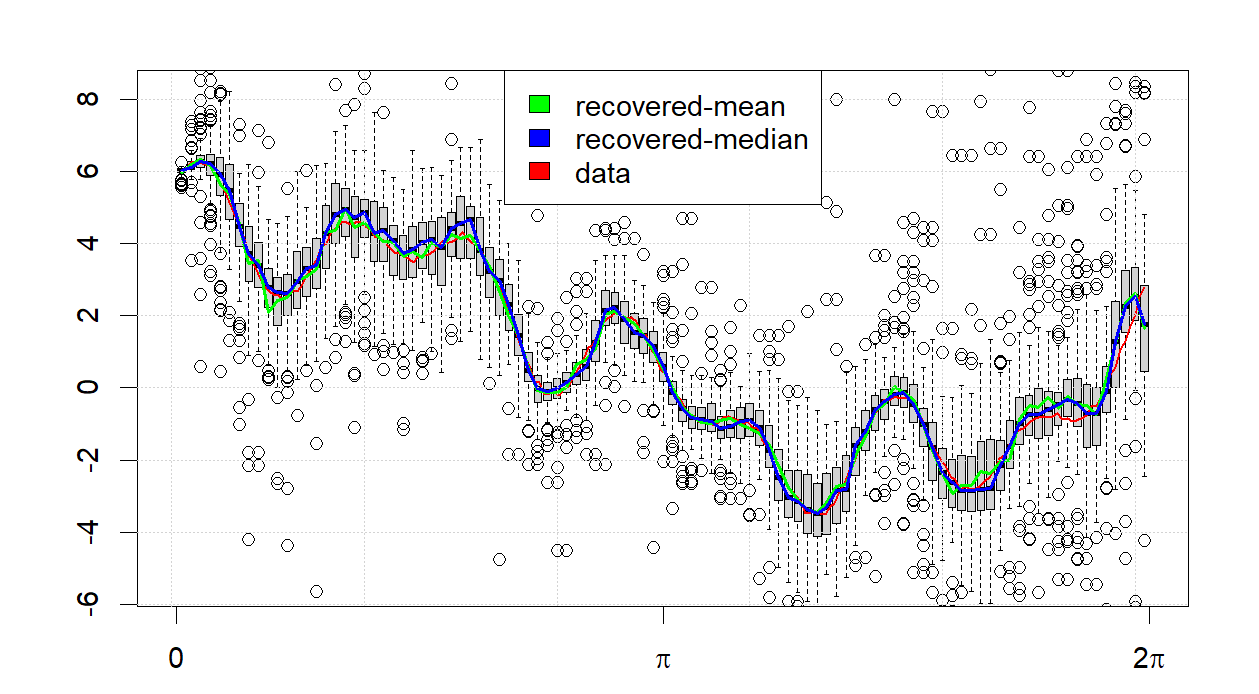}
  \caption{Another Box plot of recovered (interpolated) $f^p$ versus original $f$ (red).}
  \label{fig:rec2}
\end{figure}

Consider the original kime-surface shown in Figure \ref{fig:km02}. This surface is generated by applying LT to the result of the discrete ILT. In Figure \ref{fig:rec3}, we take a slice of the reconstructed kime-surface corresponding to $Re(z) = 0.5$ and compare it with the corresponding slice of the original signal. Since both signals are complex-valued, we plot their magnitudes. Even when the slice does not exactly correspond to the location of the sampling points, the recovered kime surface still preserves much of the underlying information stored in the original signal. Note that the discrepancy (error) rate mostly reflects differences between the locations of slicing points through the kime-surface. 


\begin{figure}[H]
    \centering
  \includegraphics[width=0.7\textwidth]{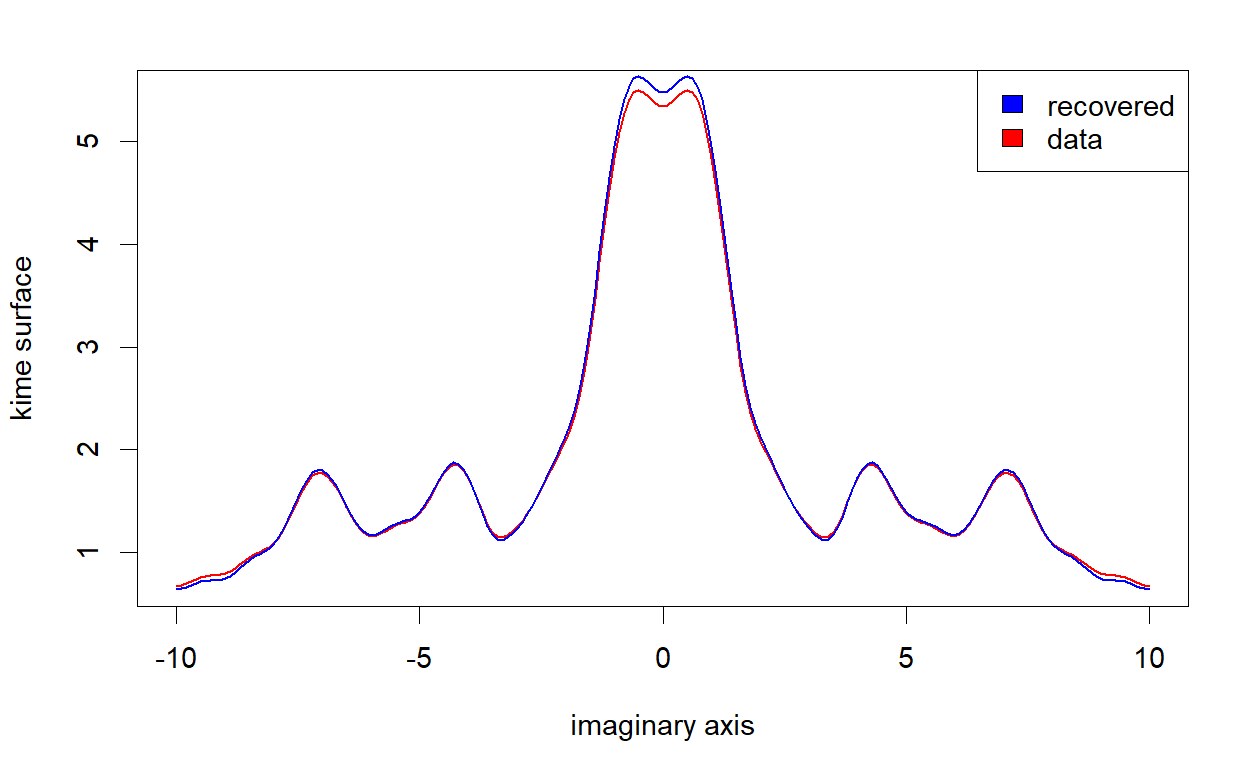}
  \caption{Magnitude of the reconstructed kime surface (at $Re(z) = 0.5$) compared to the original signal.}
  \label{fig:rec3}
\end{figure}

In Figure \ref{fig:recr}, we repeat $100$ times the experiment shown in Figure \ref{fig:rec} using different sets of sampling points. The graph shows the absolute and relative errors of the reconstructed kime surface (from recovered function, at the sample points) and recovered function. The two error rates are not strongly related, which confirms the intuition that large inverse LT errors are due to singularities of the linear system (of the discrete Laplace transform), rather than the numeric accuracy of solving the LT. 
Specifically, when matrix $\pmb{A}$ has SVD decomposition $A=\sum_k \sigma_k u_k v_k^{\dagger}$, the pseudo inverse is then $A^{\dagger}=\sum_k \frac{1}{\sigma_k} v_k u_k^{\dagger}$, and the smallest $\sigma_n(A)$ becomes the dominant singular values $\sigma_1(A^{\dagger})=\frac{1}{\sigma_n(A)}$.


\begin{figure}[H]
    \centering
  \includegraphics[width=0.7\textwidth]{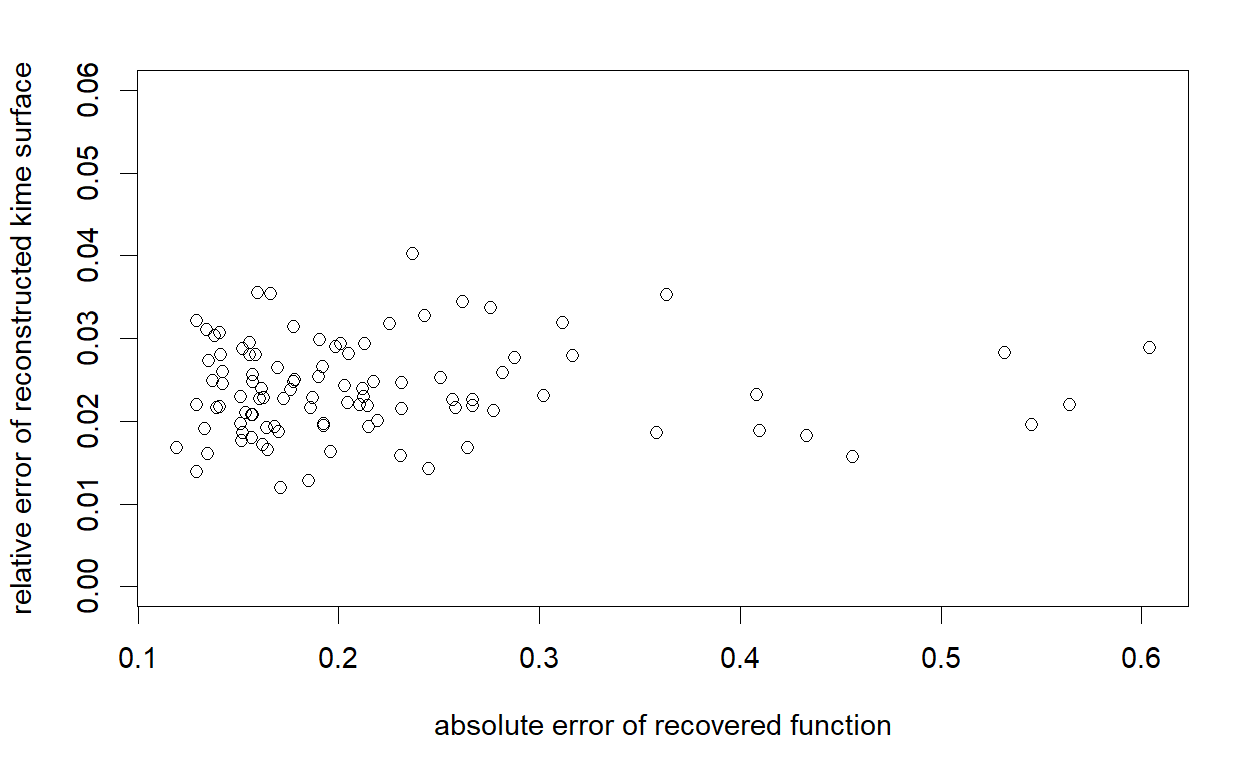}
  \caption{Error of recovered function and error of reconstructed kime surface using the recovered function.}
  \label{fig:recr}
\end{figure}

\subsection{Analysis of discrete ILT} 
In this section, we will provide an intuitive and empirical perspective on the efficacy of the randomized domain partitioning strategy which works well in practice, which will be supplemented with "approximate" scenarios that would work well in theory. Specifically, we show that the LT matrix derived from the preset model can be hard to invert; however, randomizing the domain partitioning significantly improves the process. We borrow ideas and estimation-strategies from random matrix theory. The preset model represents a partition function with fixed intervals whose partitioning values are to be determined. Indeed, there are other similar strategies that can be employed. 

The proposed discrete ILT technique is based on randomizing the partition, $p_i$, and the data points, $z_j$. Ideally we want to derive a model or explicate the joint probability distribution of the partition and point processes. However, theoretical error bounds for alternative strategies remain unknown. Hence, we present its weaker versions with relaxed constraints. 

\subsubsection{Theoretical Approximations and Results: Assumptions \& Analysis}


Our overarching goal is to bound the smallest singular values of rectangular, complex the randomized matrix with non-iid entries $A= \Big(\frac{1}{-z_i}(\exp(-z_ip_j)-\exp(-z_ip_{j-1}))\Big)_{1\leq i\leq n',1\leq j\leq n}$ via the randomized strategy (Algorithm \ref{alg:cap}) presented, which works pretty well in data experiments. Specifically, the two main challenges are (1) working with complex matrices (many real matrix results may be extended from real to complex \cite{vershynin2010introduction}); and (2) investigate the synergies between the newly proposed numerical methods and current state-of-the-art random matrix theory techniques, where the partition points $p_j$ are not necessary independent. This independence may yield non-trivial covariance matrices or violate isotropicity conditions \cite{vershynin2010introduction}.


\begin{definition}
    A random vector $X$ in $\mathbb{R}^n$ is isotropic if $\Sigma(X)=\mathbbm{E}[XX^T]=I_{n\times n} .$ Similarly, a random vector in $\mathbb{C}^n$ is isotropic if $\Sigma(X)=\mathbb{E}[XX^*]=I_{n\times n}$.
\end{definition}

\textbf{\underline{ Random Partitioning}}\\
Our surrogate random interval partitioning strategy is to first sample random variables $X_i$ from the same distribution $X$, $X_i\sim X$, and then perform normalization to fit a model on the $(0,2\pi)$ interval. The $i$th interval length is then $2\pi\frac{X_i}{\sum_{j=1}^{n+1} X_{j}}$, and the endpoints of the interval are $\left [2\pi\frac{\sum_{j=1}^{i-1}X_j}{\sum_{j=1}^{n+1} X_{j}},
2\pi\frac{\sum_{j=1}^iX_j}{\sum_{j=1}^{n+1} X_{j}} \right ]$ and $p_J=2\pi\frac{\sum_{j=1}^JX_j}{\sum_{j=1}^{n+1} X_{j}}$ ($1\leq J \leq n$). 

Note that the denominator sums cover  $n+1$ terms, as we need to throw out the last data point which deterministically produces $2\pi$. However, such partitions introduce correlations , which may impact the subsequent analyses using the random matrix results.

\begin{lemma}
    If $X_i >0$ and $X_i\overset{\text{i.i.d}}{\sim} X$, $U=\frac{\sum_{j=1}^{i}X_j}{\sum_{j=1}^{n+1} X_{j}}$ and $V=\frac{\sum_{j=1}^k X_j}{\sum_{j=1}^{n+1} X_{j}}$ can't be independent $\forall i<k .$
\label{non-indep}
\end{lemma}
\begin{proof}
    We argue that $f_{V\mid U}(v\mid U=u)=f_{V}(v),\forall u$ leads to a contradiction. Let's fix the value for $u$, then $f_{V\mid U}(v\mid U=u)=0, \forall 0<v<u$, but $f_{V}(v)>0,\forall 0<v<u$. Hence the independence condition $f(V\mid U)=f(V)$ is violated.
 
\end{proof}
\begin{corollary}
    $U=\frac{X_j}{\sum_{j=1}^{n+1} X_{j}}$ and $V=\frac{X_k}{\sum_{j=1}^{n+1} X_{j}}$ $\forall j\neq k$ aren't independent.
\end{corollary}
\begin{proof}
    If they are independent, then $f(U)=U,g(V)=1-V$ should be independent, but with a proper reordering of the indices we see that this can't be the case as it is demanding that  $Y=\frac{X_1}{\sum_{j=1}^{n+1} X_{j}}$ and $Z=\frac{\sum_{j=1}^{n} X_j}{\sum_{j=1}^{n+1} X_{j}}$ are independent, which is a special case for \cref{non-indep}.
\end{proof}
\textbf{\underline{Idealized scenario}}

To ensure independence and increase randomness to obtain theoretical lower bounds on singular values, one possible scenario is to break the interval $(0,2\pi)$ into $n$ parts, domain segments $p_j\in [\frac{2\pi}{n+1}j-\frac{\pi}{n+1},\frac{2\pi}{n+1}j+\frac{\pi}{n+1}], \forall j\in\{1,2,\cdots ,n\}$, with the partition points chosen independently in these domain segments. This effectively eradicates the prior information that $p_i$'s value incurred on $p_{i+1}$ in the random interval partition case, since the distribution of  $p_{i+1}$ will be shifted once $p_{i}$ is fixed. Another candidate partition strategy, guaranteeing symmetry of the interval $(0,2\pi)$, is $p_j\in (\frac{2\pi j}{n}-\frac{2\pi}{n},\frac{2\pi j}{n})$. See Supplementary material (\ref{randompartit}) for a visualization, derivation of the random partition distribution, and the contrast of this assumption with random partitioning (\ref{contrastcase}).

Let's consider the LT matrix $\mathbf{A} = (a_{ij})$ as a size $n'\times n$ matrix where $n'$ is the number of data points and $n$ is the number of partitions. Naturally $n'\ge n$, and $n' = g(n)$ for some specific function, $g(n)$. In practice, this can be a fixed aspect ratio like $n' = 2n$, or quadratic like $n' = n^2$. We want to determine the smallest singular value of this LT matrix.

\begin{theorem}
    (Non-asymptotic random matrix with heavy tails on independent rows \cite{vershynin2010introduction} p.26) Let $A$ be an $N\times n$ matrix whose rows $A_i$ are independent isotropic vectors. Let $m$ be a number such that $\|A_i\|_2\leq \sqrt{m}$ almost surely $\forall$ i. Then for $t\geq 0$ with probability at least $1-2n \exp(-ct^2)$,
    \begin{equation}
        \sqrt{N} -t\sqrt{m}\leq s_{\text{min}}(A),
    \end{equation}
    where $c>0$ is an absolute constant.
    \label{randmth}
\end{theorem}
The following theorem of Rudelson and Vershynin \cite{rudelson2009smallest} suggests that the smallest singular value of a Hermitian random matrix, with upper triangle i.i.d. elements and a distribution density satisfying a Gaussian bound with unit variance, is of the size $\sqrt{n'}-\sqrt{n}$, which is a more stringent condition compared to the version in Theorem \ref{randmth}. 


\begin{theorem}[Rudelson, Vershynin 2009 \cite{rudelson2009smallest}]\label{thm:mvt}
Let $A$ be an $n' \times n$ random matrix, $n' \geq n$, whose elements are independent copies of a mean zero sub-Gaussian random variable with unit variance. Then, for every $\varepsilon>0$, we have
\begin{equation}
\mathbb{P}\left(s_{\text{min}}(A)\equiv\sigma_{n}(A) \leq \varepsilon(\sqrt{n'}-\sqrt{n-1})\right) \leq(C \varepsilon)^{n'-n+1}+e^{-c n'},   
\end{equation}

where $\sigma_n$ is the smallest singular value; $C, c>0$ depend (polynomially) only on the sub-Gaussian moment $B$.

\end{theorem}

An idealized sampling strategy for the original matrix to guarantee isotropicity 
$$A_{ij}=-\frac{1}{z_i}\Big(\exp(-z_ip_{i,j})-\exp(-z_ip_{i,j-1})\Big)$$
is not available, however, in Appendix \ref{isotropicb} we demonstrate that isotropicity may be guaranteed for the case

$$B_{ij}=\Big(\exp(-z_ip_{i,j})-\exp(-z_ip_{i,j-1})\Big).$$

The following two lemmas present constraints for the \textbf{phase sampling} ensuring the matrix element isotropicity. A schematic for this is the following
\begin{figure}[H]
    \centering
    \includegraphics[width=16cm,height=7cm]{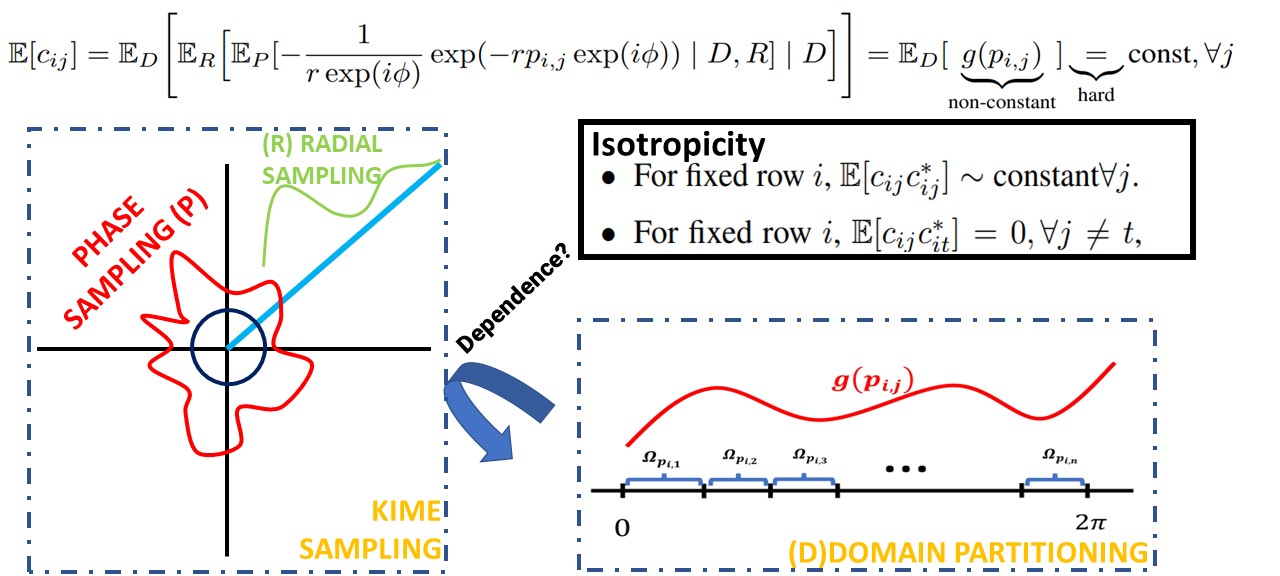}
    \caption{Illustration of the integral transforms (expectation) to get the isotropicity).}
    \label{fig:my_label}
\end{figure}


\begin{lemma}
($B_{ij}$ sampling) Given that $\phi$ and $a$ (the radial sampling, and the domain partitioning altogether is resembled by $a$) are independent, i.e., the sampling scheme, kime phase, and kime magnitude (radial) sampling processes are independent, then the only $f_{\Phi}(\phi)\in L^2([0,2\pi])$ such that  
\begin{equation}
    \int_0^{2\pi} \exp(-a\exp(i\phi))f_{\Phi}(\phi)d\phi = \text{const},
\end{equation}  
is $f_{\Phi}(\phi)=\frac{1}{2\pi}.$
\end{lemma}

\begin{lemma}
($A_{ij}$ sampling) If $a$ and $\phi$ are independent (meaning that the phase sampling is independent of the radial direction and the domain partitioning), then there exists no square-integrable function $f_{\Phi}(\phi)$, such that
\begin{equation}
    \int_0^{2\pi} \exp(a\exp(i\phi))\exp(-i\phi)f_{\Phi}(\phi)d\phi = \text{const}
\end{equation}
where the constant is independent of $a$.
\end{lemma}


These two lemmas effectively imply that independent phase-space sampling  would make the isotropic condition very difficult to guarantee. That is, the integration results in a (non-constant) function of $a$, which is highly unlikely to be equal in expectation after taking the expectation over the $[0,2\pi]$ partition, for example in $B_{ij}$ case,
\begin{equation}
    \mathbb{E}[c_{ij}]=\mathbb{E}_{D}\left [ \mathbb{E}_{R}\left [ \mathbb{E}_{P}\left [-\frac{1}{r\exp(i\phi)}\exp(-rp_{i,j}\exp(i\phi))\mid D,R\right ]\mid D\right ] \right ] =\mathbb{E}_{D}[\underbrace{g(p_{i,j})}_{\text{non-constant}}] \underbrace{=}_{\text{hard}} \text{const},  \forall j
\end{equation}

where $D,R,P$ stands for domain partitioning, radial expectation, and phase variable expectation respectively. The last equality means that solving the domain partition strategy $D$ on $n$ of $j's$ equations, and often impossible (for example when $g$ is monotonic).
Proofs and further discussion is given in Appendix \ref{limitationthm}

Furthermore, the entries of the LT matrix $\mathbf{A}$ are far from independent. Hence, there is little guarantee that direct applications of random matrix theory may ensure reasonable approximations of the smallest singular value of the LT matrix. However, there are several {\it{ad hoc}} strategies (such as random noise smoothing \cite{tao2010smooth} and regularization) where we may sacrifice reconstruction accuracy for algorithmic stability. Below, we propose a conjecture about the smallest singular value of the LT matrix.\\

\subsubsection{Experiments, hypothesis, and empirical evaluation}
\begin{proposition}
The smallest singular value $\sigma_n$ of the LT matrix $\mathbf{A}$ satisfies
\begin{equation}
\mathbb{E}(\sigma_n(\mathbf{A})) = h(n)    
\end{equation}

and there exists $\gamma\in(0,2)$ such that \begin{equation}
h(n)\sim \frac{1}{n^\gamma} .    
\end{equation}
\label{hypothesis}
\end{proposition}

Theorem \ref{thm:mvt} applies to random variables with unitary variance, whereas in our case, the variance is $O(1/n^2)$. Therefore, it's reasonable to expect that the magnitude of the singular value should be proportionately smaller. 
When we have a fixed aspect ratio, e.g., $n' = 2n$, then $\sqrt{n'} - \sqrt{n-1} \sim (\sqrt{2}-1)\sqrt{n}$. Dividing by $n$, this yields $\frac{\sqrt{2}-1}{\sqrt{n}}$, which corresponds to $\gamma = 0.5$. Whereas, if we have a square matrix $n' = n$ and $\sqrt{n} - \sqrt{n-1}\sim 1/\sqrt{n}$, the same logic yields $\gamma = 1.5$. 

Clearly, the $\gamma$ interval $(0,2)$ appears quite generous and this hypothesis appears to provide rational support for analyzing the performance error. In fact, the following error analysis doesn't really require the least singular value to be of polynomial decay, however, it provides an explicit and pragmatic estimation strategy.


Recall that we start with a function $f$ and
\begin{equation}
b_i = \mathcal{L}f (z_i)+\epsilon_i\ ,    
\end{equation}
where $b_i$ is the data on $z_i$, and $\epsilon_i$ is some sampling and rounding error, which also accounts for the rounding error of $z_i$. Next, we choose $n$ and $n'$ as before, let $p$ be a partition of the interval $(0,2\pi)$ of size $n$, and let $\mathbf{u}$ and $f^p$ be a piecewise constant interpolation of $f$ on $p$, where $\mathbf{u}$ is a vector and $f^p$ a function.
Then, we can derive an upper bound of the error
\begin{equation}
\begin{aligned}
    |\mathcal{L}(f)(z_i) - \mathbf{Au}(i)| &= \left |\int_0^{2\pi}\exp(-z_it)(f(t)-f^p(t))dt\right |\\
    & \le \sum_{i}(p_i-p_{i-1})\sup_{t\in(p_{i-1},p_i)}|f(t)-f^p(t)|\\
    & \sim \frac{2\pi}{n} \sup_t |f'(t)| = \frac{C_2}{n}\ ,
\end{aligned}    
\end{equation}
where $C_2$ is a positive constant depending on $\sup_t |f'(t)|$, which is finite. 
Let $\mathbf{\bar u} = \mathbf{A^{\dagger}b}$, then 
\begin{equation}
\begin{aligned}
|\mathbf{\bar u}-\mathbf{u}| &= |\mathbf{A^{\dagger}}(\mathcal{L}f(\mathbf{z})+\mathbf{\epsilon}-\mathbf{Au})|\\
&=|\mathbf{A^{\dagger}}(\mathcal{L}f(\mathbf{z})-\mathbf{Au})+\mathbf{A^{-1}}\mathbf{\epsilon}| \\
& \le |\mathbf{A^{\dagger}}(\mathcal{L}f(\mathbf{z})-\mathbf{Au})|+|\mathbf{A^{\dagger}}\mathbf{\epsilon}|\\
& \le \frac{|\mathcal{L}f(\mathbf{z})-\mathbf{Au}|+|\mathbf{\epsilon}|}{\sigma_n(\mathbf{A})}\ .
\end{aligned}
\end{equation}

Simple calculations examining this magnitude imply that
\begin{equation}
\mathbb{E}(\mathbf{\bar u}-\mathbf{u}) = 0\ .    
\end{equation}
Let's choose $n' = 2n$ and assume again that
\begin{equation}
\sigma_n(A)\sim \frac{1}{n^{\gamma}}.     
\end{equation}
Since the 2-norm of a vector will multiply by $\sqrt{n'}$, 
\begin{equation}
\begin{aligned}
\mathbb{E}(|\mathbf{\bar u}-\mathbf{u}|) &\lesssim \frac{\sqrt{n'}C_2/n}{\sigma_n} + \frac{\sqrt{n'\Var(\epsilon_i)}}{\sigma_n}\\
& \sim C_2n^{-1/2+\gamma}+C_3n^{1/2+\gamma}\ ,
\end{aligned}
\end{equation}
where $C_3$ is a constant depending on the sampling accuracy $\Var(\epsilon)$. 


Since we are also taking the 2-norm of the vector $\mathbf{u}$, we factor out another $n^{1/2}$ to compensate when changing to the $L^2$ norm. (Informally, the vector norm overcorrects the functional norm expectation by composing $n$ components, and hence a $\sqrt{n}$ correction factor is fitting.)

The error of the function $\bar{f^p}$, corresponding to $\mathbf{u}^p$ becomes
\begin{equation}
\mathbb{E}(||f-\bar{f^p}||)\sim C_2n^{-1+\gamma}+C_3n^{\gamma}\ . 
\label{previous_exp}
\end{equation}

In general, the above estimation shows that the minimal error will depend on the two constants $C_2$ and $C_3$, which control how oscillating the function $f$ is, and how representative of the native process the observed proxy data actually is. 

Since $\gamma>0$ implies that a larger matrix will be closer to singular, increasing the size $n$ will eventually increase the approximation error.

On the other hand, we see when $\gamma\in(0,1)$ the previous expression (\ref{previous_exp}) can be minimized at $n = \frac{(1-\gamma)C_2}{\gamma C_3}$. This suggests that solving the linear system is expected to yield good results. In real experiments, we can set an iteration counter, $\#itn$, and track the recovered function $\bar{f}$
\begin{equation}
\bar{f} = \frac{1}{\#itn}\sum \bar{f^p}\ .    
\end{equation}
Assuming some sufficient independence across each $\bar{f^p}$, the law of large numbers guarantees convergence when $\#itn\longrightarrow\infty$. For an optimal value of the hyperparameter $n$, it is expected that 
\begin{equation}
\Var(\bar{f}-f)\sim\frac{C_4}{\sqrt{\#itn}}\ ,    
\end{equation}

where $C_4$ is a constant depending on $C_2$ and $C_3$.

As a direct analytical proof of proposition \ref{hypothesis} is not yet available, the following experiment provides some intuition and empirical evidence in support of this conjecture. The experimental result on Figure \ref{fig:sing} shows a graph of the singular value of such matrices in log-scale across the range of the hyperparameter $n$.

The expected machine numeric error may be significant when $n$ is large and the least singular value is very close to $0$. Otherwise, the proposed $\frac{1}{n^\gamma}$ decay appears reasonable. Note that the vertical-axis is in a logarithmic scale. As long as the smallest singular value decays ($\gamma>0$, or in fact any kind of decay), the experiments suggest an eventual growth of the error rate when the matrix size increases.

\begin{figure}[H]
  \centering
  \includegraphics[width=0.8\textwidth]{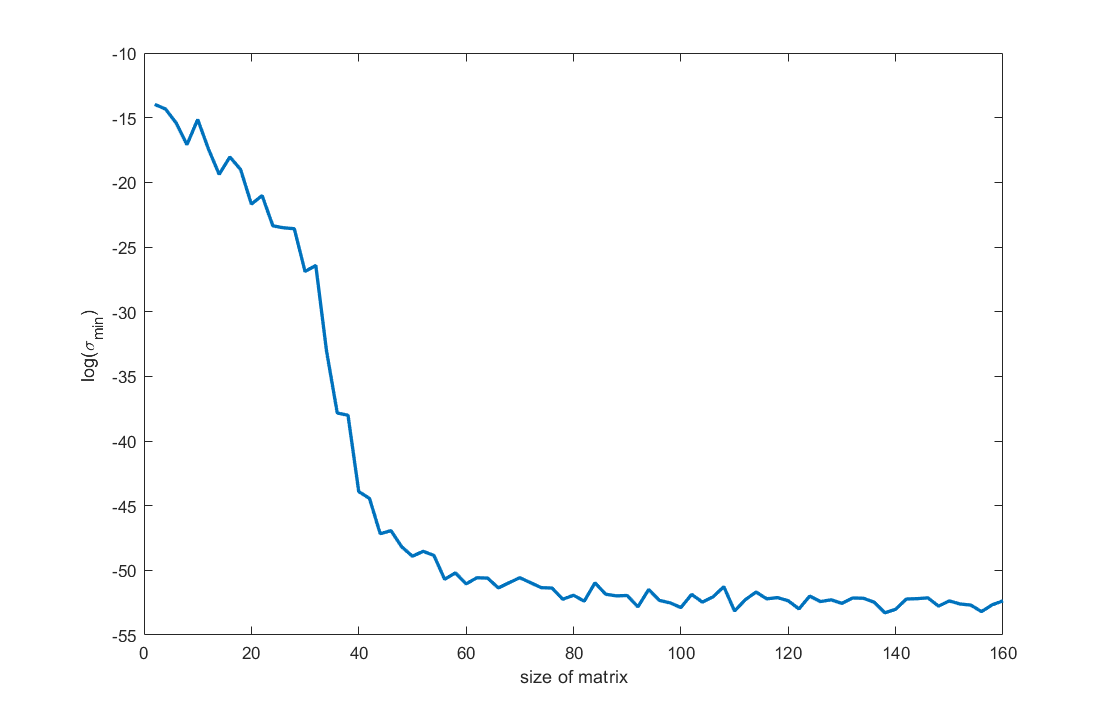}
  \caption{Least Singular Value of the $n \times 1.2n$ LT matrix.}
  \label{fig:sing}
\end{figure}

\section{Conclusion and Discussion}

This paper improves on some of the previous developments in numerical inversions of the Laplace transform. Specifically, we present new computational strategies for efficient and accurate reconstruction of longitudinal data (observed time-series) represented as complex-time surfaces via the Laplace transform. The study identified the following open questions that need to be investigated further.
\begin{itemize}
\item \textbf{Quantify the smallest singular value of the LT matrix}. The convergence of the proposed numerical algorithm depends on the constraints of the smallest singular value, specifically if it is sufficiently concentrated in distribution. In practice, this may be hard to experimentally verify. The choices of the parameters $n, n'$ also effect the algorithmic performance, however, their impact depends on optimality of the domain partitioning scheme.

\item \textbf{Independence \& convergence rate}. In practice, the use of the law of large numbers requires that the errors of each $\bar{f^p}$ are independent. Ultimately, we can't really expect that all errors are IID, as there is only a finite number of choices. However, having the sampled data points along with the corresponding errors does not provide sufficient information about the unknown ground truth process, i.e., the original function $f$. Hence, repeated application of the randomized algorithm multiple times allows us to compute and aggregate ensemble estimates, which may be pooled to obtain a more robust estimate of the original signal $f=\mathcal{L}^{-1}(F)$. Yet, it is possible to show that to a certain extent, the error can be reduced using weaker conditions than independence.
    
\item \textbf{Complexity}. The computational complexity of solving each matrix using the iterative generalized minimal residual method (GMRES) is $O(n')$ \cite{saad1986gmres}. It'd be useful to find better strategies to decide on the choices for optimal algorithmic hyperparameters, including the number of iterations to achieve a certain accuracy.
\end{itemize}

In this manuscript, we analyzed the asymptotic behavior and the error rate of estimating the discrete ILT using randomized domain partitioning and constant function basis. Analogously, similar approaches may be used to estimate the ILT in different function bases, e.g., polynomial or trigonometric function bases. More generally, for other applications where a matrix equation needs to be solved, assuming there is a way to combine the results from multiple attempts, such randomization approaches may also provide efficient and valuable approximation strategies. We also introduced the Laplace transform on groups and discussed a Clifford algebra extension for the Laplace transform.

As firm supporters of {\it{open-science}}, we shared all code, data, and results on the TCIU GitHub page (\url{https://github.com/SOCR/TCIU/}), via the CRAN R-package TCIU (\url{https://cran.r-project.org/web/packages/TCIU/}), and through the SOCR TCIU documentation site (\url{https://TCIU.predictive.space}).

\section{Acknowledgements}

This work was partially supported by NSF grants 1916425, 1734853, 1636840, 1416953, 0716055 and 1023115, NIH grants P20 NR015331, U54 EB020406, P50 NS091856, P30 DK089503, UL1 TR002240, R01 CA233487, R01 MH121079, R01 MH126137, and T32 GM141746. The funding organizations played no role in the study design, data collection and analysis, decision to publish, or preparation of the manuscript. Many colleagues at the University of Michigan Statistics Online Computational Resource (SOCR) and the Michigan Institute for Data Science (MIDAS) contributed ideas, infrastructure, and support for the project.

\newpage

\bibliographystyle{ieeetr}
\bibliography{citations.bib}
\newpage
\section{Appendix}
\subsection{Appendix A \label{sec:Appendix1}}
\textbf{Series expansion for $f(z) = \frac{1}{\text{cosh}z}$}: We apply a Wick rotation $z\to iz$ and then then the formulation is in a more amenable form. Namely, $\overline{f}(z)=\frac{1}{\text{cos}z}$, which has simple poles $z_0$, i.e., $(z-z_0)\overline{f}(z)$ is holomorphic over the neighborhood of $z_0$, where $z_0\in S=\{1/2\pi+k\pi:k\in\mathbb{Z}\}$.  Therefore the full expansion can be written as
\begin{equation}
    \overline{f}(z)=\sum_{z_0\in S}\text{res}_{-1,z_0}(z-z_0)^{-1}+G(z)
\end{equation}
where $G(z)$ is a holomorphic function. The residues by definition and apply L'H\^{o}spital's rule
\begin{equation}
    \text{res}_{-1,z_0} = \lim_{z\to z_0}(z-z_0)\overline{f}(z)=\frac{1}{\text{sin}(z_0)}=(-1)^k .
\end{equation}
We can aggregate $z_0$ and $-z_0$. Therefore,
\begin{equation}
    {res}_{-1,z_0}(z-z_0)^{-1}+{res}_{-1,-z_0}(z+z_0)^{-1} = (-1)^k\frac{1}{z-z_0} +(-1)^{-k-1}\frac{1}{z+z_0}=\frac{(-1)^k(-2z_0)}{z^2-z_0^2} .
\end{equation}
Therefore, the series expansion is  
\begin{equation}
    \overline{f}(z) = 2 \pi \sum_{n=0}^{\infty} \frac{(-1)^{n}(n+1 / 2)}{(n+1 / 2)^{2} \pi^{2}-z^{2}} .
\end{equation}
We arrive at the desired result \ref{cosh:eq} by undoing the Wick rotation. The performance of the 2,5,7,11 first terms truncation residuals are visualized in the following.
\begin{figure}[H]
    \centering
    \includegraphics[width=16cm,height=6cm]{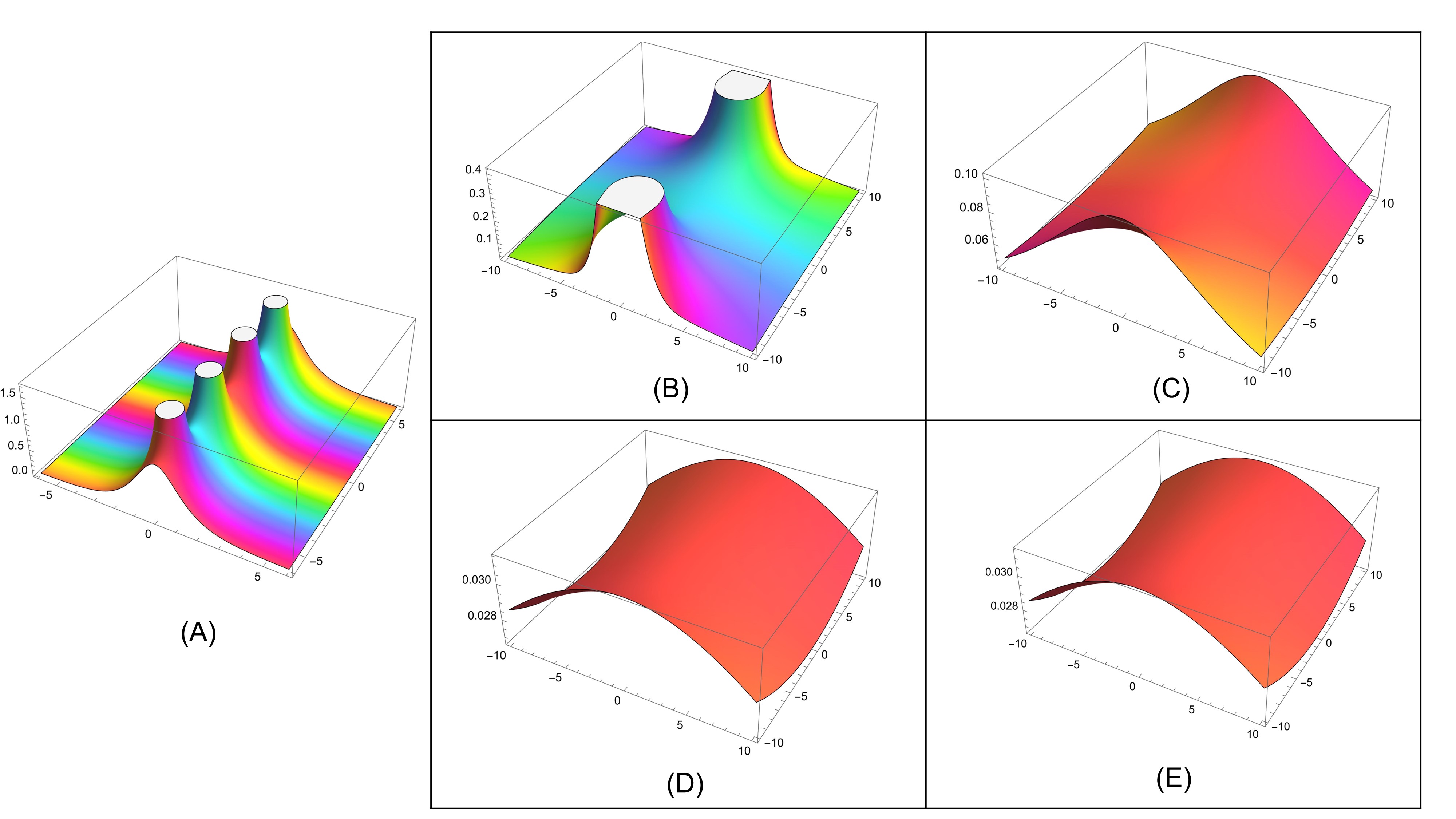}
    \caption{Panel $(A)$ indicates the $\frac{1}{\cosh z}$ function. Panels $(B),(C),(D)$, and $(E)$ correspond to the truncated series expansion at $n=1,4,6,10$, where the $y$ axis is the magnitude of error. The error decreases with the inclusion of more terms, e.g., the inclusion of $10$ terms gives around $0.03$ approximation error.}
    \label{fig:cosh}
\end{figure}

\textbf{Euler Transformation}\\

\underline{Series Summation:} Let $S=\sum_{n=i}^{\infty}\frac{ \binom{n}{i}}{2^n}$, then $S=2$: Multiply by $2$ both sides of the equation. Then, $2S=\sum_{n=i}^{\infty}\frac{ \binom{n}{i}}{2^{n-1}}=\sum_{n=i-1}^{\infty}\frac{ \binom{n+1}{i}}{2^n}$. Subtracting with the original we get $S=\sum_{n=i-1}^{\infty}\frac{ \binom{n}{i-1}}{2^n}$. Apply this $i$ times, we get the familiar result, which is $S=\sum_{n=0}^{\infty} \frac{1}{2^n}=2$. \\

\subsection{Appendix B Boundedness for Laplace transform}
\paragraph{Classical Laplace transform is a bounded linear operator} The linearity of LT is seen via $aF(z)+bF(z)=\mathcal{L}(af(t)+bf(t))$
To yield sufficient decay and integrability condition, the domain for $\hat{f}$ is often taken to be the right half plane $\mathbb{C}_{+}=\{z\in\mathbb{C};Re(z)>0\}$, which is one scenario where Hardy space $H^2$ is commonly defined. We require that $\|\mathcal{L}f\|_{H^2(\mathbb{C}_{+})} \leq C\cdot \|f\|_{L^2([0,\infty))} \text{ for some constant } C$.
\begin{equation*}
    F(z\equiv x+iy)=\mathcal{L}f(z) = \int_{0}^{\infty} f(t)e^{-zt} dt= \int_{0}^{\infty} f(t)e^{-tx}e^{-2\pi iyt} dt = \mathcal{F}\Big(f(t)e^{-tx}H(t) \Big)(y) .
\end{equation*}

We establish $f\in L^2([0,\infty)) $ then $\mathcal{L}f\in H^2(\mathbb{C}_{+})$. We need to show that $\mathcal{L}f$ is holomorphic and the norm  \cite{elliott2012composition} is bounded $\| F\|_{H^2(\mathbb{C}_{+})} = \sup_{x>0} \sqrt{\int_{-\infty}^{\infty} |F(x+iy)|^2dy} < \infty $. The holomorphicity can be argued by for any closed curve $\gamma$ defined in $\mathbb{C}_{+}$ : $\oint_{\gamma} F(Z)dz=\oint_{\gamma}\int_{0}^{\infty}\underbrace{e^{-tz}f(t)}_{\text{holomorphic in } s}dtdz=0$. The boundedness is established by 
\begin{equation*}
    \sup_{x>0}\sqrt{\int_{-\infty}^{\infty} |F(x+iy)|^2dy}
=\sup_{x>0}\sqrt{\int_{0}^{\infty} |f(t)|^2e^{-2xt} dt}=\sqrt{\int_{0}^{\infty} |f(t)|^2 dt}<\infty \quad \forall f\in L^2([0,\infty)),
\end{equation*}
where we used Plancherel theorem,
\begin{equation*}
    \int_{-\infty}^{\infty} |f(t)e^{-tx}H(t)|^2 dt= \int_{-\infty}^{\infty} |F(x+iy)|^2dy 
\end{equation*}
Therefore $\|\mathcal{L}f\|_{H^2(\mathbb{C}_{+})}= \|f\|_{L^2([0,\infty))}$.
\subsection{Appendix C. Laplace transform from two transform rules \label{two-rules}}
\begin{lemma}
If a function $f$ is analytic, and we require two properties
\begin{itemize}
    \item Linearity: $\mathcal{L}[\alpha f+\beta g]=\alpha \mathcal{L}[f]+\beta \mathcal{L}[g]$
    \item Algebraic Derivatives: $\mathcal{L}[f'](s) = s\mathcal{L}[f](s)-f(0)$
\end{itemize}
then Laplace transform $\mathcal{L}[f](s)=\int_0^{\infty} e^{-st}f(t)dt$ is the \textbf{only} transform that satisfy these properties. If we consider signals with finite support, and WLOG consider the $[0,2\pi]$ domain. If we switch the algebraic derivative condition to
\begin{itemize}
    \item Algebraic Derivatives (finite): $\mathcal{L}[f'](s) = s\mathcal{L}[f](s)-f(0)+e^{-2\pi s }f(2\pi)$.
\end{itemize}

Then, linearity and finite Algebraic Derivatives are characteristic of the transform  $\mathcal{L}[f](s)=\int_0^{2\pi} e^{-st}f(t)dt$.
\end{lemma}

\begin{proof}
    We prove the finite case and the infinite case follows by extension. Consider $f=0$. The the linearity condition implies that $\mathcal{L}[0]=0$, we then consider $f=1$, then the algebraic derivative condition implies $s\mathcal{L}[1](s)-1+e^{-2\pi s}=\mathcal{L}[0](s)=0$. Then $\mathcal{L}[1](s)=\frac{1-e^{-2\pi s}}{s}$. Then consider the function $f(t)=\frac{t^n}{n!}$, the result can be calculated recursively using the algebraic derivative condition
    \begin{subequations}
        \begin{equation}
            s\mathcal{L}\Big[\frac{t^n}{n!}\Big](s)+e^{-2\pi s}\frac{(2\pi)^n}{n!} = \mathcal{L}\Big[\frac{t^{n-1}}{(n-1)!}\Big](s) ,
        \end{equation}
        \begin{equation}
            \vdots
        \end{equation}
        \begin{equation}
            s\mathcal{L}[t](s)+e^{-2\pi s}(2\pi) = \mathcal{L}[1](s)=\frac{1-e^{-2\pi s}}{s} .
        \end{equation}
    \end{subequations}
    Therefore, 
    \begin{equation}
        s^{n+1}\mathcal{L}\Big[\frac{t^n}{n!}\Big](s) = (1-e^{-2\pi s})-e^{-2\pi s}\sum_{n=1}^n \frac{(2\pi s)^n}{n!}=1-\sum_{n=0}^n \frac{(2\pi s)^n}{n!}=\frac{1}{n!}\gamma(n+1,2\pi s) ,
        \label{incom-gam}
    \end{equation}
    where $\gamma$ is the lower incomplete gamma function and $\gamma(n+1,2\pi s) = \int_{0}^{2\pi s}t^ne^{-t}dt$. By analyticity, we can expand $f$ as an infinite series,
    \begin{equation}
        \begin{split}
            \mathcal{L}[f](s) &= \mathcal{L}\Big[\sum_{n=0}^{\infty}\frac{f^{(n)}(0)t^n}{n!}\Big](s) = \sum_{n=0}^{\infty}f^{(n)}(0)\mathcal{L}\Big[\frac{t^n}{n!}\Big](s)  \quad \text{(Linearity)}\\&=\sum_{n=0}^{\infty}f^{(n)}(0)\int_{0}^{2\pi s}\frac{1}{n!s^{n+1}}t^ne^{-t}dt=\frac{1}{s}\sum_{n=0}^{\infty}f^{(n)}(0)\int_{0}^{2\pi s}\frac{1}{n!}(\frac{t}{s})^ne^{-t}dt \quad \text{(Equation } \ref{incom-gam} \text{)}\\
            &=\frac{1}{s}\int_{0}^{2\pi s}\sum_{n=0}^{\infty}f^{(n)}(0)\frac{(t/s)^n}{n!}e^{-t}dt = \frac{1}{s}\int_0^{2\pi} f(t/s)e^{-t}dt = \int_0^{2\pi}f(t)e^{-st}dt .
        \end{split}
    \end{equation}
\end{proof}

\subsection{Appendix D. Meijer-G special functions and Laplace transform proofs\label{meijerg}}
\begin{theorem}
(Saxena \cite{mathai2006generalized} p.80) The most general form of the Laplace transform is 

\begin{equation}
\begin{split}
&\int_{0}^{\infty}x^{\sigma -1} G_{p,q}^{m,n}\left(wx\bigg|\begin{array}{c}
a_p\\
b_q\end{array}\right) G_{\gamma,\delta}^{\alpha,\beta}\left(\eta x^{k/\rho}\bigg|\begin{array}{c}
c_{\gamma}\\
d_{\delta}\end{array}\right)dx\\&=C_1\cdot G_{\rho\gamma+kq,\rho\delta+k\rho}^{\rho\alpha+kn,\rho\beta+km}
\Bigg(\frac{\eta^{\rho}\rho^{\rho(\gamma-\delta)}}{w^kk^{k(p-q)}}\bigg|\begin{array}{c}
\overbrace{\Delta(\rho,c_1),...,\Delta(\rho,c_{\beta}),\Delta(k,1-b_1-\sigma),...\Delta(k,1-b_m-\sigma)}^{\beta\rho+km};\\
\underbrace{\Delta(\rho,d_1),...,\Delta(\rho,d_{\alpha}),\Delta(k,1-a_1-\sigma),...\Delta(k,1-a_n-\sigma)}_{\rho\alpha+kn};\end{array}\\
&\begin{array}{c}
\overbrace{\Delta(k,1-b_{m+1}-\sigma),...,\Delta(k,1-b_{q}-\sigma),\Delta(\rho,c_{\beta+1}),...,\Delta(\rho,c_{\gamma})}^{\rho(\gamma-\beta)+k(q-m)}\\
\underbrace{\Delta(k,1-a_{n+1}-\sigma),...,\Delta(k,1-a_{p}-\sigma),\Delta(\rho,c_{\alpha+1}),...,\Delta(\rho,c_{\delta})}_{\rho(\delta-\alpha)+k(p-n)}\end{array}\Bigg) ,
\end{split}
\end{equation}
where $C_1$ is some scalar factor and the $\Delta(\cdot,\cdot)$ is a short hand notation for a set of indexing argument in the Meijer-G function. For example, $\Delta(\rho, c_j)=\{\frac{c_j+1}{\rho},\frac{c_j+2}{\rho},...,\frac{c_j+\rho-1}{\rho}\}$, and the indices by default are non-negative integers.
\label{theoremgeneral}
\end{theorem}

The Gauss's multiplication lemma comes in handy in transforming the scaling on the pole information for the Meijer-G functions, and relating Laplace transform with the formula isn't too far-fetched as the definition for Meijer-G consists of gamma function products.

\begin{lemma}
The Gauss's multiplication formula is
\begin{equation}
  \Gamma(kz) = (2\pi)^{1/2-k/2} k^{kz-1/2}\prod_{j=0}^{k-1}\Gamma\left (z+\frac{j}{k}\right )\ , \quad \forall k=1,2,3,4,\cdots \ .
\end{equation} 
In the $k=2$ case, this degenerates to the Legendre duplication formula.
\end{lemma}

\begin{proof}
One proof is to expand the gamma function elements using the recurrence relation and applying the Euler form for gamma function and utilizing the Stirling's formula . That is

\begin{equation}
    \Gamma\Big(z+\frac{k}{n}\Big) = \lim_{m\to \infty} \frac{\sqrt{2\pi}(\frac{mn}{e})^mm^{z+k/n-1/2}}{(nz+k)(nz+k+n)\cdots (nz+k-n+mn)}
\end{equation}
And then aggregating the elements and since the limit exits the product of the limit also exists. We can aggregate and get the desired result by inverting the procedure of deriving limit form from gamma functions.

Another proof is to use the Euler reflection formula  $\Gamma(z)\Gamma(1-z)=\frac{\pi}{\sin(\pi z)}, \forall z\not\in \mathbb{Z}$, and the periodicity in the product form to establish sine identities using the roots for the $z^n-1=\prod_{i=0}^{n-1}(z-e^{2\pi ik/n})$ to establish a sine identity $\prod_{k=1}^{n-1}\sin(k\pi/n)=n2^{1-n}$ to use with the reflection formula to get the desired result.

\end{proof}

In similar fashion, an easier calculation facilitates the fractional scaling on the variable $z$ (i.e., $z^{\frac{1}{k}}$).

\begin{theorem}
The scaling formula on the variable $z$ is
\begin{equation}
\begin{split}
   & G_{p,q}^{m,n}\left(z\bigg|\begin{array}{c}
a_p\\
b_q\end{array}\right) =\frac{h^{1+\nu+(p-q)/2}}{(2\pi)^{(h-1)\delta}}G_{hp,hq}^{hm,hn}
\Bigg(\frac{z^h}{h^{h(q-p)}}\bigg|\begin{array}{c}
\overbrace{a_1/h,\cdot\cdot\cdot,(a_1+h-1)/h}^{\Delta(h,a_1)},\cdot\cdot\cdot,\overbrace{a_n/h,\cdot\cdot\cdot,(a_n+h-1)/h}^{\Delta(h,a_n)};\\
\underbrace{b_1/h,\cdot\cdot\cdot,(b_1+h-1)/h}_{\Delta(h,b_1)},\cdot\cdot\cdot,\underbrace{b_m/h,\cdot\cdot\cdot,(b_m+h-1)/h}_{\Delta(h,b_m)};\end{array}\\
\\
&\begin{array}{c}
\overbrace{a_{n+1}/h,\cdot\cdot\cdot,(a_{n+1}+h-1)/h}^{\Delta(h,a_{n+1})},\cdot\cdot\cdot,\overbrace{a_n/h,\cdot\cdot\cdot,(a_n+h-1)/h}^{\Delta(h,a_n)}\\
\underbrace{b_{m+1}/h,\cdot\cdot\cdot,(b_{m+1}+h-1)/h}_{\Delta(h,b_{m+1})},\cdot\cdot\cdot,\underbrace{b_q/h,\cdot\cdot\cdot,(b_q+h-1)/h}_{\Delta(h,b_q)}\end{array}\Bigg), h=1,2,3,\cdots\ .
\end{split}
\end{equation}
\end{theorem}
\begin{proof}
For the sake of disposition, we derive from the RHS.
Substituting $z\to \frac{x}{k}$ in the previous lemma with 
\begin{equation}
  \Gamma\Big(k\frac{x}{k}\Big) = (2\pi)^{1/2-k/2} k^{kz-1/2}\prod_{j=0}^{k-1}\Gamma\Big(\frac{x}{k}+\frac{j}{k}\Big), \quad \forall k=1,2,3,4,\cdots .
  \label{gaussmult}
\end{equation} 
Now we write out the inverse transform of the Meijer-G function via Mellin transform explicitly, i.e.,
\begin{equation}
\begin{split}
    RHS &= \frac{1}{2\pi i}\int_{\mathcal{C}} \frac{\prod_{j=1}^m\prod_{t=0}^{h-1}\Gamma(\frac{b_j+t}{h}+s)\prod_{j=1}^n\prod_{t=0}^{h-1}\Gamma(1-\frac{a_j+t}{h}-s)}{\prod_{j=n+1}^p\prod_{t=0}^{h-1}\Gamma(\frac{a_j+t}{h}+s)\prod_{j=m+1}^q\prod_{t=0}^{h-1}\Gamma(1-\frac{b_j+t}{h}-s)}\left(\frac{z^h}{h^{h(q-p)}}\right)^{-s}ds \\&= \bigg\{\begin{array}{c}
         hs=s' \\
         ds=\frac{1}{h}ds' 
    \end{array}\bigg\} \\
        &=\frac{1}{2\pi i} \int_{\mathcal{C}} \frac{\overbrace{\prod_{j=1}^m\prod_{t=0}^{h-1}\Gamma(\frac{b_j+t}{h}+\frac{s'}{h})}^{A}\overbrace{\prod_{j=1}^n\prod_{t=0}^{h-1}\Gamma(1-\frac{a_j+t}{h}-\frac{s'}{h})}^{B}}{\underbrace{\prod_{j=n+1}^p\prod_{t=0}^{h-1}\Gamma(\frac{a_j+t}{h}+\frac{s'}{h})}_{C}\underbrace{\prod_{j=m+1}^q\prod_{t=0}^{h-1}\Gamma(1-\frac{b_j+t}{h}-\frac{s'}{h})}_{D}}\left(z^{-s'}h^{s'(q-p)}\right)\frac{1}{h}ds' \ .
\end{split}
\label{trhs}
\end{equation}
We examine the $A, B, C, D$ components more closely applying the Gauss multiplication formula (Equation \ref{gaussmult}). By insisting $k=h, b_j+s'=x$, the $A$ component is 
\begin{equation}
    A = \underbrace{\Big(\prod_{j=1}^m\Gamma(b_j+s')\Big)}_{(1)}\underbrace{\Bigg[(2\pi)^{\frac{h-1}{2}}\bigg]^{m}}_{(2)}\underbrace{\bigg[h^{\frac{1}{2}-(b_j+s')}\bigg]^{m}}_{(3)}.
\end{equation}
Similarly, the $B,C,D$ components are (For the $B,D$ components, rearranging is required to get the original Meijer-G formulation. For example, $\prod_{t=0}^{h-1}\Gamma(1-\frac{a_j+t}{h}-\frac{s'}{h})=\prod_{t=0}^{h-1}\Gamma(\frac{1}{h}+\frac{t}{h}-\frac{a_j}{h}-\frac{s'}{h})$, and insisting $x=1-a_j-s'$)
\begin{equation}
    B=\Big(\prod_{j=1}^n\Gamma(1-a_j-s')\Big)\Bigg[(2\pi)^{\frac{h-1}{2}}\bigg]^{n}\bigg[h^{\frac{1}{2}-(1-a_j-s')}\bigg]^{n},
\end{equation}
\begin{equation}
    C=\Big(\prod_{j=n+1}^p\Gamma(a_j+s')\Big)\Bigg[(2\pi)^{\frac{h-1}{2}}\bigg]^{p-n}\bigg[h^{\frac{1}{2}-(a_j+s')}\bigg]^{p-n}, \text{ and}
\end{equation}
\begin{equation}
    D=\Big(\prod_{j=m+1}^q\Gamma(1-b_j-s')\Big)\Bigg[(2\pi)^{\frac{h-1}{2}}\bigg]^{q-m}\bigg[h^{\frac{1}{2}-(1-b_j-s')}\bigg]^{q-m}.
\end{equation}
The first components (1) simply give rise to the usual Meijer-G components, the second factor term is (2) for $A,B,C,D$ combined is 
$$(2)=\Big[(2\pi)^{(h-1)/2}\Big]^{m+n-(p-n)-(q-m)}=\Big[(2\pi)^{(h-1)/2}\Big]^{2m+2n-p-q}=\Big[(2\pi)^{(h-1)/2}\Big]^{2\delta},2\delta = 2m+2n-p-q.$$ 
The  final term is
\begin{equation}
\begin{split}
    (3) &= \frac{\prod_{j=1}^m h^{1/2-(b_j+s')}\prod_{j=1}^n h^{1/2-(1-a_j-s')}}{\prod_{j=n+1}^p h^{1/2-(a_j+s')}\prod_{j=m+1}^{q}h^{1/2-(1-b_j-s')}}\\&=h^{\frac{1}{2}(m-n-(p-n)+(q-m))}h^{\sum_{j=1}^pa_j-\sum_{j=1}^qb_j}h^{s'(-m+n+(p-n)-(q-m))}=h^{\frac{1}{2}(q-p)}h^{\nu}h^{s'(p-q)},
\end{split}
\end{equation}
where $\nu = \sum_{j=1}^pa_j-\sum_{j=1}^qb_j$
Plug back in (equation \ref{trhs})
\begin{equation}
\begin{split}
     RHS&=(\frac{1}{2\pi i})\int_{C}\frac{\prod_{j=1}^m\Gamma(b_j-s')\prod_{j=1}^n\Gamma(1-a_j+s')}{\prod_{j=m+1}^q\Gamma(1-b_{j}+s')\prod_{j=n+1}^q\Gamma(a_j-s')}(2\pi)^{(h-1)\delta }h^{\nu-1+\frac{1}{2}(q-p)}z^{-s'}ds'\\
     &=(2\pi)^{(h-1)\delta }h^{\nu-1+\frac{1}{2}(q-p)}G_{p,q}^{m,n}\left(z\bigg|\begin{array}{c}
a_p\\
b_q\end{array}\right)
\end{split}
\end{equation}
\end{proof}

Analogously, we can prove the a general form of Meijer-G function in Theorem \ref{theoremgeneral}, which demonstrates the closeness of Meijer function under integral transforms. Before that, we shortly show the calculation for the laplace-like transform (which can be identified with the forward and inverse Mellin transform)
\begin{lemma}
    \begin{equation}
        \int_0^{\infty}x^{s-1} G_{p,q}^{m,n}\left(wx\bigg|\begin{array}{c}a_p\\
b_q\end{array}\right)dx = w^{-s} \frac{\Gamma(b_1+s)...\Gamma(b_m+s)\Gamma(1-a_1-s)...\Gamma(1-a_n-s)}{\Gamma(1-b_{m+1}-s)...\Gamma(1-b_q-s)\Gamma(a_{n+1}+s)...\Gamma(a_{p}+s)} .
    \end{equation}
\end{lemma}
\begin{proof}
    This is by definition of the Meijer-G function and the Mellin transform definition. That is,
\begin{subequations}
\begin{equation}
  \text{Forward transform :} F(s) = \left(\mathcal{M}\right)(s) = \int_0^{\infty} x^{s-1}f(x)dx 
\end{equation}    
\begin{equation}
 \text{Inverse transform :} f(x)=\left(\mathcal{M}^{-1}F\right)(x)=\int_{C}G_{p,q}^{m,n}\left(\begin{array}{c}
a_p\\
b_q
\end{array}\bigg| x\right)x^{-s}ds
\end{equation}
\begin{equation}
    G_{p,q}^{m,n}\left(\begin{array}{c}
a_p\\
b_q
\end{array}\bigg| x\right)=\frac{1}{2\pi i}\int_{C}\frac{\Gamma(b_1+s)...\Gamma(b_m+s)\Gamma(1-a_1-s)...\Gamma(1-a_n-s)}{\Gamma(1-b_{m+1}-s)...\Gamma(1-b_q-s)\Gamma(a_{n+1}+s)...\Gamma(a_{p}+s)}x^{-s}ds
\end{equation}
\end{subequations}
where the contour integration $C$ integrates from $c-i\infty$ to $c+i\infty$, and the $w^{-s}$ naturally arises as a scaling factor of the integration.

\end{proof}
\begin{proof}
    We return to the theorem. We start with the left hand side, 
\begin{equation}
\begin{split}
    &\int_{0}^{\infty}x^{\sigma -1} G_{p,q}^{m,n}\left(wx\bigg|\begin{array}{c}
a_p\\
b_q\end{array}\right) G_{\gamma,\delta}^{\alpha,\beta}\left(\eta x^{k/\rho}\bigg|\begin{array}{c}
c_{\gamma}\\
d_{\delta}\end{array}\right)dx\\&=\int_{0}^{\infty}x^{\sigma -1} G_{p,q}^{m,n}\left(wx\bigg|\begin{array}{c}
a_p\\
b_q\end{array}\right) \frac{1}{2\pi i}\int_{C}\frac{\prod_{j=1}^{\alpha}\Gamma(d_j+s)\prod_{j=1}^{\beta}\Gamma(1-c_j-s)}{\prod_{j=\beta+1}^{\gamma}\Gamma(c_j+s)\prod_{j=\alpha+1}^{\delta}\Gamma(1-d_j-s)}\big(\eta x^{\frac{k}{\rho}}\big)^{-s}dsdx\\
&=\frac{1}{2\pi i}\int_{C}\eta^{-s}\frac{\prod_{j=1}^{\alpha}\Gamma(d_j+s)\prod_{j=1}^{\beta}\Gamma(1-c_j-s)}{\prod_{j=\beta+1}^{\gamma}\Gamma(c_j+s)\prod_{j=\alpha+1}^{\delta}\Gamma(1-d_j-s)}\left(\int_0^{\infty}x^{\sigma -1} G_{p,q}^{m,n}\left(wx\bigg|\begin{array}{c}
a_p\\
b_q\end{array}\right)x^{-s\frac{k}{\rho}}dx\right) ds .
\end{split}
\end{equation}

Applying the previous lemma for the laplace like transform
\begin{equation}
    \int_0^{\infty}x^{\sigma -1} G_{p,q}^{m,n}\left(wx\bigg|\begin{array}{c}
a_p\\
b_q\end{array}\right)x^{-s\frac{k}{\rho}}dx=w^{-\sigma+s\frac{k}{\rho}} \frac{\prod_{j=1}^m\Gamma(b_j+s)\prod_{j=1}^n\Gamma(1-a_j-s)}{\prod_{j=m+1}^q\Gamma(1-b_j-s)\prod_{j=n+1}^p\Gamma(a_{j}+s)}\bigg|_{s=\sigma-\frac{ks}{\rho}} .
\end{equation}
Summarizing the LHS contains certain incompatible gamma function that needs further transformation. That is,
\begin{equation}
\resizebox{\hsize}{!}{$ LHS =\frac{1}{2\pi i}w^{-\sigma} \int_{C}\frac{\overbrace{\prod_{j=1}^{\alpha}\Gamma(d_j+s)}^{O}\overbrace{\prod_{j=1}^{\beta}\Gamma(1-c_j-s)}^{P}}{\underbrace{\prod_{j=\beta+1}^{\gamma}\Gamma(c_j+s)}_{Q}\underbrace{\prod_{j=\alpha+1}^{\delta}\Gamma(1-d_j-s)}_{R}}\frac{\overbrace{\prod_{j=1}^m\Gamma(b_j+\sigma-\frac{ks}{\rho})}^{S}\overbrace{\prod_{j=1}^n\Gamma(1-a_j-(\sigma-\frac{ks}{\rho}))}^{T}}{\underbrace{\prod_{j=m+1}^q\Gamma(1-b_j-(\sigma-\frac{ks}{\rho}))}_{U}\underbrace{\prod_{j=n+1}^p\Gamma(a_{j}+\sigma-\frac{ks}{\rho})}_{V}}\eta^{-s}w^{\frac{ks}{\rho}}ds$}
\label{fulllhs}
\end{equation}

On the right hand side, the equation is
\begin{equation}
\begin{split}
      RHS &= \frac{1}{2\pi i}\int_{C}\frac{A(s)B(s)}{C(s)D(s)}\big(\frac{\eta^{\rho}\rho^{\rho(\gamma-\delta)}}{w^kk^{k(p-q)}}\big)^{-s}ds= \bigg\{\begin{array}{c}
         \rho s=s' \\
         ds=\frac{1}{\rho}ds'
    \end{array}\bigg\}\\&= \frac{1}{2\pi i}\int_{C}\frac{A(\frac{s'}{\rho})B(\frac{s'}{\rho})}{C(\frac{s'}{\rho})D(\frac{s'}{\rho})}\eta^{-s'}\rho^{-(\gamma-\delta)s'}w^{\frac{ks'}{\rho}}k^{\frac{ks'(p-q)}{\rho}}\frac{1}{\rho}ds',
\end{split}
\label{Therhstocancel}
\end{equation}
where the variable exchange formula (scaling) invokes the utilization of Gauss multiplication formula, which is an effective transformation for the pole landscape. Next, we analyze the components ($A,B,C,D$) separately. Let's start with the $A$ part $(b_j+s)$:
\begin{subequations}
    \begin{equation}
A(s)=\prod_{j=1}^{\alpha}\prod_{t=0}^{\rho-1}\Gamma(\frac{d_j+t}{\rho}+s)
\prod_{j=1}^{n}\prod_{t=0}^{k-1}\Gamma(\frac{1-a_j-\sigma+t}{k}+s).
\end{equation}
\begin{equation}
\begin{split}
     A\Big(\frac{s'}{\rho}\Big)&=\prod_{j=1}^{\alpha}\prod_{t=0}^{\rho-1}\Gamma(\frac{d_j+t}{\rho}+\frac{s'}{\rho})
\prod_{j=1}^{n}\prod_{t=0}^{k-1}\Gamma(\frac{1-a_j-\sigma+t}{k}+\frac{s'}{\rho})\\
&=\Big(\prod_{j=1}^{\alpha}\Gamma(d_j+s')\Big) (2\pi)^{\frac{\alpha(\rho-1)}{2}}\rho^{\alpha(\frac{1}{2}-(d_j+s'))}\Big(\prod_{j=1}^{n}\Gamma(1-a_j-\sigma+\frac{ks'}{\rho})\Big)(2\pi)^{\frac{n(k-1)}{2}}k^{n(\frac{1}{2}-(1-a_j-\sigma+\frac{ks'}{\rho}))}\\
&=O\cdot T\cdot (2\pi)^{\frac{\alpha(\rho-1)}{2}}\rho^{\alpha(\frac{1}{2}-(d_j+s'))} \cdot (2\pi)^{\frac{n(k-1)}{2}}k^{n(\frac{1}{2}-(1-a_j-\sigma+\frac{ks'}{\rho}))},
\end{split}
\end{equation}
\end{subequations}
where in the second step we apply the multiplication formula. Similarly, the $B, C, D$ components can be computed. Next, we examine the $B$ part $(1-a_j-s)$
\begin{subequations}
    \begin{equation}
    \begin{split}
        B(s) &= \prod_{j=1}^{\beta}\prod_{t=0}^{\rho-1}\Gamma(1-\frac{c_j+t}{\rho}-s)
\prod_{j=1}^{m}\prod_{t=0}^{k-1}\Gamma(1-\frac{1-b_j-\sigma+t}{k}-s)\\ & = \prod_{j=1}^{\beta}\prod_{t=0}^{\rho-1}\Gamma(\frac{1-c_j+t}{\rho}-s)
\prod_{j=1}^{m}\prod_{t=0}^{k-1}\Gamma(\frac{b_j+\sigma+t}{k}-s)
    \end{split}
    \end{equation}
\begin{equation}
\begin{split}
     B\Big(\frac{s'}{\rho}\Big)&=\prod_{j=1}^{\beta}\prod_{t=0}^{\rho-1}\Gamma(\frac{1-c_j+t}{\rho}-\frac{s'}{\rho})
\prod_{j=1}^{m}\prod_{t=0}^{k-1}\Gamma(\frac{b_j+\sigma+t}{k}-\frac{s'}{\rho})\\
&=\Big(\prod_{j=1}^{\beta}\Gamma(1-c_j-s')\Big) (2\pi)^{\frac{\beta(\rho-1)}{2}}\rho^{\beta(\frac{1}{2}-(1-c_j-s'))}\Big(\prod_{j=1}^{m}\Gamma(b_j+\sigma-\frac{ks'}{\rho})\Big)(2\pi)^{\frac{m(k-1)}{2}}k^{m(\frac{1}{2}-(b_j+\sigma-\frac{ks'}{\rho}))}\\
&=P\cdot S\cdot (2\pi)^{\frac{\beta(\rho-1)}{2}}\rho^{\beta(\frac{1}{2}-(1-c_j-s'))} \cdot (2\pi)^{\frac{m(k-1)}{2}}k^{m(\frac{1}{2}-(b_j+\sigma-\frac{ks'}{\rho}))}
\end{split}
\end{equation}
\end{subequations}
The $C$ sector $(a_j+s)$ is
\begin{subequations}
    \begin{equation}
        C(s) =\prod_{j=m+1}^{q}\prod_{t=0}^{k-1}\Gamma(\frac{1-b_j-\sigma+t}{k}+s) \prod_{j=\beta+1}^{\gamma}\prod_{t=0}^{\rho-1}\Gamma(\frac{c_j+t}{\rho}+s)
   \end{equation}
\begin{equation}
\begin{split}
     C\Big(\frac{s'}{\rho}\Big)&=\prod_{j=m+1}^{q}\prod_{t=0}^{k-1}\Gamma(\frac{1-b_j-\sigma+t}{k}+\frac{s'}{\rho}) \prod_{j=\beta+1}^{\gamma}\prod_{t=0}^{\rho-1}\Gamma(\frac{c_j+t}{\rho}+\frac{s'}{\rho})\\
&=\Big(\prod_{j=m+1}^{q}\Gamma(1-b_j-\sigma+\frac{ks'}{\rho}) \Big) (2\pi)^{\frac{(q-m)(k-1)}{2}}k^{(q-m)(\frac{1}{2}-(1-b_j-\sigma+\frac{ks'}{\rho}))}\\&\Big(\prod_{j=\beta+1}^{\gamma}\Gamma(c_j+s')\Big)(2\pi)^{\frac{(\gamma-\beta)(\rho-1)}{2}}\rho^{(\gamma-\beta)(\frac{1}{2}-(c_j+s'))}\\
&=U\cdot Q\cdot (2\pi)^{\frac{(q-m)(k-1)}{2}}k^{(q-m)(\frac{1}{2}-(1-b_j-\sigma+\frac{ks'}{\rho}))}\cdot (2\pi)^{\frac{(\gamma-\beta)(\rho-1)}{2}}\rho^{(\gamma-\beta)(\frac{1}{2}-(c_j+s'))} .
\end{split}
\end{equation}
\end{subequations}

The $D$ sector $(1-b_j-s)$ is
\begin{subequations}
    \begin{equation}
        D(s) =\prod_{j=n+1}^{p}\prod_{t=0}^{k-1}\Gamma(1-\frac{1-a_j-\sigma+t}{k}-s) \prod_{j=\alpha+1}^{\delta}\prod_{t=0}^{\rho-1}\Gamma(1-\frac{d_j+t}{\rho}-s), 
   \end{equation}
\begin{equation}
\begin{split}
     D\Big(\frac{s'}{\rho}\Big )&=\prod_{j=n+1}^{p}\prod_{t=0}^{k-1}\Gamma(1-\frac{1-a_j-\sigma+t}{k}-\frac{s'}{\rho}) \prod_{j=\alpha+1}^{\delta}\prod_{t=0}^{\rho-1}\Gamma(1-\frac{d_j+t}{\rho}-\frac{s'}{\rho})\\
     &=\prod_{j=n+1}^{p}\prod_{t=0}^{k-1}\Gamma(\frac{a_j+\sigma+t}{k}-\frac{s'}{\rho}) \prod_{j=\alpha+1}^{\delta}\prod_{t=0}^{\rho-1}\Gamma(\frac{1-d_j+t}{\rho}-\frac{s'}{\rho})\\
&=\Big(\prod_{j=n+1}^{p}\Gamma(a_j+\sigma-\frac{ks'}{\rho}) \Big) (2\pi)^{\frac{(p-n)(k-1)}{2}}k^{(p-n)(\frac{1}{2}-(a_j+\sigma-\frac{ks'}{\rho}))}\\&\Big(\prod_{j=\alpha+1}^{\delta}\Gamma(1-d_j-s')\Big)(2\pi)^{\frac{(\delta-\alpha)(\rho-1)}{2}}\rho^{(\delta-\alpha)(\frac{1}{2}-(1-d_j-s'))}\\
&=V\cdot R\cdot (2\pi)^{\frac{(p-n)(k-1)}{2}}k^{(p-n)(\frac{1}{2}-(a_j+\sigma-\frac{ks'}{\rho}))}\cdot (2\pi)^{\frac{(\delta-\alpha)(\rho-1)}{2}}\rho^{(\delta-\alpha)(\frac{1}{2}-(1-d_j-s'))} .
\end{split}
\end{equation}
\end{subequations}
Notice that the $O,P,Q,R,S,T,U,V$ factors correspond to their counterparts in (Equation \ref{fulllhs}). We examine the $s'$ dependence in $\frac{A(\frac{s'}{\rho})B(\frac{s'}{\rho})}{C(\frac{s'}{\rho})D(\frac{s'}{\rho})}=C_2\cdot F(s')$, where $C_2$ is some scalar independent of $s'$. We single out the $s'$ terms
\begin{equation}
\begin{split}
      \frac{A(\frac{s'}{\rho})B(\frac{s'}{\rho})}{C(\frac{s'}{\rho})D(\frac{s'}{\rho})}&\sim\frac{\rho^{\alpha(\frac{1}{2}-(d_j+s'))}k^{n(\frac{1}{2}-(1-a_j-\sigma+\frac{ks'}{\rho}))}\rho^{\beta(\frac{1}{2}-(1-c_j-s'))}k^{m(\frac{1}{2}-(b_j+\sigma-\frac{ks'}{\rho}))}}{k^{(q-m)(\frac{1}{2}-(1-b_j-\sigma+\frac{ks'}{\rho}))}\rho^{(\gamma-\beta)(\frac{1}{2}-(c_j+s'))}k^{(p-n)(\frac{1}{2}-(a_j+\sigma-\frac{ks'}{\rho}))}\rho^{(\delta-\alpha)(\frac{1}{2}-(1-d_j-s'))}}\\&\sim\underbrace{\rho^{-\alpha s'+\beta s'-(\gamma-\beta)(-s')-(\delta-\alpha)s'}}_{\rho^{(\gamma-\delta)s'}}\underbrace{k^{-\frac{nks'}{\rho}+\frac{mks'}{\rho}-(q-m)(-\frac{ks'}{\rho})-(p-n)\frac{ks'}{\rho}}}_{k^{(q-p)\frac{ks'}{\rho}}}=F(s'),
\end{split}
\label{finalsdep}
\end{equation}
which exactly cancels out the $\rho,k$ in the exponential $s'$ term in Equation \ref{Therhstocancel}. Comparing equations (\ref{fulllhs}, \ref{Therhstocancel}, \ref{finalsdep}), we establish that $C_1=\frac{LHS}{RHS}$, the remaining scalar factor $C_1$ has three components:
\begin{itemize}
    \item The original $w^{-\sigma}$ term, and a $\rho$ term in the original transformation left out.
    \item The $2\pi$ terms
    \item The remaining terms in the $\rho,k$ exponentials.
\end{itemize}
The $2\pi$ terms are
\begin{equation}
\begin{split}
\frac{A(\frac{s'}{\rho})B(\frac{s'}{\rho})}{C(\frac{s'}{\rho})D(\frac{s'}{\rho})}\sim C_{12}&=\frac{(2\pi)^{\frac{\alpha(\rho-1)}{2}}(2\pi)^{\frac{n(k-1)}{2}}(2\pi)^{\frac{\beta(\rho-1)}{2}}(2\pi)^{\frac{m(k-1)}{2}}}{(2\pi)^{\frac{(q-m)(k-1)}{2}}(2\pi)^{\frac{(\gamma-\beta)(\rho-1)}{2}}(2\pi)^{\frac{(p-n)(k-1)}{2}}(2\pi)^{\frac{(\delta-\alpha)(\rho-1)}{2}}}\\&=(2\pi)^{\frac{\rho-1}{2}(2\alpha+2\beta-\gamma-\delta)}(2\pi)^{\frac{k-1}{2}(2m+2n-p-q)} .
\end{split}
\end{equation}

The remaining factor in the $\rho$ exponentials. Here implicitly, we are using the fact that $d_j=\frac{\sum_{j=1}^{\alpha}d_j}{\alpha}$ (similarly for $c_j$) for notation simplicity
\begin{equation}
    \frac{\rho^{\alpha(\frac{1}{2}-d_j)}\rho^{\beta(\frac{1}{2}-1+c_j)}}{\rho^{(\gamma-\beta)(\frac{1}{2}-c_j)}\rho^{(\delta-\alpha)(\frac{1}{2}-1+d_j)}}=\underbrace{\rho^{\frac{1}{2}\Big(\alpha-\beta-(\gamma-\beta)+(\delta-\alpha)\Big)}}_{\rho^{\frac{1}{2}(\delta-\gamma)}}\underbrace{\rho^{-\sum_{j=1}^{\alpha}d_j+\sum_{j=1}^{\beta}c_j-\sum_{j=\beta+1}^{\gamma}(-c_j)-\sum_{j=\alpha+1}^{\delta}d_j}}_{\rho^{\sum_{j=1}^{\gamma}c_j-\sum_{j=1}^{\delta}d_j}} .
\end{equation}

Similarly, the $k$ exponentials are
\begin{equation}
\begin{split}
    &\frac{k^{n(\frac{1}{2}-(1-a_j-\sigma))}k^{m(\frac{1}{2}-(b_j+\sigma))}}{k^{(q-m)(\frac{1}{2}-(1-b_j-\sigma))}k^{(p-n)(\frac{1}{2}-(a_j+\sigma))}}\\&=\underbrace{k^{\frac{1}{2}\Big(-n+\beta+(q-m)-(p-n)\Big)}}_{k^{\frac{1}{2}(q-p)}}\underbrace{k^{\sum_{j=1}^na_j-\sum_{j=1}^mb_j-\sum_{j=m+1}^qb_j+\sum_{j=n+1}^pa_j}}_{k^{\sum_{j=1}^pa_j-\sum_{j=1}^qb_j}}\underbrace{k^{n\sigma-m\sigma-(q-m)\sigma+(p-n)\sigma}}_{k^{\sigma(p-q)}} .
\end{split}
\end{equation}

Therefore, the remaining factor for $C_1$ is
\begin{equation}
    C_1=w^{-\sigma}\rho^{1+\frac{1}{2}(\delta-\gamma)+\sum_{j=1}^{\gamma}c_j-\sum_{j=1}^{\delta}d_j} k^{\sigma(p-q)+\frac{1}{2}(q-p)+\sum_{j=1}^pa_j-\sum_{j=1}^qb_j}(2\pi)^{\frac{\rho-1}{2}(2\alpha+2\beta-\gamma-\delta)}(2\pi)^{\frac{k-1}{2}(2m+2n-p-q)}.
\end{equation}
\end{proof}

\subsubsection{Selected calculations of the Laplace transform from Meijer-G functions \label{laplace_calmeijerg}}
\begin{theorem}
It follows from the previous theorem \ref{theoremgeneral}
\begin{equation}
    \int_0^{\infty} e^{-s\cdot x}G_{p,q}^{m,n}\left(\begin{array}{c}
a_p\\
b_q
\end{array}\bigg| wx^2\right)dx = \frac{1}{\sqrt{\pi}s}G_{p+2,q}^{m,n+2}\left(\begin{array}{c}
0,\frac{1}{2},a_p\\
b_q
\end{array}\bigg| \frac{4w}{s^2}\right) .
\label{m-laplace}
\end{equation}
\end{theorem}
The following theorem concerns the pole and residue geometry and its information in reconstructing the original function.
\begin{theorem}
\begin{equation}
    \cos x = \sqrt{\pi}G_{0,2}^{1,0}\left(\begin{array}{c}
;\\
0;\frac{1}{2}
\end{array}\bigg| \frac{x^2}{4}\right) .
\end{equation}
\end{theorem}

\begin{proof}
  By definition, we can write out the right hand side as
\begin{equation}
\begin{split}
RHS &= \sqrt{\pi} \Big(\frac{1}{2\pi i}\Big)\int_{C}\frac{\prod_{j=1}^m\Gamma(b_j-s)\prod_{j=1}^n\Gamma(1-a_j+s)}{\prod_{j=m+1}^q\Gamma(1-b_{j}+s)\prod_{j=n+1}^q\Gamma(a_j-s)}z^sds\bigg|_{z=\frac{x^2}{4}}\\
&=\sqrt{\pi}\Big(\frac{1}{2\pi i}\Big)\int_C \frac{\Gamma(-s)}{\Gamma(1-\frac{1}{2}+s)}z^sds\bigg|_{z=\frac{x^2}{4}}=\sqrt{\pi}\Big(\frac{1}{2\pi i}\Big) \underbrace{2\pi i\sum_{k=0}^{\infty}{\text{res}}_{k}}_{\text{residue theorem}} .
\end{split}
\end{equation}
The gamma function $\Gamma(z)$ has poles at $0,-1,-2,\cdots $, and $\frac{1}{\Gamma(z)}$ is entire, since the gamma function has no zeros.
Using the fact that only the poles from the numerator result in residues that contributes to the integration, and $\text{res}_{z=n}\Gamma(-z)=\frac{(-1)^{n}}{n!}$ (the rest are scalar factors), we have
\begin{equation}
\begin{split}
RHS &= \sqrt{\pi}\Big(\frac{1}{2\pi i}\Big)2\pi i\sum_{k=0}^{\infty}\frac{(-1)^{k}}{k!}\frac{1}{\Gamma(\frac{1}{2}+k)}(\frac{x^2}{4})^k=\sum_{k=0}^{\infty}\frac{(-1)^{k}}{(2k)!}x^{2k}=\cos x ,   
\end{split}
\end{equation}
where we used the fact that $\frac{k!\Gamma(\frac{1}{2}+k)4^k}{\sqrt{\pi}}=(2k)!$, since $k!2^k=2\cdot 4\cdot 6\cdot\cdot\cdot 2k$ and $\frac{\Gamma(\frac{1}{2}+k)2^k}{\sqrt{\pi}}=1\cdot 3\cdot5\cdot\cdot\cdot (2k-1)$ 
\end{proof}
\begin{corollary}
Using the general form of the Meijer-G function being closed under Laplace transform (Equ \ref{m-laplace}), we have
\begin{equation}
 \mathcal{L}(\cos(wx))(s)=\frac{s}{w^2+s^2} .
\end{equation}
\end{corollary}

\begin{proof}
\begin{equation}
\begin{split}
    \mathcal{L}(\cos(wx))(s) &= \frac{1}{\sqrt{\pi}s}\sqrt{\pi}G_{2,2}^{1,2}\left(\begin{array}{c}
0,\frac{1}{2}-\\
0,\frac{1}{2}
\end{array}\bigg| \frac{4w'}{s^2}\right)\bigg|_{w'=\frac{w^2}{4}}\\
&=\frac{1}{s}\left(\frac{1}{1+x}\right)\bigg|_{x=\frac{w^2}{s^2}}=\frac{1}{s}\cdot \frac{1}{1+\frac{w^2}{s^2}}=\frac{s}{w^2+s^2} .
\end{split}
\end{equation}
\end{proof}

\begin{lemma}
Using the following representation of the  hypergeometric function, we get the Laplace transform for the Bessel function of the first kind $J_n(at)$ as a corollary:
    \begin{equation}
        {}_2^1F(\frac{1+n}{2},\frac{2+n}{2},1+n,x)=\frac{2^n(1-\sqrt{1-x})^n}{\sqrt{1-x}x^n}, \forall n\in \mathbb{Z} .
    \end{equation}
\end{lemma}

\begin{corollary}
    \begin{equation}
 \mathcal{L}(J_n(at))(s)=\frac{(-s+\sqrt{a^2+s^2})^n}{a^n\sqrt{s^2+a^2}} .
\end{equation}
\end{corollary}
\begin{proof}
    Applying Equation (\ref{m-laplace}), we need 
    $G_{1,2}^{2,2}\left(\begin{array}{c}
0,\frac{1}{2}-\\
\frac{n}{2},-\frac{n}{2}
\end{array}\bigg|\frac{a^2t^2}{s^2}\right)$
    referring to mathematica or we may refer to the special case for Meijer-G tabularized at wolfram website,
\begin{equation}
\begin{split}
LHS&=\frac{1}{\sqrt{\pi}s}\cdot 2^{-n}\sqrt{\pi}(\frac{a^2}{s^2})^{n/2}\cdot {}_2^1F\Big(\frac{1+n}{2},\frac{2+n}{2},1+n,-\frac{a^2}{s^2}\Big)\\&=2^{-n}s^{-1-n}a^n\frac{2^n(1-\sqrt{1-x})^n}{\sqrt{1-x}x^n}\bigg|_{x=-\frac{a^2}{s^2}}\\&=s^{n-1}\frac{1}{a^n}\frac{(-1+\sqrt{1+\frac{a^2}{s^2}})^n}{\sqrt{1+\frac{a^2}{s^2}}}=\frac{(-s+\sqrt{a^2+s^2})^n}{a^n\sqrt{s^2+a^2}}.
\end{split}
\end{equation}
\end{proof}

\subsection{Appendix E. Theoretical Approximations and Analysis of ILT}
\subsubsection{Concrete strategy that guarantees isotropicity \label{isotropicb}}
From the limitation theorem, to guarantee isotropicity, we couple the domain partitioning strategy to the phase sampling strategy.  
Denote $B_{ij}'=\exp(-z_ip_{i,j})$. If we relax the independence of phase and the sampling, we can get achievable strategies that satisfies $\mathbb{E}[b_{ij}']=0$, $\mathbb{E}[b_{ij}'b_{ij}'^*]=1$. Consider the following von Mises distribution 
\begin{equation}
    f_{\Phi}(\phi) = \frac{\exp(-a\cos(\phi))}{2\pi I_0(a)}
\end{equation}
and the following two integrals
\begin{equation}
 \mathbb{E}[c_{ij}]=   \int_0^{2\pi} \exp(ae^{i\phi})f_{\Phi}(\phi)d\phi = \frac{2\pi J_0(a)}{2\pi I_0(a)} =\frac{J_0(a)}{I_0(a)}
\end{equation}
\begin{equation}
    \mathbb{E}[c_{ij}c_{ij}^*]  =\int_0^{2\pi} \exp(ae^{i\phi})\exp(ae^{i\phi})^*f_{\Phi}(\phi)d\phi = \int_0^{2\pi} \exp(2a\cos\phi)\frac{\exp(-a\cos(\phi))}{2\pi I_0(a)}d\phi = 1 .
\end{equation}

Therefore, if we focus on the zeros of the Bessel functions of the $0^{th}$-kind, $J_0(a)$, we can guarantee isotropicity. The numerical algorithm roughly goes the following:
\begin{itemize}
    \item Use $r$ (radial component) as a scaling parameter so that we have enough zeros from [0,$2\pi$] to fit in the zeros such that $J_{0}(r\cdot p_{i,j})=0$ deterministically for $p_{i,j}$.
    \item Then at each entry sample the phase according to the distribution $ f_{\Phi}(\phi) = \frac{\exp(-r\cdot p_{i,j}\cos(\phi))}{2\pi I_0(r\cdot p_{i,j})}$, then the vectors are isotropic, hence singular values of $C_{ij} = \exp(-z_ip_{i,j}), j=1,2,3,...n+1$ are bounded.
\end{itemize}
\begin{lemma}
    The smallest singular value for the first order difference matrix $\sigma_n(D_{ij})$
    \begin{equation}
        D_{ij} = \begin{bmatrix}
1 & 0 & 0 & \cdots & 0 \\
-1 & 1 & 0 & \cdots & 0 \\
0 & -1 & 1 & \cdots & 0 \\
0 & 0 & -1 & \cdots & 0 \\
\vdots  & \vdots  & \vdots  & \ddots & \vdots  \\
0 & 0 & 0 & \cdots & 1 \\
0 & 0 & 0 & \cdots & -1
\end{bmatrix}
\end{equation}
is bounded below \cite{109946}.
\end{lemma}
\begin{proof}
The SVD $D_{n\times(n-1)}=U\Sigma V^T$ is given by 
\begin{subequations}
    \begin{equation}
    U_{n\times (n-1)} = \sqrt{\frac{2}{n}}\begin{bmatrix}
\cos(\frac{\pi}{2n}) & \cos(\frac{2\pi}{2n}) & \cdots & \cos(\frac{(n-1)\pi}{2n}) \\
\cos(\frac{3\pi}{2n})  & \cos(\frac{3\cdot 2\pi}{2n}) & \cdots & \cos(\frac{3\cdot (n-1)\pi}{2n}) \\
\vdots  & \vdots  & \ddots & \vdots  \\
\cos(\frac{(2n-1)\cdot \pi}{2n})  & \cos(\frac{(2n-1)\cdot 2\pi}{2n}) & \cdots & \cos(\frac{(2n-1)\cdot (n-1)\pi}{2n})
\end{bmatrix}
\end{equation}
\begin{equation}
    \Sigma_{(n-1)\times(n-1)}=\begin{bmatrix}
2\sin(\frac{\pi}{2n}) & 0 & \cdots & 0 \\
0 & 2\sin(\frac{2\pi}{2n}) & \cdots & 0 \\

\vdots  & \vdots    & \ddots & \vdots  \\
0 & 0 & \cdots & 2\sin(\frac{(n-1)\pi}{2n})
\end{bmatrix}
\end{equation}
\begin{equation}
  V^T_{(n-1)\times (n-1)}= 
\sqrt{\frac{2}{n}}\begin{bmatrix}
\sin(\frac{\pi}{n}) & \sin(\frac{2\pi}{n}) & \cdots & \sin(\frac{(n-1)\pi}{n}) \\
\sin(\frac{2\pi}{n}) & \sin(\frac{2\cdot 2\pi}{n}) & \cdots & \sin(\frac{(n-1)\cdot 2\pi}{n}) \\
\vdots  & \vdots  & \ddots & \vdots  \\
\sin(\frac{(n-1)\pi}{n}) & \sin(\frac{2\cdot 2\pi}{n}) & \cdots & \sin(\frac{(n-1)\cdot (n-1)\pi}{n}) \\
\end{bmatrix}
\end{equation}
\end{subequations}

    The \textbf{unity constraints} $UU^T=I_{n}$ can be checked from 
    \begin{equation}
        \cos^2 \left(\frac{\pi}{2n}m\right)+\cos^2\left(\frac{3\pi}{2n}m\right)+\cdot\cdot\cdot +\cos^2 \left(\frac{(2n-1)\pi}{2n}m\right)= \sum_{j=1}^n \cos^2 \left(\frac{(2j-1)\pi}{2n}m\right)=\frac{n}{2}, \text{m is odd}
    \end{equation}
Applying half angle formula, this equality is equivalent to $\sum_{j=1}^n \cos \left(\frac{(2j-1)\pi}{n}m\right)=0$ which can be checked by considering the real parts of the roots $(x^{1/m})^n+1=0$
and the $x^{(n-1)/m}$ term in the root expansion. Similarly, $VV^T=I_{n-1}$ can be checked
\begin{equation}
    \sin^2 \left(\frac{\pi}{n}t\right)+\sin^2\left(\frac{2\pi}{n}t\right)+\cdot\cdot\cdot +\sin^2 \left(\frac{(n-1)\pi}{n}t\right)= \sum_{j=1}^{n-1} \sin^2 \left(\frac{j\pi}{2n}t\right)=\frac{n}{2}, 
\end{equation}
This equality is equivalent to $\sum_{j=1}^{n-1} \cos \left(\frac{2j\pi}{n}t\right)=\frac{1}{2}$, which can be checked by considering the real parts of the roots $(x^{1/t})^n-1=0$
and the $x^{(n-1)/t}$ term in the root expansion.
To \textbf{verify} that $D=U\Sigma V^T$,we note
\begin{equation}
    D_{mt}= \frac{4}{n}\sum_{k=1}^{n-1}\sin\Big(\frac{k\pi}{2n}\Big)\sin\Big(\frac{t\pi}{n}k\Big)\cos\Big(\frac{(2m-1)\pi}{2n}k\Big) =  \frac{2}{n}\sum_{k=1}^{n-1}\sin\Big(\frac{t\pi}{n}k\Big)  \Big(\sin\Big(\frac{mk\pi}{n}\Big)+\sin\Big(\frac{(1-m)k\pi}{n}\Big)\Big) 
\end{equation}
where we contract using $\sin\Big(\frac{k\pi}{2n}\Big)\cos\Big(\frac{(2m-1)\pi}{2n}k\Big) = \frac{1}{2}\Big(\sin(\frac{mk\pi}{n})+\sin(\frac{(1-m)k\pi}{n})\Big)$ we first check that $D_{mm}=1, \forall m = 1,2,\cdot\cdot\cdot, n-1$. That is,
\begin{equation}
    D_{mm}=\frac{2}{n}\Big(\underbrace{\sum_{k=1}^{n-1}\sin^2\Big(\frac{m\pi}{n}k\Big)}_{n/2}+\underbrace{\sum_{k=1}^{n-1}  \sin\Big(\frac{mk\pi}{n}\Big)\sin\Big(\frac{(1-m)k\pi}{n}\Big)}_{0}\Big)=1
\end{equation}
where the second term can be seen by reversing the multiplication to addition. we next check that $D_{(m+1)m}=-1, \forall m = 1,2,\cdot\cdot\cdot, n-1$.
\begin{equation}
    D_{mm}=\frac{2}{n}\Big(\underbrace{\sum_{k=1}^{n-1}\sin\Big(\frac{m\pi}{n}k\Big)\sin\Big(\frac{(m+1)\pi}{n}k\Big)}_{0}+\underbrace{\sum_{k=1}^{n-1}  \sin\Big(\frac{mk\pi}{n}\Big)\sin\Big(\frac{(-m)k\pi}{n}\Big)}_{-2/n}\Big)=-1
\end{equation}
With a similar procedure, we can check that all other cases vanishes $D_{mt}=0$, if $t\neq m $ and $t\neq m-1 $ 
\begin{equation}
    D_{mt}=\frac{2}{n}\Big(\underbrace{\sum_{k=1}^{n-1}\sin(\frac{t\pi}{n}k)\sin\Big(\frac{m\pi}{n}k\Big)}_{0}+\underbrace{\sum_{k=1}^{n-1}  \sin\Big(\frac{tk\pi}{n}\Big)\sin\Big(\frac{(1-m)k\pi}{n}\Big)}_{0}\Big)=0
\end{equation}
\end{proof}
Thus, $\sigma_n(D_{ij})=2\sin(\frac{2\pi}{n})$.
\begin{theorem}
    Let $B=CD$, since $\sigma_n(D)$ is bounded, then the smallest singular value for $B=\Big(\exp(-z_ip_{i,j})-\exp(-z_ip_{i,j-1})\Big)$ is bounded.
\end{theorem}
\begin{proof}
\begin{equation}
    \sigma_n(B)=\sigma_n(CD) = \min_{x\in \mathbb{C}^n}\frac{\|CDx\|}{\|x\|} = \min_{x\in \mathbb{C}^n}\frac{\|CDx\|}{\|Dx\|}\frac{\|Dx\|}{\|x\|}\geq \min_{x\in \mathbb{C}^n}\frac{\|Cx\|}{\|x\|}\min_{x\in \mathbb{C}^n}\frac{\|Dx\|}{\|x\|}=  \sigma_n(C)\cdot \sigma_n(D) .
\end{equation}
\end{proof}

\subsubsection{Isotropicity limitation \label{limitationthm}}
\underline{\textbf{The uniform phase is the only way to erase domain variable independence for $B_{ij}$}}
\begin{lemma}
If $\phi$ and $a$ is independent, then the only $f_{\Phi}(\phi)\in L^2([0,2\pi])$ such that  
\begin{equation}
    \int_0^{2\pi} \exp(-a\exp(i\phi))f_{\Phi}(\phi)d\phi = \text{const},
\end{equation}  
is $f_{\Phi}(\phi)=\frac{1}{2\pi}.$
\end{lemma}
\begin{proof}
    We use Feynmann's integral trick and study $I(a) =  \int_0^{2\pi} \exp(-a\exp(i\phi))f_{\Phi}(\phi)d\phi$. Since $f_{\Phi}(\phi)$ is a probability distribution, then $I(a)=1$. By assumption, $\Big(\frac{d^j}{(da)^j}I(a)\Big)\Big|_{a=0}=0, \forall j=1,2,3,4,\cdots $. We expand $\exp(-a\exp(i\phi))$ around zero. Then,
    \begin{equation}
        \exp(-a\exp(i\phi))= 1- a\exp(i\phi)+\frac{a^2\exp(2i\phi)}{2!}+\cdots \ .
    \end{equation}
    After Taylor expansion, the vanishing derivative at any order translates to hierarchical Fourier basis constraints on $f_{\Phi}(\phi)$. Meaning, 
    \begin{equation}
        \int_0^{2\pi} \exp(i\cdot 0\cdot \phi)f_{\Phi}(\phi)d\phi=1
        \label{const-int}
    \end{equation}
    \begin{equation}
        \int_0^{2\pi}\exp(i\cdot j\cdot \phi)f_{\Phi}(\phi)d\phi=0, j=1,2,3,4,...
        \label{re-im}
    \end{equation}
Since $\left \{\sqrt{\frac{1}{2\pi}},\frac{1}{\sqrt{\pi}}\cos(n\phi),\frac{1}{\sqrt{\pi}}\sin(n\phi)\right \}$ forms the Fourier-Euler basis in $L^2([0,2\pi])$, and the basis coefficients for these basis correspond precisely to Equation \ref{const-int}, Equation \ref{re-im}'s real part, and Equation \ref{re-im}'s imaginary part. 
Thus, $f_{\Phi}(\phi)=\frac{1}{2\pi}.$
\end{proof}

Next, we construct sampling strategy so that the matrix $C_{ij} = \exp(-z_ip_{i,j})-1, j=1,2,3,...n+1$. To bound singular values, we need $C_{ij}$ to have isotropic rows. By isotropic vector row, we mean a random vector $X$ in $\mathbb{C}^n$ with $\Sigma(X)=\mathbb{E}[XX^*]=I_{n\times n}$. We need 
\begin{itemize}
    \item For fixed row $i$, $\mathbb{E}[c_{ij}c_{ij}^*]\sim \text{constant} \forall j$. 
    \item For fixed row $i$, $\mathbb{E}[c_{ij}c_{it}^*]=0, \forall j \neq t$, In fact, if we gurantee independence among the row entries, then $\mathbb{E}[c_{ij}c_{it}^*] = \mathbb{E}[c_{ij}]\mathbb{E}[c_{it}^*]$ and we show a stronger condition that they have mean 0,  i.e., $\mathbb{E}[c_{ij}]=0$.
\end{itemize}
\begin{lemma}
Using the previous lemma, if we have \textbf{uniform phase distribution}, i.e., $f_{\Phi}(\phi)=\frac{1}{2\pi}, 0\leq \phi \leq 2\pi$. Then $\mathbb{E}[c_{ij}]=0$ given the domain partitioning $p_{i,j}$ and phase sampling are independent.
\end{lemma}
\begin{proof}
We write out the expectation using our previous lemma
\begin{equation}
    \mathbb{E}[c_{ij}]=\mathbb{E}_{P}[\mathbb{E}_{Z}[c_{ij}\mid P]]=\mathbb{E}_{P}[\mathbb{E}_{Z}[\exp(-z_ip_{i,j})-1]]=\mathbb{E}_{P}[\int_0^{2\pi} \exp(-re^{i\phi}p_{i,j})f_{\phi}(\phi)d\phi]-1=1-1=0 .
\end{equation}
Although this mean $0$ expectation holds for any circular symmetric distribution, but to kill the excess degrees of freedom. Here the sampling is physical since the phase is uniform, therefore we are sampling from cocentric circles foliated by radius. Each foliation (fixed radius closed curve) is sufficient to reconstruct the information in the analytic function. However, we now argue the partition strategy so that $\mathbb{E}[c_{ij}c_{ij}^*]\sim \text{constant} \forall j$ fixing $i$ does not exist.
\begin{equation}
\begin{split}
    c_{ij}c_{ij}^* &=\Bigg(\exp(re^{i\phi}p_{i,j})-1\Bigg)\Bigg(\exp(re^{i\phi}p_{i,j})-1\Bigg)^* \\
    &= [\exp(r\cos\phi p_{i,j})\cos (r\sin\phi p_{i,j})-1]^2+\exp(2r\cos\phi p_{i,j})\sin (r\sin\phi p_{i,j})^2\\
    &=1+\exp(2r\cos \phi p_{i,j})-2\exp(r\cos \phi p_{i,j})\cos(r\sin\phi p_{i,j}) .
\end{split}
\end{equation}
The second term gives the normalizing factor for the von Mises distribution, and the third term integrates to $4\pi$. 

\begin{equation}
\begin{split}
    \mathbb{E}_{\Phi}[c_{ij}c_{ij}^*] &=\int_0^{2\pi}\Big(1+\exp(2\cos \phi p_{i,j})-2\exp(\cos \phi p_{i,j})\cos(\sin\phi p_{i,j})\Big)\frac{1}{2\pi}d\phi  \\
    &=\frac{}{}I_0(2p_{i,j}\cdot r)- 1
\end{split}
\end{equation}

where $I_0$ is the modified Bessel function of the $0^{th}$ kind. We want

\begin{equation}
    \int_0^{\infty} ( I_{0}(2p_{i,j}\cdot r)-1) f(r)dr = \text{const}\neq 0 .
\end{equation}

If we study the asymptotic here 
\begin{equation}
 I_{0}(2p_{i,j}\cdot r)-1 = p_{i,j}^2r^2+O(r^4), \int_0^{\infty} ( I_{0}(2p_{i,j}\cdot r)-1) f(r)dr \propto p_{i,j}^2 \neq \text{constant} .
\end{equation}

Also since $f(r)\geq 0$. The integration $g(p_{i,j}) = \int_0^{\infty} ( I_{0}(2p_{i,j}\cdot r)-1) f(r)dr$ increases with $p_{i,j}$. Therefore, $\mathbb{E}[c_{ij}c_{ij}^*]=g(p_{i,j})$ cannot be constant. Thus, no practical numerical strategy can guarantee isotropicity in this case.
\end{proof}

\begin{lemma}
If $a$ and $\phi$ are independent, then there exists no $f_{\Phi}(\phi)$ such that
\begin{equation}
    \int_0^{2\pi} \exp(a\exp(i\phi))\exp(-i\phi)f_{\Phi}(\phi)d\phi = \text{const}
\end{equation}
where the constant is independent of $a$.
\end{lemma}
\begin{proof}
Similarly, let $I(a) =  \int_0^{2\pi} \exp(a\exp(i\phi)-i\phi)f_{\Phi}(\phi)d\phi$. We observe $I'(0)=1$. Therefore, it is impossible for any function that is independent on $a$ whose derivative at zero is $1$. Thus, $f_{\Phi}(\phi)$ does not exist.
\end{proof}

\begin{itemize}
\item Interestingly, $I'(a) =  \int_0^{2\pi} \exp(a\exp(i\phi))f_{\Phi}(\phi)d\phi$, which is the previous case, and if we let $f_{\Phi}(\phi)=\frac{1}{2\pi}$. Then, $I'(a)=1, \ \forall a$. Consider $I(0)$, which is $I(0)=0$. Then combing with the first order condition we have $I(a)=a$.
\item Thus, as a by product, we have
\begin{equation}
    \int_0^{2\pi} \exp(ae^{i\phi}) e^{-i\phi}d\phi=2\pi a, \forall a\neq 0.
\end{equation}
\end{itemize}

\subsection{Appendix F: Random partitioning}
\label{randompartit}
\subsubsection{Uniform distribution case}

If $X_i\sim \text{Unif}[0,1]$, then $L_n=\sum_{i=1}^nX_i\sim Irwin-Hall(n)$, the partitioning result distrbution for $T=\frac{L_k}{L_{n+1}}=\frac{\sum_{j=1}^{k}X_j}{\sum_{j=1}^{n+1} X_{j}}$ can be obtained by
\begin{equation}
\begin{split}
    P(T\leq t) &= P\left(L_k\leq t(L_k+\sum_{i=k+1}^{n+1}X_i)\right)=\int_{0}^k P\left(L_k\leq t(L_k+\sum_{i=k+1}^{n+1}X_i)\mid L_k=l\right)P(L_{k}=l)dl\\
    &=\int_{0}^k P\left(\sum_{i=k+1}^{n+1}X_i \geq \frac{l-tl}{t}\right)P(L_{k}=l)dl .
\end{split}
\label{unifresint}
\end{equation}
\begin{figure}[H]
    \centering
    \includegraphics[width=13cm,height=5cm]{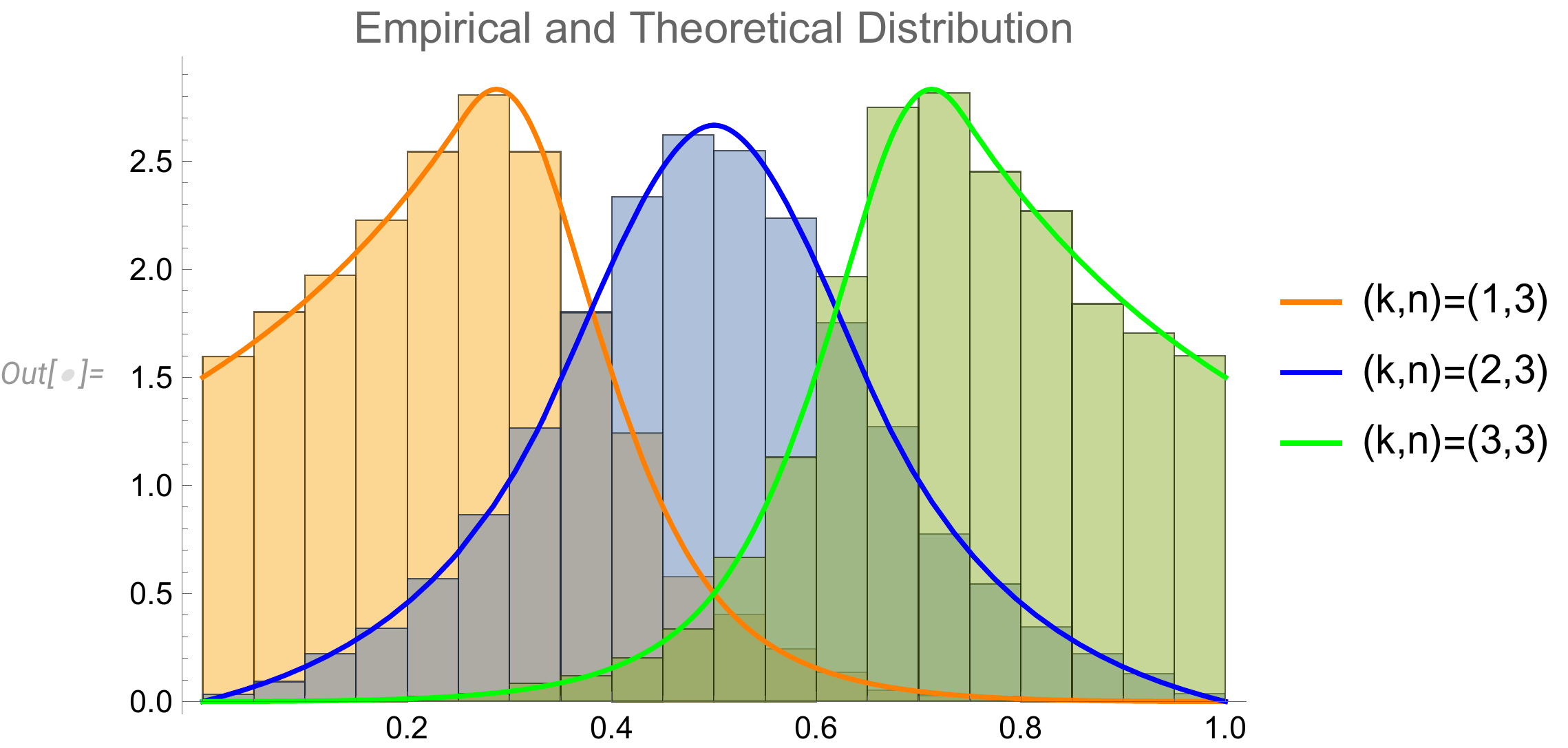}
    \caption{An example for $n=3$ for the uniform sum, with the empirical distribution (histogram) generated from uniform sampling and the theoretical distribution (curve) from \ref{unifresint}}
    \label{unif_rev_plot}
\end{figure}

\subsubsection{Exponential distribution with order statistics from uniform sampling}
The exponential case has a tight connection with the order statistics from uniform sampling with the following result
\begin{lemma}
    Let $X_1,X_2,\cdots ,X_{n+1}\sim \exp(\lambda)$. The joint distribution of $X_1,(X_1+X_2),\cdots,(X_1+\cdots+X_{n+1})$ is a uniform ordered distribution. That is, they are equivalent to the order statistics of a uniform distribution by iid sampling from the uniform (\text{Unif}[$0,s$]) and then perform the ordering $0<X_{(1)}< \cdots <X_{(n)}<s$.
\end{lemma}
\begin{proof}
Let $S_i=X_1+X_2+\cdots +X_i$ naturally $S_i$ are ordered statistics with $X_i>0$. 
    We can calculate the joint distribution from the exponential distribution
    \begin{equation}
    \begin{split}
        f_{S_{1},\cdots ,S_{n+1}}(s_1,\cdots ,s_{n+1})&=
        f_{X_1}(s_1)f_{X_2}(s_2-s_1)\cdots f_{X_{n+1}}(s_{n+1}-s_{n})\\&=\lambda^{n+1} e^{-\lambda s_{n+1}} \mathbbm{1}\{0<s_1<\cdots <s_n<s_{n+1}\}.
    \end{split}
    \label{tot-dist}
    \end{equation}
The marginal on $S_n$ can be obtained from the well known fact that $X_i\sim \text{exp}(\lambda)\Longrightarrow \sum_i X_i\sim \Gamma(n,\lambda)$:
\begin{equation}
    f_{S_n}(s)=\frac{s^{n-1}e^{-\lambda s}\lambda^n}{\Gamma(n)}\mathbbm{1}\{s>0\}.
    \label{marginals}
\end{equation}
Combining \ref{tot-dist},\ref{marginals} we fixate the length at $s$,
\begin{equation}
     f_{S_{1},\cdots ,S_{n}\mid S_{n+1}=s}(s_1,\cdots , s_n)=\frac{f_{S_{1},\cdots ,S_{n+1}}(s_1,\cdots ,s_{n+1})}{f_{S_{n+1}}(s)}=\frac{\Gamma(n+1)}{s^{n}}
     \mathbbm{1}\{0<s_1<\cdots <s_{n}<s\}.
     \label{fulljoint}
\end{equation}

The $r^{th}$ order statistic for the random variable arrangement partition $0<X_{(1)}< \cdots  <X_{(n)}<s$ with IID uniform sampling ($Unif[0,s]$) is given by
\begin{equation}
    f_{X_{(r)}}(x)=\left (\frac{1}{s}\right )^{n} \frac{\Gamma(n+1)}{\Gamma(r)\Gamma(n+1-r)}x^{r-1}(s-x)^{n-r}.
\end{equation}
The conditional distribution $y>x$, is effectively reducing the problem to $x<X_{(r+1)}< \cdots <X_{(n)}<s$
\begin{equation}
\begin{split}
     f_{X_{(r+1)}\mid X_{(r)}}(y\mid x) &= \left (\frac{1}{s-x}\right )^{n-r} \frac{\Gamma(n+1-r)}{\Gamma(1)\Gamma(n-r)}y^{1-1}(s-y)^{n-r-1}\\&= \left (\frac{1}{s-x}\right )^{n-r}(s-y)^{n-r-1}\  (n-r) .
\end{split}
\end{equation}
We can iterate this process using the Markovian nature, and obtain the joint as in \ref{fulljoint} using $f_{S_{1},\cdots ,S_{n-1}\mid S_n=s}(s_1,\cdots , s_{n-1})=f_{X_{(1)}}(x) f_{X_{(2)}\mid X_{(2)}}(x_2\mid x_1)\cdots f_{X_{(n-1)}\mid X_{(n-2)}}(x_{n-1}\mid x_{n-2})=\frac{\Gamma(n+1)}{s^{n}}
     \mathbbm{1}\{0<x_1<\cdots <x_{n}<s\}$.
\end{proof}
\subsubsection{Contrasting random partitioning from uniform, exponential, and equidistant schemes}
\label{contrastcase}
To contrast the deviation of the theoretical partition with the empirical random partitioning, we illustrate this via the following graphic. We use $n$ here to denote the number of datapoints between the interval $(0,2\pi)$ to avoid ambiguity. We use E1 to denote the partition strategy: $p_j\in [\frac{2\pi}{n+1}j-\frac{\pi}{n+1},\frac{2\pi}{n+1}j+\frac{\pi}{n+1}], \forall j\in\{1,2,\cdots ,n\}$ and E2 to denote the partition strategy: $p_j\sim Unif (\frac{2\pi j}{n}-\frac{2\pi}{n},\frac{2\pi j}{n})$.
\begin{figure}[H]
    \centering
    \includegraphics[width=16cm,height=4cm]{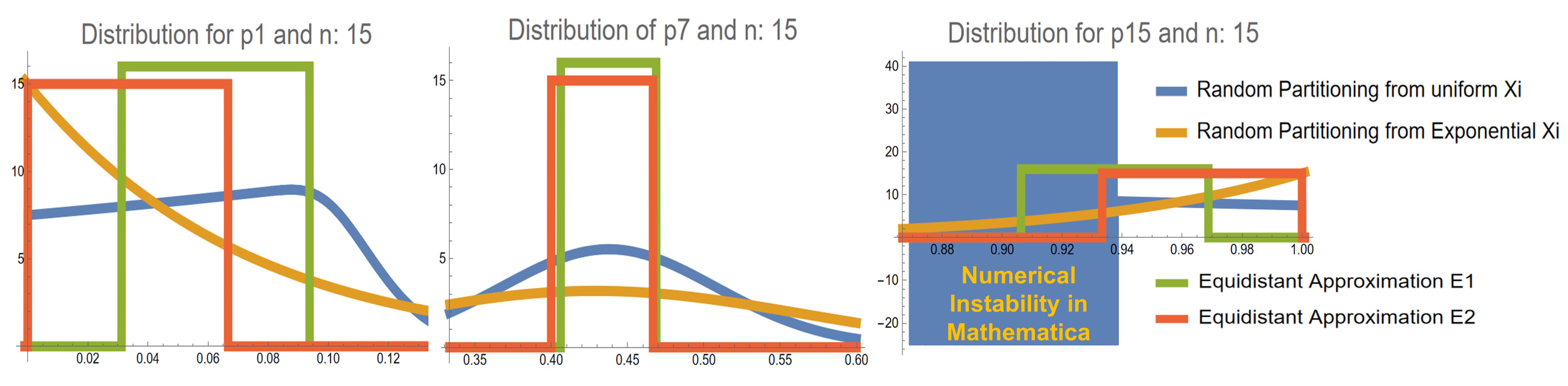}
    \caption{A side by side comparison of the random interval partitioning scheme as well as the equidistant approximation to kill the dependence and obtain theoretical bounds for the smallest singular values from random matrix theory, the blurriness (instability that gives rise to huge fluctuations in the rightmost figure) is attributed to the numerical instability for calculating large scale integration from calculating the density of ratios of Irwin-Hall distribution.}
    \label{Irwin-hall-plot}
\end{figure}

\end{document}